\documentclass[12pt,reqno]{amsart}
\usepackage[T1]{fontenc}
\usepackage[width=6.5in, left=1in, right=1in, height=8.8in, top=1.1in,bottom=1.1in]{geometry}

\usepackage{amsmath,amsfonts,amsthm,amssymb,color}
\usepackage{pdfsync}
\usepackage{hyperref}
\newcommand{\der}{\delta}
\newcommand{\doo}{\delta_{1}}
\newcommand{\dt}{\delta_{2}}

\newcommand{\di}{\diamond}
\newcommand{\dom}{\text{Dom}}
\newcommand{\hatotimes}{{\hat {\otimes}}}
\newcommand{\hone}{\hat{1}}
\newcommand{\htwo}{\hat{2}}
\newcommand{\id}{{\rm Id}}

\newcommand{\ott}{[0,1]}
\newcommand{\ou}{[0,1]}

\newcommand{\cb}{{\mathcal B}}
\newcommand{\cac}{{\mathcal C}}

\newcommand{\cf}{{\mathcal F}}

\newcommand{\ch}{{\mathcal H}}

\newcommand{\cm}{{\mathcal M}}
\newcommand{\cn}{{\mathcal N}}

\newcommand{\cp}{{\mathcal P}}

\newcommand{\cs}{{\mathcal S}}

\newcommand{\cz}{{\mathcal Z}}

\newcommand{\al}{\alpha}
\newcommand{\ga}{\gamma}
\newcommand{\ep}{\varepsilon}
\newcommand{\ka}{\kappa}

\newcommand{\si}{\sigma}
\newcommand{\vp}{\varphi}

\newcommand{\laa}{\Lambda}

\newcommand{\D}{{\mathbb D}}
\newcommand{\E}{{\mathbb E}}

\newcommand{\PP}{{\mathbb P}}

\newcommand{\R}{{\mathbb R}}
\newcommand{\X}{{\mathbb X}}

\newcommand{\bx}{{\bf x}}

\newcommand{\1}{{\bf 1}}

\newcommand{\lla}{\left\langle}
\newcommand{\rra}{\right\rangle}
\newcommand{\lcl}{\left\{}
\newcommand{\rcl}{\right\}}
\newcommand{\lp}{\left(}
\newcommand{\rp}{\right)}
\newcommand{\lc}{\left[}
\newcommand{\rc}{\right]}
\newcommand{\lln}{\left|}
\newcommand{\rrn}{\right|}

\newcommand{\bean}{\begin{eqnarray*}}
\newcommand{\eean}{\end{eqnarray*}}
\newcommand{\ben}{\begin{enumerate}}
\newcommand{\een}{\end{enumerate}}
\newcommand{\beq}{\begin{equation}}
\newcommand{\eeq}{\end{equation}}

\newtheorem{theorem}{Theorem}[section]

\newtheorem{corollary}[theorem]{Corollary}

\newtheorem{definition}[theorem]{Definition}

\newtheorem{hypothesis}[theorem]{Hypothesis}
\newtheorem{lemma}[theorem]{Lemma}
\newtheorem{notation}[theorem]{Notation}

\newtheorem{proposition}[theorem]{Proposition}

\theoremstyle{remark}
\newtheorem{remark}[theorem]{Remark}
\theoremstyle{remark}

\begin{document}

\title[Skorohod and Stratonovich in the plane]{Skorohod and Stratonovich integration in the plane}

\author{Khalil Chouk \and Samy Tindel}

\address{Samy Tindel, Institut {\'E}lie Cartan and INRIA Nancy-Grand Est, Universit\'e de Lorraine, B.P. 239,
54506 Vand{\oe}uvre-l{\`e}s-Nancy, France.}
\email{samy.tindel@univ-lorraine.fr}

\address{Khalil Chouk, Ceremade,
Universit\'e de Paris-Dauphine,
75116 Paris, France}
\email{chouk@ceremade.dauphine.fr}

\thanks{This article was finalized while S. Tindel was visiting the University of Kansas. He wishes to express his gratitude to this institution for its warm hospitality.}

\keywords{Young integrals; Rough path; Stochastic integral; Malliavin calculus}
\date{\today}

\begin{abstract}
This article gives an account on various aspects of stochastic calculus in the plane. Specifically, our aim is 3-fold: (i) Derive a pathwise change of variable formula for a path $x:[0,1]^{2}\to\R$ satisfying some H\"older regularity conditions with a H\"older exponent greater than $1/3$. (ii) Get some Skorohod change of variable formulas for a large class of Gaussian processes defined on $[0,1]^{2}$. (iii) Compare the bidimensional integrals obtained with those two methods, computing explicit correction terms whenever possible. As a byproduct, we also give explicit forms of corrections in the respective change of variable formulas. 
\end{abstract}

\maketitle

\tableofcontents

\section{Introduction}
Stochastic calculus for processes indexed by the plane (or higher order objects) is notoriously a cumbersome topic. In order to get an idea of this fact, let us start from the simplest situation of a smooth function $x$ indexed by $[0,1]^{2}$ and a regular function $\varphi\in C^{2}(\R)$. Then some elementary computations show that 
\begin{equation}\label{eq:simple-change-vb}
[\der \vp(x)]_{s_1s_2;t_1t_2}
=\int_{[s_1,s_2]\times[t_1,t_2]} \vp^{(1)}(x_{u;v}) \, d_{uv}x_{u;v}
+ \int_{[s_1,s_2]\times[t_1,t_2]} \vp^{(2)}(x_{u;v}) \, d_{u}x_{u;v} d_{v}x_{u;v},
\end{equation}
for all $0\le s_1<s_2\le 1$ and $0\le t_1<t_2\le 1$, where we have set $[\der y]_{s_1s_2;t_1t_2}$ for the planar increment of $y$ in the rectangle $[s_1,s_2]\times[t_1,t_2]$, namely
\begin{equation}\label{eq:def-planar-increments-intro}
[\der y]_{s_1s_2;t_1t_2} 
\equiv 
y_{s_2;t_2} - y_{s_1;t_2} - y_{s_2;t_1} + y_{s_1;t_1}.
\end{equation}
This simple formula already exhibits the extra term $\int \vp^{(2)}(x_{u;v}) \, d_{u}x \, d_{v}x$ with respect to integration in $\R$, and the mixed differential term $d_{u}x \, d_{v}x$ is one of the main source of complications when one tries to extend \eqref{eq:simple-change-vb} to more complex situations. 

\smallskip

Moving to stochastic calculus in the plane, generalizations of \eqref{eq:simple-change-vb} to a random process $x$ obviously starts with change of variables formulas involving the Brownian sheet or martingales indexed by the plane. Relevant references include \cite{CW,Nu84,WZ}, and some common features of the formulas produced in these articles are the following:
\begin{itemize}
\item 
Higher order derivatives of $f$ showing up.
\item
Mixed differentials involving partial derivatives of $x$ and quadratic variation type elements.
\item
Huge number of terms in the formula due to boundary effects.
\end{itemize}
This non compact form of stochastic calculus in the plane has certainly been an obstacle to its development, and we shall go back to this problem later on.

\smallskip

Some recent advances in generalized stochastic calculus have also paved the way to change of variables formulas in the plane beyond the martingale case. One has to distinguish two type of contributions in this direction:

\smallskip

\noindent
\emph{(a)} Skorohod type formulas for the fractional Brownian sheet (abbreviated as fBs in the sequel) with Hurst parameters greater than $1/2$ have been obtained in \cite{TVp1} thanks to a combination of differential calculus in the plane and stochastic analysis tools inspired by ~\cite{AMN}. A subsequent generalization to Hurst parameters smaller than $1/2$ is available in \cite{TVp2}, invoking the notion of extended divergence introduced in \cite{LN}. Notice however that the extended divergence leads to a rather weak notion of integral, and might not be necessary when the Hurst parameters of the fBs are greater than $1/4$.

\smallskip

\noindent
\emph{(b)} The article \cite{Gp} focuses on pathwise methods for stochastic calculus in the plane, and builds an analog of the rough paths theory for functions indexed by the plane. In particular, generalizations of \eqref{eq:simple-change-vb}  with Stratonovich type integrals are given for functions with H\"older regularity greater than $1/3$. The construction is deterministic and general, and only requires the existence of a stack of iterated integrals of $x$ called rough path, denoted by $\X$. One can show in particular that $\X$ exists when $x$ is a fBs.

\smallskip

The current article is a contribution to these recent advances on generalized stochastic calculus in the plane. Namely, we focus on 3 different problems: (i) A complete exposition of the Stratonovich type change of variables formula obtained through rough paths techniques. (ii) Generalization of  \cite{TVp1} to a fairly general Gaussian process $x$. (iii) Comparison of Stratonovich and Skorohod formulas, analogously to the 1 dimensional situation handled in~\cite{HJT}. Before we further comment on these contributions, we now describe our main results more specifically.

\subsection{Some general notation}
\label{sec:general-notation}
Before we can turn to the description of our main results, we introduce some general notation concerning differential calculus in the plane. Let us mention first that we shall separate as much as possible the first and the second direction of integration, which will be respectively be denoted by direction 1 and direction 2. Thus the evaluation of a function $f:[0,1]^{2k}\to\R$ will be denoted by $f_{s_{1}\cdots s_{k};t_{1}\cdots t_{k}}$.  We also set $d_{12}x$ for the differential $d_{uv}x$ and $d_{1}x\,d_{2}x$ for the differential $d_{u}x \, d_{v}x$. In fact, since the differential element $d_{1}x\,d_{2}x$ is essential for our purposes, we further shorten it into $d_{\hone\htwo}x$.

\smallskip

Another notation which will be used extensively throughout the paper is the following: we set $y=\vp(x)$, and for all $j\ge 1$ we write $y^{j}$ for the function $\vp^{(j)}(x)$. With those first shorthands, equation \eqref{eq:simple-change-vb} for a smooth function $x:[0,1]^{2}\to\R$ can be written as
\begin{equation}\label{eq:strato-intro-notation-d12}
\der y = \int_{1}\int_{2} y^{1} \, d_{12}x + \int_{1}\int_{2} y^{2} \, d_{\hone\htwo}x.
\end{equation}
This kind of compact notation is of course useful when cumbersome computations come into the picture.

\smallskip

Let us anticipate a little on the notation for planar increments which will be introduced at Section \ref{sec:planar-increments}: we denote by $\cp_{k,l}$ the set of $\R$-valued functions involving $k$ variables in direction 1 and $l$ variables in direction 2, satisfying some vanishing conditions on diagonals. We mostly deal with spaces of the form $\cp_{2,2}$ and introduce some H\"older norms. Namely, if $f\in\cp_{2,2}(V)$, we set
\begin{equation*}
\|f\|_{\ga_1;\, \ga_2} =
\sup\lcl  \frac{|f_{s_1 s_2; \, t_1 t_2}|}{|s_2-s_1|^{\ga_1} |t_2-t_1|^{\ga_2}} \, ; \, s_1,s_2,t_1,t_2\in\ott\rcl ,
\end{equation*}
and we denote by $\cp_{2,2}^{\ga_1, \ga_2}(V)$ the space of increments in $\cp_{2,2}(V)$ whose $\|\cdot \|_{\ga_1;\, \ga_2}$ norm is finite. 

\subsection{Stratonovich type formula in the Young case}
\label{sec:strato-young-intro}
We assume here that $x:[0,1]^{2}\to\R$ is a path such that the rectangular increments $\der x$ of $x$ satisfy  $\der x\in\cp_{2,2}^{\ga_{1},\ga_{2}}$ with $\ga_{1},\ga_{2}>1/2$, which corresponds to the case where integration with respect to $x$ can be handled by Young techniques in the plane. Our change of variable formula in this situation relies on the definition of 2 increments $\bx^{1;2},\bx^{\hone;\htwo}\in\cp_{2,2}^{\ga_{1},\ga_{2}}$ defined as follows (see also Definition \ref{not:1st-order-increments-x} for further information):
\begin{equation*}
\bx^{1;2} = \int_{1}\int_{2} d_{12}x, 
\quad\text{and}\quad
\bx^{\hone;\htwo} = \int_{1}\int_{2} d_{1}x \, d_{2}x,
\end{equation*}
where the integrals can be understood in the Young sense.

\smallskip

With these notations in hand, the change of variables formula can be read as:

\begin{theorem}\label{thm:young-strato-intro}
Let $x:[0,1]^{2}\to\R$ be a path such that $\der x\in\cp_{2,2}^{\ga_{1},\ga_{2}}$ with $\ga_{1},\ga_{2}>1/2$. Then the increments
\begin{equation}\label{eq:def-z1-z2-intro}
z^{1} = \int_{1}\int_{2} y^{1} \, d_{12}x,
\quad\text{and}\quad
z^{2} = \int_{1}\int_{2} y^{2} \, d_{\hone\htwo}x,
\end{equation}
are well defined in the 2d-Young sense. Moreover:

\smallskip

\noindent
\emph{(i)} Both $z^{1}$ and $z^{2}$ can be decomposed as:
\begin{equation}\label{eq:dcp-z1-z2-intro}
z^{1} = y^{1} \, \bx^{1;2} + \rho^{1},
\quad\text{and}\quad
z^{2} = y^{2} \, \bx^{\hone;\htwo} + \rho^{2},
\end{equation}
where $\rho^{1},\rho^{2}$ are sums of increments with double regularity $(2\ga_1,2\ga_2)$ in at least one direction.

\smallskip

\noindent
\emph{(ii)} 
Provided $x$ is a smooth path, the increments $z^{1}$ and $z^{2}$ are defined as Riemann-Stieljes integrals.

\smallskip

\noindent
\emph{(iii)} 
If $x^{n}$ is a sequence of smooth functions such that the related increments $\bx^{n;1;2}, \bx^{n;\hone;\htwo}$ converge respectively to $\bx^{1;2}$ and $\bx^{\hone;\htwo}$ in $\cp_{2,2}^{\ga_{1},\ga_{2}}$, then $z^{1,n},z^{2,n}$ also converge respectively to $z^{1}$ and $z^{2}$.

\smallskip

\noindent
\emph{(iv)} Some Riemann sums convergences hold true: if $\pi^{1}_{n}$ and  $\pi^{2}_{n}$ are 2 partitions of $[s_{1},s_{2}]\times[t_{1},t_{2}]$ whose mesh goes to 0 as $n\to\infty$, then
\begin{equation}\label{eq:cvgce-riem-sums-z1-z2}
\lim_{n\to\infty} \sum_{\pi^{1}_{n},\pi^{2}_{n}} y^{1}_{\si_{i};\tau_{j}} \, \bx^{1;2}_{\si_{i}\si_{i+1};\tau_{j}\tau_{j+1}} 
=z^{1}_{s_{1}s_{2};t_{1}t_{2}},
\ \, \text{and}\ \,
\lim_{n\to\infty} \sum_{\pi^{1}_{n},\pi^{2}_{n}} y^{2}_{\si_{i};\tau_{j}} \, \bx^{\hone;\htwo}_{\si_{i}\si_{i+1};\tau_{j}\tau_{j+1}} 
=z^{2}_{s_{1}s_{2};t_{1}t_{2}}.
\end{equation}

\smallskip

\noindent
\emph{(v)} The change of variables formula \eqref{eq:strato-intro-notation-d12} still holds true when integrals are understood in the Young sense. 
\end{theorem}

Observe that this theorem is not new and can be easily recovered from considerations contained in \cite{FV-bk,QT}. However, we express it here in terms which allow easy generalizations to Skorohod type integrals and to rough situations as well.

\subsection{Stratonovich type formula in the rough case}
\label{sec:strato-rough-intro}

Consider now a function $x$ whose  rectangular increments $\der x$ only satisfy  $\der x\in\cp_{2,2}^{\ga_{1},\ga_{2}}$ with $\ga_{1},\ga_{2}>1/3$. The definition of $z^{1},z^{2}$ as in \eqref{eq:def-z1-z2-intro} and the equivalent of formula \eqref{eq:simple-change-vb} require now a huge additional effort. In particular, the correct definition relies on the introduction of a collection of iterated integrals of $x$ (called rough path above $x$ and denoted by $\X$ by analogy with the 1-d case) that we proceed  to describe.

\smallskip

The reader will soon observe that the definition of $\X$ involves a whole zoology of objects which are somehow tedious to describe. In this article we shall index those objects by the directions of integration, trying to separate as much as possible direction 1 and direction 2 as we already did for the first order integrals $\bx^{1;2}$ and $\bx^{\hone;\htwo}$. Moreover, when one tries to define iterated integrals in the plane, the following extra facts have to be taken into account: 

\smallskip

\noindent
\emph{(i)} The differentials with respect to $x$ can be in one direction only ($d_{1}x$ or $d_{2}x$) or bidirectional. This reflects into some indices 0 when we don't integrate in a given direction, and 1 or 2 otherwise. Furthermore, as already mentioned,  our bidirectional differentials can be either of type $d_{12}x$ or $d_{1}x \, d_{2}x=d_{\hone\htwo}x$. We keep our convention of indices $1;2$ for differentials of the type $d_{12}x$ and $\hone;\htwo$ for differentials of the type $d_{\hone\htwo}x$. As an example of these conventions, we define $\bx^{1\hone;0\htwo}\in\cp_{2,2}$ in the following way for a smooth function $x$:
\begin{equation*}
\bx^{1\hone;0\htwo} = \int_{1} d_{1}x \int_{2} d_{\hone\htwo}x,
\quad\text{that is}\quad
\bx^{1\hone;0\htwo}_{s_{1}s_{2};t_{1}t_{2}} = 
\int_{s_{1}}^{s_{2}}\int_{t_{1}}^{t_{2}}
\lp \int_{s_{1}}^{\si_{2}} d_{1}x_{\si_{1};t_{1}} \rp d_{1}x_{\si_{2};\tau_{1}} d_{2}x_{\si_{2};\tau_{1}}.
\end{equation*}

\smallskip

\noindent
\emph{(ii)} We manipulate objects which are either iterated integrals or products of iterated integrals. We indicate that one starts a new integral in one specific direction and gets a product of increments by placing a new $\int$ sign, and this is translated by a $\cdot$ in the indices of $\bx$. For instance, modifying our previous example, we define $\bx^{1\cdot\hone;0\htwo}\in\cp_{3,2}$ in the following way for a smooth function $x$:
\begin{equation*}
\bx^{1\cdot\hone;0\htwo} = \int_{1} d_{1}x \int_{1}\int_{2} d_{\hone\htwo}x,
\quad\text{that is}\quad
\bx^{1\cdot\hone;0\htwo}_{s_{1}s_{2}s_{3};t_{1}t_{2}} = 
\int_{s_{1}}^{s_{2}} d_{1}x_{\si_{1};t_{1}}
\int_{s_{2}}^{s_{3}}\int_{t_{1}}^{t_{2}}
 d_{1}x_{\si_{2};\tau_{1}} d_{2}x_{\si_{2};\tau_{1}}.
\end{equation*}
Notice that those breaks in integration can occur at different steps in each direction 1 or~2. The resulting overlapping integrals will be an important source of technical troubles in the remainder of the paper.

\smallskip

With these preliminary considerations in mind, our assumptions on the function $x$ are of the following form:
\begin{hypothesis}\label{hyp:rough-path-intro}
The function $x$ is such that $\der x\in\cp_{2,2}^{\ga_{1},\ga_{2}}$ with $\ga_{1},\ga_{2}>1/3$. Moreover, the following rough path $\X$ can be constructed out of $x$:
\begin{table}[htdp]
%\caption{default}
\begin{center}
\begin{tabular}{|c|c|c||c|c|c|}
\hline
\emph{Increment} & \emph{Interpretation} & \emph{Regularity} & \emph{Increment} & \emph{Interpretation} & \emph{Regularity} \\
\hline
$\bx^{1;2}$ & $\int_{1}\int_{2} d_{12}x$ & $(\ga_{1},\ga_{2})$ & 
$\bx^{\hone;\htwo}$ & $\int_{1}\int_{2} d_{\hone\htwo}x$ & $(\ga_{1},\ga_{2})$ \\
\hline
$\bx^{11;02}$ & $\int_{1}d_{1}x\int_{2} d_{12}x$ & $(2\ga_{1},\ga_{2})$ & 
$\bx^{1\hone;0\htwo}$ & $\int_{1}d_{1}x\int_{2} d_{\hone\htwo}x$ & $(2\ga_{1},\ga_{2})$ \\
\hline
$\bx^{01;22}$ & $\int_{2}d_{2}x\int_{1} d_{12}x$ & $(\ga_{1},2\ga_{2})$ & 
$\bx^{0\hone;2\htwo}$ & $\int_{2}d_{2}x\int_{1} d_{\hone\htwo}x$ & $(\ga_{1},2\ga_{2})$ \\
\hline
$\bx^{11;22}$ & $\int_{1}\int_{2} d_{12}x d_{12}x$ & $(2\ga_{1},2\ga_{2})$ & 
$\bx^{1\hone;2\htwo}$ & $\int_{1}\int_{2} d_{12}x d_{\hone\htwo}x$ & $(2\ga_{1},2\ga_{2})$ \\
\hline
$\bx^{\hone1;\htwo2}$ & $\int_{1}\int_{2} d_{\hone\htwo}x d_{12}x$ & $(2\ga_{1},2\ga_{2})$ & 
$\bx^{\hone\hone;\htwo\htwo}$ & $\int_{1}\int_{2}  d_{\hone\htwo}x d_{\hone\htwo}x$ & $(2\ga_{1},2\ga_{2})$ \\
\hline
\end{tabular}
\end{center}
%\label{default}
\end{table}%

In the table above, all increments belong to $\cp_{2,2}$, so that a regularity $(\al,\beta)$ means that the increment lyes into $\cp_{2,2}^{\al,\beta}$. Furthermore, the stack $\X$ is a geometric rough path, insofar as there exists a regularization $x^{n}$ of $x$ such that $\lim_{n\to\infty}\|x-x^{n}\|_{\ga_1,\ga_2}=0$ and such that all the integrals in $\X^{n}$, constructed out of $x^{n}$ in the Lebesgue-Stieljes sense, converge with respect to their natural respective norms in $\cp_{2,2}^{\ga_1,\ga_2}$, $\cp_{2,2}^{2\ga_1,\ga_2}$, $\cp_{2,2}^{\ga_1,2\ga_2}$ or $\cp_{2,2}^{2\ga_1,2\ga_2}$. Note that the natural H\"older norm of a rough path is denoted by $\cn$ in the sequel.
\end{hypothesis}

\begin{remark}\label{rmk:complete-rough-path}
As we shall see at Section \ref{sec:pathwise-strato}, Hypothesis \ref{hyp:rough-path-intro} is not completely sufficient in order to settle a satisfying integration theory with respect to $x$. In fact the rough path $\X$ should also include higher order increments like $\bx^{11\cdot1;022}$ or $\bx^{1\cdot11;22\cdot2}$ (and other extra terms). We have only stated Hypothesis \ref{hyp:rough-path-intro} here in order to keep our exposition into some reasonable bounds.
\end{remark}

Now we can state our Stratonovich integration theorem in the rough case, which mimics Theorem \ref{thm:young-strato-intro}:

\begin{theorem}\label{thm:rough-strato-intro}
Let $x:[0,1]^{2}\to\R$ be a path such that $\der x\in\cp_{2,2}^{\ga_{1},\ga_{2}}$ with $\ga_{1},\ga_{2}>1/3$ and assume the further rough path Hypothesis~\ref{hyp:rough-path-intro}. Consider a function $\vp\in C_{b}^{8}(\R)$. Then the increments $z^{1}$ and $z^{2}$ given by \eqref{eq:def-z1-z2-intro} are well defined as continuous functions of the rough path $\X$. Moreover:

\smallskip

\noindent
\emph{(i)} 
The increment $z^{1}$ can be decomposed as:
\begin{equation}\label{eq:strato-der-z1-rough}
z^{1} = 
y^1 \, \bx^{1;2} + y^{2} \, \bx^{11;02} + y^{2} \, \bx^{01;22} 
+ y^{2} \, \bx^{11;22} + y^{3} \, \bx^{\hone1;\htwo2} + \rho^{1},
\end{equation}
and the increment $z^{2}$ admits a decomposition of the form 
\begin{equation*}
z^{2} = 
y^2 \, \bx^{\hone;\htwo} + y^{3} \, \bx^{1\hone;0\htwo} + y^{3} \, \bx^{0\hone;2\htwo} 
+ y^{3} \, \bx^{1\hone;2\htwo} + y^{4} \, \bx^{\hone\hone;\htwo\htwo} + \rho^{2},
\end{equation*}
where $\rho^{1},\rho^{2}$ are sums of increments with triple regularity $(3\ga_1,3\ga_2)$ in at least one direction.

\smallskip

\noindent
\emph{(ii)} If $x^{n}$ is a sequence of smooth functions such that the related rough path $\X^{n}$ converges to $\X$, then $z^{1,n},z^{2,n}$ (defined in the Lebesgue-Stieljes sense) also converge respectively to $z^{1}$ and $z^{2}$.

\smallskip

\noindent
\emph{(iii)} The change of variables formula \eqref{eq:strato-intro-notation-d12} still holds true when integrals are understood in the rough path sense. 
\end{theorem}

Obviously, Theorem \ref{thm:rough-strato-intro} would be of little interest if we could not apply it to processes of interest. To this regard, our guiding example will be the fractional Brownian sheet (fBs in the sequel). Let us recall that this is a centered Gaussian process $x$ defined on $[0,1]^{2}$, with a covariance function $R_{s_{1}s_{2};t_{1}t_{2}}  =\E[x_{s_{1};t_{1}}x_{s_{2};t_{2}}]$ defined by 
\begin{equation}\label{eq:def-cov-fBs}
R_{s_{1}s_{2};t_{1}t_{2}}  
=\frac14 \lp |s_1|^{2\ga_{1}} + |s_2|^{2\ga_{1}} - |s_1-s_2|^{2\ga_{1}}\rp 
\lp |t_1|^{2\ga_{2}} + |t_2|^{2\ga_{2}} - |t_1-t_2|^{2\ga_{2}}\rp,
\end{equation}
where the Hurst parameters $\ga_{1},\ga_{2}$ lye into $(0,1)$.
Many possible representations are available for the fBs, among which we will appeal to the so-called harmonizable representation (see relation \eqref{eq:harmonizable-fBs} below for further details). This allows a natural approximation of $x$ by a sequence of smooth processes $x^{n}$ thanks to a cutoff in frequency, and we recall the following convergence result established in~\cite{Gp}:
\begin{proposition}
Let $x$ a fBs with Hurst parameters $\gamma_j>1/3$, for $j=1,2$. Define the regularization $x^n$  of $x$ given by a frequency cutoff  on $B(0,n)$ in the harmonizable representation of $x$. Then:

\smallskip

\noindent
\emph{(i)}
The family of iterated integral $\mathbb X^n$ defined in~\eqref{hyp:rough-path-intro} associated to $x^n$ fulfills the relation $\lim_{n,m\to\infty}\E[ \cn^{p}(\X^{m}-\X^{n}) ] = 0$
for all $p\ge 1$, where the norm $\cn$ is alluded to at Hypothesis~\ref{hyp:rough-path-intro}.
The limit object $\mathbb X$ is called rough sheet associated to $x$.

\smallskip

\noindent
\emph{(ii)}
Theorem \ref{thm:rough-strato-intro} applies to the fBs $x$.
\end{proposition}

As the reader might imagine, Theorem \ref{thm:rough-strato-intro} can also be applied to a wide range of Gaussian and non Gaussian processes. We focus here on fBs for sake of simplicity.

\subsection{Skorohod integration}
One of the main issue alluded to in this article is a comparison between Stratonovich and Skorohod type change of variable formulas when $x$ is a Gaussian process exhibiting some H\"older regularity in the plane. Towards this aim, our global strategy is to use our Theorems \ref{thm:young-strato-intro} and \ref{thm:rough-strato-intro} and compute corrections between  Stratonovich and Skorohod type integrals. 

\smallskip

We first focus on the Young case, assuming the same regularity conditions as in Section~\ref{sec:strato-young-intro}. We are then able to handle the case of a fairly general centered Gaussian process $x$ whose covariance function $R$ satisfies a factorization property  of the  form
\begin{equation}\label{eq:factorize-R-intro}
\E[x_{s_{1};t_{1}}x_{s_{2};t_{2}}]=R_{s_{1}s_{2};t_{1}t_{2}}=R^1_{{s_{1}s_{2}}}R^2_{t_{1}t_{2}},
\end{equation}
for two covariance functions $R^{1},R^{2}$ on $[0,1]$ and such that $R^{1},R^{2}\in\cac^{1\text{-var}}([0,1]^2)$ (which ensures that $x$ is $(\ga_1,\ga_2)$-H\"older continuous with $\ga_1,\ga_2>1/2$). Notice in particular that the fBs covariance function \eqref{eq:def-cov-fBs} satisfies condition \eqref{eq:factorize-R-intro}.

\smallskip

The standard growth assumptions on $f$ in order to get a Skorohod formula for $f(x)$ should also be met. They will feature prominently in the sequel, and we proceed to recall them now:

\begin{definition}\label{growuthcond}
Let $k\in\mathbb N$, we will say that a function $f\in C^k(\mathbb R)$ satisfies the growth
condition (GC) if there exist positive constants $c$ and $\lambda$
such that
\begin{equation}\label{eq-gc-1}
\lambda<\frac1{4\,\max_{s,t\in[0,1]}\lp R_{s}^{1}R_{t}^{2}\rp},
\quad\mbox{and}\quad
\max_{l=0,. . .,k}|f^{(l)}(\xi)|\le c\,e^{\lambda\,|\xi|^2}\,\ \text{for all }\xi\in\R.
\end{equation}
\end{definition}

With these notations in hand, and denoting quite informally  the Skorohod differentials by $d^\diamond$ (see Section \ref{sec:mall-calc-x} for further explanations), we can summarize our results in the following:

\begin{theorem}\label{thm:young-skoro-intro}
Assume $x$ is a centered Gaussian process on $[0,1]^{2}$ with a covariance function satisfying \eqref{eq:factorize-R-intro}. Consider a function $\vp\in C^{4}(\mathbb R)$ satisfying condition (GC). Then the increments
\begin{equation}\label{eq:def-z1-z2-wick-intro}
z^{1,\diamond} = \int_{1}\int_{2} y^{1} \, d^\diamond_{12}x,
\quad\text{and}\quad
z^{2,\diamond} = \int_{1}\int_{2} y^{2} \, d^\diamond_{\hone\htwo}x,
\end{equation}
are well defined in the Skorohod sense of Malliavin calculus. Moreover:

\smallskip

\noindent
\emph{(i)} Some Riemann convergences hold true: if $\pi^{1}_{n}$ and  $\pi^{2}_{n}$ are 2 partitions of $[s_{1},s_{2}]\times[t_{1},t_{2}]$ whose mesh goes to 0 as $n\to\infty$, then
\begin{eqnarray}
\lim_{n\to\infty} \sum_{\pi^{1}_{n},\pi^{2}_{n}} 
y^{1}_{\si_{i};\tau_{j}} \diamond \bx^{1;2}_{\si_{i}\si_{i+1};\tau_{j}\tau_{j+1}} 
&=&z^{1,\diamond}_{s_{1}s_{2};t_{1}t_{2}}
\label{eq:riemann-wick-z1}\\
\lim_{n\to\infty} \sum_{\pi^{1}_{n},\pi^{2}_{n}} 
y^{2}_{\si_{i};\tau_{j}} \diamond\der_2x_{s_i;t_jt_{j+1}}\diamond\der_1x_{s_is_{i+1};t_j} 
&=&z^{2,\diamond}_{s_{1}s_{2};t_{1}t_{2}},  \label{eq:riemann-wick-z2}
\end{eqnarray}
where $\diamond$ stands for the Wick product in the left hand side of the relations above, and where the convergence holds in both a.s and $L^{2}(\Omega)$ sense.

\smallskip

\noindent
\emph{(ii)} The change of variables formula for $y=f(x)$ becomes
\begin{multline}\label{eq:skor-formula-smooth}
\der y_{s;t}=z^{1,\diamond}_{s;t}+z^{2,\diamond}_{s;t}
+\frac12\int_{1}\int_{2} y^2_{u;v} \, d_{1}R^1_u \, d_{2}R^2_v
+\frac12\int_{1}\int_{2} y^3_{u;v}R^1_{u} \, d_{2}R^2_v \, d^{\diamond}_1x_{u;v}\\
+\frac12\int_{1}\int_{2} y^3_{u;v}R^{2}_{v} \, d_{1}R^1_u \, d^{\diamond}_2x_{u;v}
+\frac14\int_{1}\int_{2} y^4_{u;v}R^1_uR^2_v \, d_{1}R^1_u \, d_{2}R^2_v.
\end{multline}

\smallskip

\noindent
\emph{(iii)} Explicit corrections between $z^{1}$, $z^{2}$ and $z^{1,\diamond}$, $z^{2,\diamond}$ can be computed (see relations \eqref{eq:rel-sko-strato-d12x-young} and~\eqref{eq:strato-skoro-mixed-intg}).

\end{theorem}

Finally, let us move to the Skorohod change of variable in the rough situation. For simplicity of exposition, we have restricted our analysis to the fractional Brownian sheet, mainly because our computations heavily hinges on the explicit regular approximation sequence $x^{n}$ given by the harmonizable representation of fBs (similarly to the construction of the rough path above $x$). 
The Skorohod change of variable (consistent with the formulas obtained in \cite{TVp2}) and Skorohod-Stratonovich comparison we obtain in this case are summarized as follows:
 
\begin{theorem}\label{thm:fBs-skoro-intro}
Assume $x$ is a fractional Brownian sheet on $[0,1]^{2}$, with $\ga_{j}>1/3$ for $j=1,2$. Then the increments $z^{1,\diamond},z^{2,\diamond}$ of equation \eqref{eq:def-z1-z2-intro} are well defined in the Skorohod sense of Malliavin calculus. Moreover:

\smallskip

\noindent
\emph{(i)} Both $z^{1,\diamond}$ and $z^{2,\diamond}$ can be seen as respective limits of $z^{n,1,\diamond}$ and $z^{n,2,\diamond}$, computed as in Theorem \ref{thm:young-skoro-intro} for the regularized process $x^{n}$.

\smallskip

\noindent
\emph{(ii)} For all $f\in C^{6}(\mathbb R)$, the change of variables formula \eqref{eq:skor-formula-smooth} still holds, and can be read as:
\begin{multline}\label{eq:skor-formula-fBs}
\der y_{s;t}=z^{1,\diamond}+z^{2,\diamond}
+ 2\ga_{1}\ga_{2} \int_{1}\int_{2} y^{2}_{u;v} \, u^{2\ga_{1}-1} v^{2\ga_{2}-1} \, du dv
+ \ga_{2} \int_{1}\int_{2} y^{3}_{u;v} \, u^{2\ga_{1}} v^{2\ga_{2}-1} \, d^{\diamond}_{1}x_{u;v} dv \\
+ \ga_{1} \int_{1}\int_{2} y^{3}_{u;v} \, u^{2\ga_{1}-1} v^{2\ga_{2}} \, d^{\diamond}_{2}x_{u;v} du
+\ga_{1}\ga_{2} \int_{1}\int_{2} y^{4}_{u;v} \, u^{4\ga_{1}-1} v^{4\ga_{2}-1} \, du dv.
\end{multline}

\smallskip

\noindent
\emph{(iii)} Explicit corrections between $z^{1}$, $z^{2}$ and $z^{1,\diamond}$, $z^{2,\diamond}$ can be computed (see relations \eqref{eq:correc-sko-strat-f-d12x} and \eqref{eq:correc-ito-strato-mixed-fBs}).
\end{theorem}

\subsection{Further comments}
As the reader might have noticed, our paper gives a rather complete picture of pathwise Stratonovich and It\^o-Skorohod integration for processes indexed by the plane. In the case of pathwise integration, we take up the methodology introduced in~\cite{Gp}, based on a highly nontrivial extension of the rough path theory, and compute explicitly all the terms involved in our Stratonovich expansion. We wished the exposition to be as clear and self-contained as possible, which explains the length of Sections \ref{sec:planar-young} and  \ref{sec:pathwise-strato}.

\smallskip

In order to put our strategy for the It\^o case into perspective, notice that 2 types of methodologies are usually available for changes of variables in case of a Gaussian process $x$: 

\smallskip

\noindent
\emph{(a)} Define a divergence type operator $\der^{\diamond}$ for $x$ and proceed by integration by parts on expressions like $\E[\der f(x) \, G]$, where $G$ is a smooth functional of $x$. This is the strategy invoked e.g. in \cite{AMN,HJT}. 

\smallskip

\noindent
\emph{(b)} Base the calculations on the pathwise change of variables formula of type \eqref{eq:simple-change-vb}. This formula is generally related to some converging Riemann sums like in Theorem \ref{thm:young-strato-intro}, and one can compute corrections between Wick and ordinary products in relation \eqref{eq:cvgce-riem-sums-z1-z2}. This is the method implicitly adopted in \cite{TVp1} and we also resort to this second strategy here, which allows to derive our Skorohod formula and its comparison with the Stratonovich formula at the same time.

\smallskip

\noindent
Unfortunately, the Wick corrections strategy does not work for the rough case, even in the explicit situation of a fractional Brownian sheet. This mainly stems from the fact that  convenient Riemann sums related to formula \eqref{eq:simple-change-vb} are not available (so far) in the case of Theorem~\ref{thm:rough-strato-intro}. This drawback led us to change our strategy again, and proceed by regularization. Indeed, as mentioned before, one can come up with an explicit regular approximation $x^{n}$ of $x$. For this regularization, we can apply Theorem \ref{thm:young-skoro-intro} and get some It\^o-Stratonovich corrections. Invoking the fact that $\der^{\diamond}$ is a closable operator, we can then take limits in our operations as $n\to\infty$. This allows to compare the changes of variables formulas \eqref{eq:simple-change-vb} and \eqref{eq:skor-formula-fBs}, but the interpretation in terms of Riemann-Wick sums is obviously lost in this case. Notice that an approximation procedure (expressed in terms of the extended divergence operator) is also at the heart of \cite{TVp2} for irregular fBs.

\smallskip

Finally, let us say a few words about possible extensions of our work:

\smallskip

\noindent
$\bullet$ Generalizations of Skorohod's change of variable to a Gaussian process without the factorization hypothesis \eqref{eq:factorize-R-intro} on the covariance function of $x$ are certainly possible. However, at a technical level, one should be aware of the fact that the analysis of mixed terms like $\int_{1}\int_{2} y^3_{u;v}R^{2}_{v} \, d_{1}R^1_u \, d^{\di}_{2}x_{u;v}$ would require tools of Young integration in dimension 4. These techniques have been used e.g in \cite{FR}, and the elaboration we need would certainly be cumbersome. We have thus sticked to the factorized case for $R$ for sake of readability.

\smallskip

\noindent
$\bullet$ As mentioned before, our strategy for the Skorohod formula in the rough case relies heavily on a suitable regularization of $x$. Instead of treating the explicit fBs example, we could have stated some general approximation assumptions satisfied in the fBs case. Once again, we have chosen to specialize our study here for sake of clarity. The general case might be handled in a subsequent paper, and we also hope to design a strategy based on Riemann-Wick sums in the next future.

\smallskip

Here is how our article is structured: We recall some basic notation of algebraic integration in dimension 1 at Section \ref{sec:one-dim}, and extend it to integration in the plane at Section~\ref{sec:alg-intg-plane}. The Stratonovich change of variable formula is handled at Section \ref{sec:planar-young} for the Young case and at Section \ref{sec:pathwise-strato} in the rough situation. We then move to Skorohod type formulas at Sections \ref{sec:young-vs-malliavin} and~\ref{sec:skorohod-rough}, respectively for the regular and rough cases.

\section{Algebraic integration in dimension 1}
\label{sec:one-dim}

We recall here the minimal amount of notation concerning algebraic integration theory in $\R$, in order to prepare the ground for further developments in the plane. We refer to ~\cite{Gu,GT} for a more detailed introduction.

\subsection{Increments}\label{incr}

The extended pathwise integration we will deal with is based on the
notion of \emph{increments}, together with an elementary operator $\der$
acting on them. The algebraic structure they generate is described
in \cite{Gu,GT}, but here we present  directly the definitions of
interest for us, for sake of conciseness. First of all,  for a vector space $V$ and an integer $k\ge
1$ we denote by $\cac_k(V)$ the set of functions $g : [0,1]^{k} \to
V$ such that $g_{t_1 \cdots t_{k}} = 0$ whenever $t_i = t_{i+1}$ for
some $i\le k-1$. Such a function will be called a
\emph{$(k-1)$-increment}, and we   set $\cac_*(V)=\cup_{k\ge
1}\cac_k(V)$. We can now define the announced elementary operator
$\der$ on $\cac_k(V)$:
\begin{equation}
 \label{eq:coboundary}
\delta : \cac_k(V) \to \cac_{k+1}(V), \qquad
(\delta g)_{t_1 \cdots t_{k+1}} = \sum_{i=1}^{k+1} (-1)^{k-i}
g_{t_1  \cdots \hat t_i \cdots t_{k+1}} ,
\end{equation}
where $\hat t_i$ means that this particular argument is omitted.
A fundamental property of $\der$, which is easily verified,
is that
$\delta \delta = 0$, where $\delta \delta$ is considered as an operator
from $\cac_k(V)$ to $\cac_{k+2}(V)$.
We    denote $\cz\cac_k(V) = \cac_k(V) \cap \text{Ker}\delta$
and $\cb \cac_k(V) =
\cac_k(V) \cap \text{Im}\delta$.

\vspace{0.3cm}

Some simple examples of actions of $\der$,
which will be the ones we will really use throughout the paper,
are obtained by letting
$g\in\cac_1$ and $h\in\cac_2$. Then, for any $s,u,t\in\ott$, we have
\begin{equation}
\label{eq:simple_application}
 \der g_{st} = g_t - g_s,
\quad\mbox{ and }\quad
\der h_{sut} = h_{st}-h_{su}-h_{ut}.
\end{equation}
Furthermore, it is easily checked that
$\cz \cac_{k+1}(V) = \cb \cac_{k}(V)$ for any $k\ge 1$.
In particular, the following basic property holds:
\begin{lemma}\label{exd}
Let $k\ge 1$ and $h\in \cz\cac_{k+1}(V)$. Then there exists a (non unique)
$f\in\cac_{k}(V)$ such that $h=\der f$.
\end{lemma}

Lemma \ref{exd} can be rephrased as follows: any element
$h \in\cac_2(V)$ such that $\der h= 0$ can be written as $h = \der f$
for some (non unique) $f \in \cac_1(V)$. Thus we get a heuristic
interpretation of $\der |_{\cac_2(V)}$:  it measures how much a
given 1-increment  is far from being an  exact increment of a
function, i.e., a finite difference.

\vspace{0.3cm}

Notice that our future discussions will mainly rely on
$k$-increments with $k \le 2$, for which we will make  some
analytical assumptions. Namely,
we measure the size of these increments by H\"older norms
defined in the following way: for $f \in \cac_2(V)$ let
\begin{equation}\label{eq:def-norm-C2}
\| f\|_{\mu} =
\sup_{s,t\in\ott}\frac{|f_{st}|}{|t-s|^\mu},
\quad\mbox{and}\quad
\cac_2^\mu(V)=\lcl f \in \cac_2(V);\, \|f\|_{\mu}<\infty  \rcl.
\end{equation}
Obviously, the usual H\"older spaces $\cac_1^\mu(V)$ will be determined
       in the following way: for a continuous function $g\in\cac_1(V)$, we simply set
\begin{equation}\label{def:hnorm-c1}
\|g\|_{\mu}=\|\der g\|_{\mu},
\end{equation}
and we will say that $g\in\cac_1^\mu(V)$ iff $\|g\|_{\mu}$ is finite.
Notice that $\|\cdot\|_{\mu}$ is only a semi-norm on $\cac_1(V)$.
For $h \in \cac_3(V)$ set in the same way
\begin{eqnarray}
 \label{eq:normOCC2}
 \|h\|_{\gamma,\rho} &=& \sup_{s,u,t\in\ott}
\frac{|h_{sut}|}{|u-s|^\gamma |t-u|^\rho}\\
%\quad\mbox{and}\quad
\|h\|_\mu &= &
\inf\left \{\sum_i \|h_i\|_{\rho_i,\mu-\rho_i} ;\, h =
\sum_i h_i,\, 0 < \rho_i < \mu \right\} ,\nonumber
\end{eqnarray}
where the last infimum is taken over all sequences $\{h_i \in \cac_3(V) \}$
such that $h
= \sum_i h_i$ and for all choices of the numbers $\rho_i \in (0,z)$.
Then  $\|\cdot\|_\mu$ is easily seen to be a norm on $\cac_3(V)$, and we set
$$
\cac_3^\mu(V):=\lcl h\in\cac_3(V);\, \|h\|_\mu<\infty \rcl.
$$
Eventually,
let $\cac_3^{1+}(V) = \cup_{\mu > 1} \cac_3^\mu(V)$,
and notice  that the same kind of norms can be considered on the
spaces $\cz \cac_3(V)$, leading to the definition of some spaces
$\cz \cac_3^\mu(V)$ and $\cz \cac_3^{1+}(V)$.

\vspace{0.3cm}

With these notations in mind
the following proposition is a basic result, which  belongs to  the core of
our approach to pathwise integration. Its proof may be found
in a simple form in \cite{GT}.
\begin{proposition}[The $\Lambda$-map]
\label{prop:Lambda}
There exists a unique linear map $\Lambda: \cz \cac^{1+}_3(V)
\to \cac_2^{1+}(V)$ such that
$$
\delta \Lambda  = \id_{\cz \cac_3^{1+}(V)}
\quad \mbox{ and } \quad \quad
\Lambda  \delta= \id_{\cac_2^{1+}(V)}.
$$
In other words, for any $h\in\cac^{1+}_3(V)$ such that $\der h=0$
there exists a unique $g=\laa(h)\in\cac_2^{1+}(V)$ such that $\der g=h$.
Furthermore, for any $\mu > 1$,
the map $\laa$ is continuous from $\cz \cac^{\mu}_3(V)$
to $\cac_2^{\mu}(V)$ and we have
\begin{equation}\label{ineqla}
\|\Lambda h\|_{\mu} \le \frac{1}{2^\mu-2} \|h\|_{\mu} ,\qquad h \in
\cz \cac^{\mu}_3(V).
\end{equation}
\end{proposition}

\smallskip

Let us mention at this point a first link between the structures
we have introduced so far and the problem of integration of irregular
functions.

\begin{corollary}
\label{cor:integration}
For any 1-increment $g\in\cac_2 (V)$ such that $\der g\in\cac_3^{1+}$,
set
$
\delta f = (\id-\Lambda \delta) g
$.
Then
$$
\delta f_{st} = \lim_{|\Pi_{st}| \to 0} \sum_{i=0}^{n-1}
g_{t_i\,t_{i+1}},
$$
where the limit is over any partition $\Pi_{st} = \{t_0=s,\dots,
t_n=t\}$ of $[s,t]$, whose mesh tends to zero. Thus, the 1-increment
$\delta f$ is the indefinite integral of the 1-increment $g$.
\end{corollary}

\subsection{Products of increments}\label{sec:comp-cac}

For notational sake, let us specialize now to the case $V=\R$, and just write $\cac_{k}^{\ga}$ for $\cac_{k}^{\ga}(\R)$. The usual product of two increments considered on $(\cac_*,\delta)$ is obtained by gluing one variable in each increment (see e.g \cite{Gu,GT}):
\begin{definition}\label{def:pdt-increments-1d}
For  $g\in\cac_n$ and $h\in\cac_m$, we denote by  $gh$ the element of $\cac_{n+m-1}$ defined by
\begin{equation}\label{eq:convention-pdt}
(gh)_{t_1,\dots,t_{m+n-1}}=
g_{t_1,\dots,t_{n}} h_{t_{n},\dots,t_{m+n-1}},
\quad
t_1,\dots,t_{m+n-1}\in\ott.
\end{equation}
\end{definition}
However, another product (defined without gluing of variables) turns out to be useful for further computations in the plane. This product is called \emph{splitting} and is defined below:
\begin{definition}
For  $g\in\cac_n$ and $h\in\cac_m$, we denote by  $S(g,h)$ the element of $\cac_{n}\otimes\cac_{m}$ defined by
\begin{equation}\label{eq:convention-split}
\lc S(g,h) \rc_{t_1,\dots,t_{m+n}}=
g_{t_1,\dots,t_{n}} h_{t_{n+1},\dots,t_{m+n}},
\quad
t_1,\dots,t_{m+n}\in\ott.
\end{equation}
Notice that $S(g,h)$ can also be considered as an increment in $\cac_{n+m}$, except that it is not required to vanish when $t_{n}=t_{n+1}$.
\end{definition}
The splitting operation will have to be inverted at some point. The inverse operation is called \emph{gluing}:
\begin{definition}\label{def:gluing}
For $n,m\ge 1$, we say that $f$ is an element of $\cm_{n,m}$ if it can be written as a finite linear combination of the following type:
\begin{equation}\label{eq:typical-f-Mmn}
f_{t_{1},\ldots, t_{n+m}} = \sum_{j\in J} \al_j \, g_{t_1,\dots,t_{n}}^{j} \, h_{t_{n+1},\dots,t_{m+n}}^{j},
\quad\mbox{with}\quad 
\al_j\in\R, g^{j}\in\cac_{n}, h^{j}\in\cac_{m}.
\end{equation}
For $f$ of the form \eqref{eq:typical-f-Mmn}, we then define $G(f)\in\cac_{n+m-1}$ as:
\begin{equation*}
[G(f)]_{t_{1},\ldots, t_{n+m-1}}
= \sum_{j\in J} \al_j \, g_{t_1,\dots,t_{n}}^{j} \, h_{t_{n},\dots,t_{m+n-1}}^{j}.
\end{equation*}
\end{definition}
With these definitions in mind, let us remark that if $f\in\cm_{n,m}$ is simply of the form $f_{t_{1},\ldots, t_{n+m}} = g_{t_1,\dots,t_{n}} \, h_{t_{n+1},\dots,t_{m+n}}$ with $g\in\cac_{n}$ and  $h\in\cac_{m}$, then $G(f)=g\, h$.

\smallskip

We now recall some elementary properties concerning products of increments:
\begin{proposition}\label{prop:difrul}
The following differentiation rules hold true:
\begin{enumerate}
\item
Let $g\in\cac_1$ and $h\in\cac_1$. Then
$gh\in\cac_1$ and
\begin{equation}\label{eq:difrulu}
\der (gh) = \der g\,  h + g\, \der h.
\end{equation}
\item
Let $g\in\cac_1$ and $h\in\cac_2$. Then
$gh\in\cac_2$ and
\begin{equation}\label{eq:difrul-C1-C2}
\der (gh) = - \der g\, h + g \,\der h.
\end{equation}
\item
Let $g\in\cac_2$ and $h\in\cac_1$. Then
$gh\in\cac_2$ and
\begin{equation}
\der (gh) = \der g\, h  + g \,\der h.
\end{equation}
\end{enumerate}
\end{proposition}

\begin{proof}
We will just prove (\ref{eq:difrulu}), the other relations being just as  simple.
If $g,h\in\cac_1 $, then
$$
\lc \der (gh) \rc_{st}
= g_th_t-g_sh_s
=g_s\lp h_t-h_s \rp +\lp  g_t-g_s\rp h_t\\
=g_s \lp \der h \rp_{st}+ \lp \der g \rp_{st} h_t,
$$
which proves our claim.

\end{proof}

\subsection{Iterated integrals as increments}

Iterated integrals of smooth functions on $\ott$ are obviously
particular cases of elements of $\cac_2$, which will be of interest
for us. A typical example of this kind of object is given as follows: consider $f^j\in\cac_1^\infty$ for $j=1,\ldots,n$ and $0\le s_1<s_2\le 1$. For $n\ge 1$, we denote by $\cs_n(s_1,s_2)$ the simplex
\begin{equation}\label{eq:def-simplex}
\cs_n(s_1,s_2)=
\lcl
(\si_1,\ldots,\si_n)\in\ott^n ; \, s_1<\si_1<\cdots<\si_n<s_2
\rcl
\end{equation}
and  we set
\begin{equation}\label{eq:def-iterated-intg}
h_{s_1s_2}^{1,\ldots,n} \equiv
\int_{\cs_n(s_1,s_2)} df_{\si_{1}}^{1} \cdots df_{\si_{n}}^{n}
= \int_{s_1}^{s_2} \int_{s_{1}}^{\si_{n-1}} \cdots \int_{s_{1}}^{\si_{2}} df_{\si_{1}}^{1}   \cdots  df_{\si_{n}}^{n}.
\end{equation}

\smallskip

We now introduce some notation for iterated integrals which is much too complicated for integration in dimension 1, but turns out to be useful for integration in the plane. Indeed, we can alternatively denote the increment $h^{1,\ldots,n}$ defined at \eqref{eq:def-iterated-intg} by
\begin{equation}\label{eq:notation-intg-multilin}
h^{1,\ldots,n} = [ \underbrace{d, \ldots, d}_{n \, {\rm times}} ] (f^1,\ldots , f^n),
\quad\mbox{or}\quad
h^{1,\ldots,n} = \int df^1\cdots  df^n,
\end{equation}
where the integration on the $n$-dimensional simplex is implicit in both cases. We shall also need a small variant of these conventions: we set
\begin{equation}\label{eq:def-partial-intg}
[\underbrace{\id, \ldots, \id}_{j \, {\rm times}} , \underbrace{d, \ldots, d}_{n-j \, {\rm times}}] (f^1,\ldots , f^n)
\equiv f^{1}\cdots f^{j} \int df^{j+1}\cdots  df^n,
\end{equation}
where all the products are understood as products of increments as in Definition \ref{def:pdt-increments-1d}.

\smallskip

With these conventions in mind, the following relations between multiple integrals and the operator $\der$ will also be useful. The reader is sent to \cite{GT} for its elementary proof.
\begin{proposition}\label{prop:dif-intg-1}
Let $f\in\cac_1^\infty$ and $g\in\cac_1^\infty$.
Then it holds that
$$
\der g = \int dg, \qquad \der\lp \int f dg\rp = 0, \qquad \der\lp
\int df dg \rp = \der f \, \der g ,
$$
and
$$
\der \lp [ d, \ldots, d ] (f^1,\ldots , f^n)\rp  =
\sum_{i=1}^{n-1}
 [ d, \ldots, d ] (f^1,\ldots , f^{j}) \,  [ d, \ldots, d ] (f^{j+1},\ldots , f^n).
$$
\end{proposition}

\section{Algebraic integration in the plane}
\label{sec:alg-intg-plane}

This section is devoted to recall the elements of algebraic integration necessary to define an integral of the form $\int_{[0,1]^{2}} f(x) \, dx$ for a H\"older function $x$ in the plane with H\"older exponent greater than $1/3$. This requires a tensorization of the algebraic structures defined in the previous section, plus some extra tools that we proceed to introduce.

\subsection{Planar increments}
\label{sec:planar-increments}

We consider here increments of a variable $s$ (also called direction 1) and a variable $t$ (also called direction 2), with $(s,t)\in [0,1]^2$. For a vector space $V$, we set
\begin{equation*}
\cp_{k,l}(V)=
\lcl
f\in\cac([0,1]^k\times[0,1]^l ;\, V) ; \, f_{s_1\cdots s_k; \, t_1\cdots t_l}=0 \mbox{ whenever } s_i=s_{i+1}
\mbox{ or } t_j=t_{j+1}
\rcl.
\end{equation*}
In the particular case $V=\R$, we simply set $\cp_{k,l}(\R)\equiv \cp_{k,l}$.

\smallskip

Some partial difference operators $\doo$ and $\dt$ with respect to the first and second direction can be defined as in the previous section. Namely, for $f\in \cp_{k,l}(V)$ we set
\begin{equation*}
\delta_1 : \cp_{k,l}(V) \to \cp_{k+1,l}(V), \qquad
\doo g_{s_1 \cdots s_{k+1}; \, t_1\cdots t_l} = \sum_{i=1}^{k+1} (-1)^{k-i}
g_{s_1  \cdots \hat s_i \cdots s_{k+1}; \, t_1\cdots t_l} ,
\end{equation*}
and we define $\dt$ similarly. The planar increment $\der$ is then obtained as $\der=\doo\,\dt$. Notice that for $f\in\cp_{1,1}$ we have
\begin{equation*}
\der f_{s_1s_2 ; \, t_1 t_2}
=f_{s_2; \, t_2} - f_{s_2; \, t_1} - f_{s_1; \, t_2} + f_{s_1; \, t_1},
\end{equation*}
which is the usual rectangular increment of a function $f$ defined on $\ott^2$ and is consistent with formula \eqref{eq:def-planar-increments-intro}. Let us label the following notation for further use:
\begin{notation}
For $j=1,2$, we set $\cz_j\cp_{k,l}= \cp_{k,l}\cap \ker(\der_j)$ and $\cb_j\cp_{k,l}= \cp_{k,l}\cap {\rm Im}(\der_j)$. We also write $\cz\cp_{k,l}$ for $\cp_{k,l}\cap \ker(\der)$ and $\cb\cp_{k,l}$ for $\cp_{k,l}\cap {\rm Im}(\der)$.
\end{notation}

\smallskip

As in the 1-d case, the H\"older regularity of planar increments is an essential feature of our generalized integration theory. On $\cp_{2,2}(V)$ and $\cp_{3,3}(V)$, it is  measured by a tensorization of the H\"older norms defined at \eqref{eq:def-norm-C2} and \eqref{eq:normOCC2}. Namely, if $f\in\cp_{2,2}(V)$, we set
\begin{equation*}
\|f\|_{\ga_1;\, \ga_2} =
\sup\lcl  \frac{|f_{s_1 s_2; \, t_1 t_2}|}{|s_2-s_1|^{\ga_1} |t_2-t_1|^{\ga_2}} \, ; \, s_1,s_2,t_1,t_2\in\ott\rcl ,
\end{equation*}
and we denote by $\cp_{2,2}^{\ga_1, \ga_2}(V)$ the space of increments in $\cp_{2,2}(V)$ whose $\|\cdot \|_{\ga_1;\, \ga_2}$ norm is finite. Along the same lines, we say that $h\in \cp_{3,3}^{\ga_1, \ga_2}(V)$ if there exist $\ka_1,\ka_2,\rho_1,\rho_2$ such that $\ka_j+\rho_j=\ga_j$, $j=1,2$, and
\begin{equation*}
\sup\lcl  \frac{|h_{s_1 s_2 s_3; \, t_1 t_2 t_3}|}{|s_2-s_1|^{\ka_1} |s_3-s_2|^{\rho_1} |t_2-t_1|^{\ka_2} |t_3-t_2|^{\rho_2}} 
\, ; \, s_1,s_2,s_3,t_1,t_2,t_3\in\ott\rcl < \infty.
\end{equation*}
Similar norms, omitted here for sake of conciseness, can be defined on $\cp_{2,3}(V)$ and $\cp_{3,2}(V)$.

\smallskip

For H\"older continuous increments with regularity greater than 1, one gets the following inversion properties, which are a direct consequence of the one dimensional Proposition~\ref{prop:Lambda}:
\begin{proposition}\label{proposition:planar Lambda}
Let $\ga_1,\ga_2>1$. Then:

\smallskip

\noindent\emph{(1)}
There exist two maps $\laa_1:\cb_1\cp_{3,3}^{\ga_1,\ga_2}\to\cp_{2,3}^{\ga_1,\ga_2}$ and $\laa_2:\cb_2\cp_{3,3}^{\ga_1,\ga_2}\to\cp_{3,2}^{\ga_1,\ga_2}$ such that $\der_j\laa_j=\id$. These maps satisfy the bound $\|\laa_j(h)\|_{\ga_1,\ga_2} \le c_{\ga_j} \|h\|_{\ga_1,\ga_2}$ for $j=1,2$.

\smallskip

\noindent\emph{(2)}
There exists a map $\laa:\cb\cp_{3,3}^{\ga_1,\ga_2}\to\cp_{2,2}^{\ga_1,\ga_2}$ such that $\der\laa=\id$. This map satisfies the bound $\|\laa(h)\|_{\ga_1,\ga_2} \le c_{\ga_1,\ga_2} \|h\|_{\ga_1,\ga_2}$.
\end{proposition}
We do not include the proof of this proposition for sake of conciseness. Let us just mention that (as the reader might imagine) we have $\laa=\laa_1\laa_2$. It should also be observed that some 2-dimensional Riemann sums are related to the sewing map $\laa$, echoing Corollary~\ref{cor:integration}:

\begin{proposition}\label{proposition:integ-2d}
Let $g\in\cp_{2,2}$ satisfying the following assumptions:
\begin{equation*}
\doo g \in \cp_{3,2}^{\ga_1,*}, \qquad
\dt g \in \cp_{2,3}^{*,\ga_2}, \qquad
\der g \in \cp_{3,3}^{\ga_1,\ga_2},
\end{equation*}
for $\ga_1,\ga_2>1$, where $*$ denotes any kind of H\"older regularity. Then there exists $f\in\cp^{1,1}$ such that
\begin{equation*}
\der f = \lc  \id - \laa_1 \doo \rc \lc  \id - \laa_2 \dt \rc g,
\quad\mbox{and}\quad
\lim_{|\pi|\to 0}\sum_{\si_i,\tau_j\in\pi} g_{\si_{i} \si_{i+1} ;  \tau_{j} \tau_{j+1}} = \der f_{s_1 s_2; t_1 t_2},
\end{equation*}
where $\pi$ designates a family of rectangular partitions of $[s_1,s_2]\times[t_1,t_2]$ whose mesh goes to 0.
\end{proposition}

\subsection{Products of planar increments}
This section is a parallel of Section \ref{sec:comp-cac}, and we mainly deal here with a state space $V=\R$. We describe the different conventions on products of 2-d increments which will be used in the sequel, starting from the equivalent of Definition \ref{def:pdt-increments-1d}:
\begin{definition}
For  $g\in\cp_{n_1,n_2}$ and $h\in\cp_{m_1,m_2}$, we denote by  $gh$ the element lying in the space $\cp_{n_1+m_1-1,n_2+m_2-1}$ defined by
\begin{equation*}
(gh)_{s_1,\ldots s_{n_1+m_1-1}; t_1,\dots,t_{n_2+m_2-1}}=
g_{s_1,\ldots s_{n_1}; t_1,\dots,t_{n_2}} h_{s_{n_1},\ldots s_{n_1+m_1-1}; t_{n_2},\dots,t_{n_2+m_2-1}}.
\end{equation*}
\end{definition}

We now define the equivalent of splitting for increments in $\cp$:
\begin{definition}\label{def:splitting-plane}
Let $g\in\cp_{n_1,n_2}$ and $h\in\cp_{m_1,m_2}$. Then:

\smallskip
\noindent $\bullet$ 
The partial splitting $S_1(g,h)$ is the element of $\cac_{n_2+m_2-1}(\cac_{n_1}\otimes\cac_{m_1})$ defined by
\begin{equation*}
\lc S_1(g,h) \rc_{s_1,\ldots s_{n_1+m_1}; t_1,\dots,t_{n_2+m_2-1}}
= g_{s_1,\ldots s_{n_1}; t_1,\dots,t_{n_2}} h_{s_{n_1+1},\ldots s_{n_1+m_1}; t_{n_2},\dots,t_{n_2+m_2-1}}.
\end{equation*}

\smallskip
\noindent $\bullet$ 
The partial splitting $S_2(g,h)$ is the element of $\cac_{n_1+m_1-1}(\cac_{n_2}\otimes\cac_{m_2})$ defined by
\begin{equation*}
\lc S_2(g,h) \rc_{s_1,\ldots s_{n_1+m_1-1}; t_1,\dots,t_{n_2+m_2}}
= g_{s_1,\ldots s_{n_1}; t_1,\dots,t_{n_1}} h_{s_{n_1},\ldots s_{n_1+m_1-1}; t_{n_2+1},\dots,t_{n_2+m_2}}.
\end{equation*}

\smallskip
\noindent $\bullet$ 
The  splitting $S(g,h)$ is the element of $\cac_{n_1}\otimes\cac_{m_1}(\cac_{n_2}\otimes\cac_{m_2})$ defined by
\begin{equation*}
\lc S_1(g,h) \rc_{s_1,\ldots s_{n_1+m_1}; t_1,\dots,t_{n_2+m_2}}
= g_{s_1,\ldots s_{n_1}; t_1,\dots,t_{n_2}} h_{s_{n_1+1},\ldots s_{n_1+m_1}; t_{n_2+1},\dots,t_{n_2+m_2}}.
\end{equation*}
\end{definition}
Notice that inverse operations of splittings can also be defined for planar increments. For sake of conciseness, we let the patient reader  generalize Definition \ref{def:gluing} in order to get the definition of gluing $G_1,G_2$ and $G$ for planar increments.

\smallskip

We close this section by introducing a last product of increments which is labeled for further computations.
\begin{definition}\label{def:g-circ-h}
Let $g\in\cp_{2,1}$ and $h\in\cp_{1,2}$. Then $g\circ h$ is the increment in $\cp_{2,2}$ defined by $[g\circ h]_{s_1s_2;t_1t_2}=g_{s_1s_2;t_1} h_{s_1;t_1t_2}$.
\end{definition}

\subsection{Iterated integrals as increments in the plane}
\label{sec:iter-intg-plane}

The relationship between iterated integrals and increments in the plane is crucial for us. Generally speaking, an iterated integral is given as follows: consider $f^j\in\cp_{1,1}^\infty$ for $j=1,\ldots,n$ and $(s_1,s_2), (t_1,t_2)\in\cs_{2}$, where we recall relation \eqref{eq:def-simplex} defining  simplexes. Then we set
\begin{eqnarray}\label{eq:def-iterated-intg-plane}
h_{s_1s_2; t_1,t_2}^{1,\ldots,n} &\equiv &
\int_{\cs_n(s_1,s_2) \times \cs_n(t_1,t_2)} 
d_{12}f_{\si_{1}; \tau_{1}}^{1} \cdots d_{12}f_{\si_{n};\tau_{n}}^{n}  \\
&=& \int_{s_1}^{s_2} \int_{t_1}^{t_2} \int_{s_{1}}^{\si_{n-1}} \int_{t_{1}}^{\tau_{n-1}} 
\cdots \int_{s_{1}}^{\si_{2}} \int_{t_{1}}^{\tau_{2}} 
d_{12}f_{\si_{1}; \tau_{1}}^{1} \cdots d_{12}f_{\si_{n};\tau_{n}}^{n} ,  \notag
\end{eqnarray}
where we recall from Section \ref{sec:general-notation} that $d_{12}f_{\si; \tau}^{j}$ stands for $\partial_{\si \tau}^{2}f_{\si; \tau}^{j}$.

\smallskip

Expression \eqref{eq:def-iterated-intg-plane} is obviously cumbersome, and it could in particular become clearer by separating the $s,\si$ from the $t,\tau$ variables. This is where the conventions introduced in equation \eqref{eq:notation-intg-multilin} turn out to be useful. Namely, one can simply tensorize \eqref{eq:notation-intg-multilin} in order to write the increment $h^{1,\ldots,n}$ defined at \eqref{eq:def-iterated-intg-plane} as
\begin{equation}\label{eq:def-intg-mult-2d}
h^{1,\ldots,n} = 
[ d_{1}, \ldots, d_{1} ] \otimes
[ d_{2}, \ldots, d_{2} ] \,
(f^1,\ldots , f^n),
\quad\mbox{or}\quad
h^{1,\ldots,n} = \int_{1} \int_{2} d_{12}f^1\cdots  d_{12}f^n,
\end{equation}
and notice that we will mainly use the second convention throughout the paper.
This notation proves to be particularly convenient when one is faced with partial integrations (as introduced in \eqref{eq:def-partial-intg}) in both directions 1 and 2. In order to illustrate this point, let us consider the simple example
\begin{equation}\label{eq:example-mixed-intg}
g\equiv [ d_{1}, d_{1} ] \otimes [ \id , d_{2} ] \, (f^1 , f^2)
= \int_1 d_1 f^1 \int_2 d_{12} f^2 .
\end{equation}
Let us now describe the algorithm which allows to go from expression \eqref{eq:example-mixed-intg} to an integral like~\eqref{eq:def-iterated-intg-plane}. It can be summarized as follows:
\begin{itemize}
\item 
For direction 1, count the number of iterated integrals starting from the left hand side (in our example this number is 2). Then interpret these integrals as integrals on the simplex in direction 1 and write the 1-variables.
\item
Do the same for direction 2. In our example, there is only one integral in this direction, so that variable 2 is frozen in the first differential $d_1f^1$.
\end{itemize}
Applying this algorithm, the reader can easily check that $g$ defined at \eqref{eq:example-mixed-intg} can be written as
\begin{equation*}
g_{s_1s_2;t_1 t_2}=
\int_{s_1<\si_1<\si_2<s_2} d_{1}f_{\si_1;t_{1}}^{1}  \int_{t_1}^{t_2} d_{12}f_{\si_2;\tau_{1}}^{2}.
\end{equation*}
For sake of conciseness, we omit generalizations of this simple example.

\section{Planar Young integration}
\label{sec:planar-young}

Before going into the computational details of Young integration, let us describe the general strategy we shall follow in order to obtain our It\^o-Stratonovich type change of variable formulae in case of non smooth functions $x$. Indeed, we start from a smooth approximation $x^{n}$ to our path $x$ and we introduce a useful notation for the remainder of the computations:

\begin{notation}\label{not:smooth-path-y-fx}
We shall drop the index $n$ of approximations in $x^{n}$, which means that $x$ will stand for a generic smooth path defined on $\ou^{2}$. For a smooth function $\vp:\R\to\R$, we also write $y$ for the path $\vp(x)$ and for all $j\ge 1$ we set $y^{j}=\vp^{(j)}(x)$.
\end{notation}

With these notations in hand, for a smooth sheet $x$ and $f\in\cac_b^{2}$ it is well known that formula \eqref{eq:simple-change-vb} holds true. Recall that we have written this relation under the following form, compatible with our convention \eqref{eq:def-intg-mult-2d}:
\begin{equation}\label{eq:ito-plane-smooth2}
\der y = \int_{1}\int_{2} y^{1} \, d_{12}x + \int_{1}\int_{2} y^{2} \, d_{\hone\htwo}x.
\end{equation}
We shall see that this formula still holds true in the limit for $x$, except that the integrals involved in the right hand side of \eqref{eq:ito-plane-smooth2} have to be interpreted in a sense which goes beyond the Riemann-Stieltjes case. Our main task will thus be to obtain a definition of $\int_{1}\int_{2} y^{1} \, d_{12}x$ and $ \int_{1}\int_{2} y^{2} \, d_{\hone\htwo}x$ involving iterated integrals of $x$ and increments of $y$ (or $y^{j}$ for $j\ge 1$) only. Though this task might overlap with some aspects of \cite{Gp}, we present it here because it is short enough and allows us to introduce part of our formalism.

\smallskip

Let us introduce what will be later interpreted as the first order elements of the planar rough path above $x$:
\begin{notation}\label{not:1st-order-increments-x}
Let $x\in\cp_{1,1}^{\ga_1,\ga_2}$ with $\ga_1,\ga_2>1/2$. We set
\begin{equation*}
\bx^{1;2} = \der x, 
\quad\text{and}\quad
\bx^{\hone;\htwo} = [\id-\laa_{1}\der_{1}] [\id-\laa_{2}\der_{2}] \left( \der_{1}x \, \der_{2}x\right).
\end{equation*}
Notice that for smooth functions we also have
\begin{equation*}
\bx^{1;2} = \int_{1}\int_{2} d_{12}x, 
\quad\text{and}\quad
\bx^{\hone;\htwo} = \int_{1}\int_{2} d_{1}x \, d_{2}x.
\end{equation*}
\end{notation}

We are now ready to express the integrals in \eqref{eq:ito-plane-smooth2} in terms of the planar sewing map $(\Lambda_i)_{i=1,2}$ and $\Lambda=\Lambda_1\Lambda_2$ and the increments introduced in Notation \ref{not:1st-order-increments-x}.

\subsection{Change of variables formula}

The following theorem gives the analog of relation~\eqref{eq:ito-plane-smooth2} in the Young case, and is a way to recast Theorem \ref{thm:young-strato-intro}. Notice that some extensions of these results are contained in~\cite{Gp}.

\begin{theorem}\label{prop:young-2d-y-dx}
Let $x\in\cp_{1,1}^{\gamma_1,\gamma_2}$ with $\ga_1,\ga_2>1/2$ and $\varphi\in C^4(\mathbb R)$. With our notations \ref{not:1st-order-increments-x} and \ref{not:smooth-path-y-fx} in mind, define two increments $z^{1},z^{2}$ as
\begin{equation}\label{eq:def-z1-z2-(Id-Lambda)}
z^{1} = [\id-\laa_{1}\der_{1}] [\id-\laa_{2}\der_{2}] ( y^1 \, \bx^{1;2} ),
\quad\text{and}\quad
z^{2} = [\id-\laa_{1}\der_{1}] [\id-\laa_{2}\der_{2}] ( y^2 \, \bx^{\hone;\htwo} ).
\end{equation}
Then items (i)--(v) of Theorem \ref{thm:young-strato-intro} hold true. Furthermore, for $j=1,2$ the following bound is satisfied: $\|z^{j}\|_{\ga_{1};\ga_{2}}\le c_{\vp} (1 + \|x\|_{\ga_{1};\ga_{2}}^{2})$.
\end{theorem}

\begin{proof}
We first introduce a formalism which will feature prominently in the rough case treated in the next section.

\smallskip

\noindent\textit{Step 1: Setting for our computations.}
Let us write the first step of our expansion in a usual integration language: for $s_1,s_2,t_1,t_2\in\ott$ and a continuously differentiable function $x$ we have
\begin{equation*}
z^{1}_{s_1s_2;t_1t_2} = \int_{s_1}^{s_2} \int_{t_1}^{t_2} y^1_{\si_1;\tau_1} \, d_{12}x_{\si_1;\tau_1},
\end{equation*}
where we have written $d_{12}x_{\si_1;\tau_1}$ instead of $d_{\si\tau}x_{\si_1;\tau_1}$.
Then according to the elementary identity $y_{\si_1;\tau_1}=y_{s_1;t_1}+\der_2 y^1_{\si_1;t_1\tau_1} $ we obtain
\begin{equation*}
z^{1}_{s_1s_2;t_1t_2} 
= \int_{s_1}^{s_2} y^1_{\si_1;t_1} \int_{t_1}^{t_2}  \, d_{12}x_{\si_1;\tau_1}
+  \int_{s_1}^{s_2} \int_{t_1}^{t_2} \der_2 y^1_{\si_1;t_1\tau_1} \, d_{12}x_{\si_1;\tau_1}.
\end{equation*}
Going on with this procedure, we end up with a decomposition of the form
\begin{multline}\label{eq:young-dcp-with-indices}
z^{1}_{s_1s_2;t_1t_2} 
= y^1_{s_1;t_1} \der x_{s_1s_2;t_1t_2}  
+ \int_{s_1}^{s_2} \der_1 y^1_{s_1\si_1;t_1} \int_{t_1}^{t_2}   d_{12}x_{\si_1;\tau_1}  \\
+ \int_{t_1}^{t_2} \der_2 y^1_{s_1;t_1\tau_1} \int_{s_1}^{s_2} d_{12}x_{\si_1;\tau_1}  
+\int_{s_1}^{s_2} \int_{t_1}^{t_2} \der y^1_{s_1\si_1;t_1\tau_1} \, d_{12}x_{\si_1;\tau_1}.
\end{multline}
It should be observed that for H\"older regularities smaller than $1/2$ the above decomposition is not sufficient to yield the application of $\laa$. Since further calculations with all explicit indices are cumbersome, we shall now show how to translate the above computations with the formalism of Section \ref{sec:iter-intg-plane}: we simply write
\begin{eqnarray}\label{eq:dcp-delta-z-1}
 z^{1} &=& \int_{1}\int_{2} y^1 \, d_{12} x = \int_{1} y^1 \int_{2}  d_{12} x + \int_{1}\int_{2} d_2 y^1 \, d_{12} x  \notag\\
&=& 
y \, \der x + \int_{1} d_1 y^1 \int_{2}  d_{12} x + \int_{2} d_2 y^1 \int_{1}  d_{12} x + \int_{1}\int_{2} d_{12}y^1 \, d_{12} x  \\
&\equiv&
y^1 \, \bx^{1;2} + a^{11;02} + a^{01;22} + a^{11;22}, \notag
\end{eqnarray} 
which is obviously a shorter expression than \eqref{eq:young-dcp-with-indices}. From now on, we shall carry on our computations with this simplified formalism.

\smallskip

\noindent
\textit{Step 2: Analysis of the integrals.}
Consider the term $a^{11;02}$ above. It is readily checked that $\der_{1}a^{11;02}=\der_{1}y^1 \, \bx^{1;2}$. In particular, $\der_{1}a^{11;02} \in \cp_{3,2}^{2\ga_1;\ga_2}$, and since $2\ga_1>1$ one can resort to Proposition \ref{proposition:planar Lambda} in order to get the relation $a^{11;02}=\laa_{1}(\der_{1}y \, \bx^{1;2})$. Proceeding in the same way for $a^{01;22}$ and  $a^{11;22}$ we end up with
\begin{equation}\label{eq:expr-a-young}
a^{11;02}=\laa_{1}(\der_{1}y^1 \, \bx^{1;2}), \quad
a^{01;22}=\laa_{2}(\der_{2}y^1 \, \bx^{1;2}), \quad
a^{11;22}=\laa(\der y^1 \, \bx^{1;2}),
\end{equation}
where we observe that $\vp\in C^{4}(\R)$ is the minimal assumption in order to have $\der y^1\in\cp_{2,2}^{\ga_{1},\ga_{2}}$.
Furthermore, according to relation \eqref{eq:difrul-C1-C2} we have $\der_{1}(y \, \bx^{1;2})=-\der_{1}y^1 \, \bx^{1;2}$, $\der_{2}(y^1 \, \bx^{1;2})=-\der_{2}y^1 \, \bx^{1;2}$ and $\der (y^1 \, \bx^{1;2})=\der y^1 \, \bx^{1;2}$, so that \eqref{eq:expr-a-young} can be expressed as
\begin{equation*}
a^{11;02}=-\laa_{1}\der_{1}(y^1 \, \bx^{1;2}), \quad
a^{01;22}=-\laa_{2}\der_{2}(y^1 \, \bx^{1;2}), \quad
a^{11;22}=\laa\der(y^1 \, \bx^{1;2}).
\end{equation*}
Plugging these relations into \eqref{eq:dcp-delta-z-1} we get 
\begin{equation*}
z^{1} = \int_{1}\int_{2} y^1\, d_{12}x = [\id-\laa_{1}\der_{1}] [\id-\laa_{2}\der_{2}] ( y^1 \, \bx^{1;2} ),
\end{equation*}
for smooth functions $z$, which corresponds to claim (ii) in Theorem \ref{thm:young-strato-intro}. Item (i)-(iii)-(v) are now a matter of straightforward limiting procedures on smooth sheets, and the assertions concerning $z^{2}$ are obtained exactly in the same way. Item (iv) is an easy consequence of Proposition \ref{proposition:integ-2d} and expression \eqref{eq:def-z1-z2-(Id-Lambda)}.

\end{proof}

\subsection{Riemann sums decompositions}
This section is meant as a preparation for Skorohod type computations. Indeed, change of variables in the Skorohod setting involve some mixed integrals with $dx\, dR$ terms, for which a suitable representation is required. It will also be convenient for us to express the integral $\int_{1}\int_{2} y^{2} \, d_1x d_2x$ in different ways, so that we first recall a proposition borrowed from \cite{Gp}:

\begin{proposition}
\label{Young-2d}
let $x\in\cp_{1,1}^{\gamma_1,\gamma_2}$ with $\ga_1,\ga_2>1/2$. Set $y=\vp(x)$ for $\vp\in C^{2}(\R)$ and $z^{2,y}\equiv\int_{1}\int_{2} y \, d_1x d_2x$, understood in the Young sense. Then the following series of identities hold true:
\begin{eqnarray*}
 z^{2,y}&=&  
[\id-\laa_{1}\der_{1}] [\id-\laa_{2}\der_{2}] ( y \, \bx^{\hone;\htwo} )
=  [\id-\laa_{1}\der_{1}] [\id-\laa_{2}\der_{2}] ( y \, \der_{1}x \, \der_{2} x )  \\
&=&  [\id-\laa_{1}\der_{1}] [\id-\laa_{2}\der_{2}] ( y \, \der_{2}x \, \der_{1} x )
=  [\id-\laa_{1}\der_{1}] [\id-\laa_{2}\der_{2}] ( y \, \der_{1}x \circ \der_{2} x ),
\end{eqnarray*}
where we recall that the notation $\circ$ has been introduced at Definition \ref{def:g-circ-h}.
\end{proposition}

The following proposition gives different ways to express the increment $z^{2,y}$ as limit of Riemann sums. 

\begin{proposition}
\label{Y-R sum}
Let $0<s_1<s_2<1$, $0<t_1<t_2<1$ and denote by $\pi_1=(s_i)_i$ and $\pi_2=(t_j)$ some partitions of the intervals $[s_1,s_2]$  and $[t_1,t_2]$ respectively. Then under the assumptions of Proposition~\ref{Young-2d} we have that 
$\int_{s_1}^{s_2}\int_{t_1}^{t_2}y_{st} \, d_{12}x_{st}$ can be written as limit of Riemann sums of the form $\lim_{|\pi_1|,|\pi_2|\to0}\sum_{i,j}y_{s_i;t_j}\der x_{s_{i}s_{i+1};t_{j}t_{j+1}},
$
and recalling that $z^{2,y}\equiv\int_{1}\int_{2} y \, d_1x d_2x$ we also have:
\begin{eqnarray*}
z^{2,y}
&=&\lim_{|\pi_1|,|\pi_2|\to0}\sum_{i,j}y_{s_i;t_j}\der_1x_{s_is_{i+1};t_j}\der_2x_{s_{i+1};t_jt_{j+1}}
\\
&=&\lim_{|\pi_1|,|\pi_2|\to0}\sum_{i,j}y_{s_i;t_j}\der_2x_{s_i;t_jt_{j+1}}\der_1x_{s_is_{i+1};t_{j+1}}
=\lim_{|\pi_1|,|\pi_2|\to0}\sum_{i,j}y_{s_i;t_j}\der_1x_{s_is_{i+1};t_j}\der_2x_{s_i;t_jt_{j+1}}.
\end{eqnarray*}
\end{proposition}

Finally we shall need an extension of the last three propositions to integrals with mixed driving noises:

\begin{proposition}
\label{mixed-Young-2D}
Let $f\in\cp_{1,1}^{\gamma_1,\gamma_2}$,$g\in\cp_{1,1}^{\rho_1,\rho_2}$ and $h\in\cp_{1,1}^{\beta_1,\beta_2}$ such that $\gamma_i+\rho_i>1$, $\beta_i+\gamma_i>1$ and $\beta_i+\rho_i>1$ for $i=1,2$. Set 
\begin{equation*}
\int_{1}\int_{2} f \, d_1g \, d_2 h
= [\id-\laa_{1}\der_{1}] [\id-\laa_{2}\der_{2}] ( f \, \der_{1}g \, \der_{2} h ) .
\end{equation*}
Then we also have
\begin{equation*}
\int_{1}\int_{2} f \, d_1g d_2 h
= [\id-\laa_{1}\der_{1}] [\id-\laa_{2}\der_{2}] ( f \, \der_{2}g \, \der_{1} h )
=  [\id-\laa_{1}\der_{1}] [\id-\laa_{2}\der_{2}] ( f \, \der_{1}g \circ \der_{2} h ).
\end{equation*}
Moreover, taking up the notations of Proposition  \ref{Y-R sum} we have the following Reimann-sum representation of our integrals:
\begin{multline*}
\int_{s_1}^{s_2}\int_{t_1}^{t_2}f_{s;t} \,  d_1 g_{s;t}d_2 h_{s;t}
=\lim_{|\pi|\to0}\sum_{i,j}f_{s_i;t_j}\der_1g_{s_is_{i+1};t_j}\der_2h_{s_{i+1};t_jt_{j+1}}
\\=\lim_{|\pi|\to0}\sum_{i,j}f_{s_i;t_j}\der_1g_{s_is_{i+1};t_j}\der_2h_{s_i;t_jt_{j+1}}
=\lim_{|\pi|\to0}\sum_{i,j}f_{s_i;t_j}\der_2g_{s_i;t_jt_{j+1}}\der_1h_{s_is_{i+1};t_{j+1}}.
\end{multline*}
\end{proposition}

\begin{proof}
Let $a^1=f\,\der_1g\,\der_2h$, $a^2=f\,\der_2g \, \der_1h$ and $a^3=f\,\der_1g\circ\der_2h$. Then by a simple computation we have that 
\begin{eqnarray*}
\der_1a^1&=&
-\der_1f \, \der_1g\,\der_2h-f\,\der_1g\,\der h
\in\cp_{3,2}^{\min(\gamma_1+\rho_1,\rho_1+\beta_1),\beta_2} \\
\der_2a^1&=&
-\der_2f\,\der_1g\,\der_2h-f\,\der g\,\der_2h
\in\cp_{2,3}^{\rho_1,\min(\gamma_2+\beta_2,\rho_2+\beta_2)},
\end{eqnarray*}
and
$$
\der a^1=\der f\,\der_1g\,\der_2h+\der_1f\,\der g\,\der_2h
+\der_2f\,\der_1g\,\der h+f\,\der g\,\der h
\in\cp_{3,3}^{\min(\gamma_1+\rho_1,\rho_1+\beta_1),\min(\gamma_2+\beta_2,\rho_2+\beta_2)}.
$$
This means that $a^{1}$ satisfies the assumptions of Proposition \ref{proposition:planar Lambda}, and the same is readily checked for  $a^2$ and $a^3$. Thus the increments $[\id-\laa_{1}\der_{1}] [\id-\laa_{2}\der_{2}](a^{j})$ are well defined for $j=1,2,3$. We set  $[\id-\laa_{1}\der_{1}] [\id-\laa_{2}\der_{2}](a^{1})=\int_{1}\int_{2} f \, d_1g d_2 h$ since both objects coincide for smooth functions $f,g,h$.

\smallskip

We now identify the increments $[\id-\laa_{1}\der_{1}] [\id-\laa_{2}\der_{2}](a^{j})$ by analyzing their Riemann sums. Indeed, a straightforward application of Proposition \ref{proposition:integ-2d} yields the following limits:
\begin{eqnarray*}
[\id-\laa_{1}\der_{1}] [\id-\laa_{2}\der_{2}](a^{1})&=&
\lim_{|\pi|\to0}\sum_{i,j}f_{s_i;t_j}\der_1g_{s_is_{i+1};t_j}\der_2h_{s_{i+1};t_jt_{j+1}} \\
\text{}[\id-\laa_{1}\der_{1}] [\id-\laa_{2}\der_{2}](a^{1})&=&
\lim_{|\pi|\to0}\sum_{i,j}f_{s_i;t_j}\der_2h_{s_i;t_jt_{j+1}}\der_1g_{s_is_{i+1};t_{j+1}} \\
\text{}[\id-\laa_{1}\der_{1}] [\id-\laa_{2}\der_{2}](a^{3})&=&
\lim_{|\pi|\to0}\sum_{i,j}f_{s_i;t_j}\der_1g_{s_is_{i+1};t_j}\der_2h_{s_i;t_jt_{j+1}}.
\end{eqnarray*}
We now prove that the 3 increments coincide by showing that the differences between Riemann sums vanish when the mesh of the partitions go to zero. Indeed, if we call $\pi_2$ the partition in direction 2, observe for instance that 
\begin{equation}
\begin{split}
\lim_{|\pi_2|\to0}\sum_{i,j}(a^1-a^3)_{s_{i}s_{i+1};t_{j}t_{j+1}}
&=\lim_{|\pi_2|\to0}\sum_{i,j}f_{s_i;t_j}\der_1g_{s_is_{i+1};t_j}\der h_{s_{i}s_{i+1};t_{j}t_{j+1}}
\\&=\sum_i\int_{t_1}^{t_2}f_{s_it}\der_1g_{s_is_{i+1};t} \, d_2\der_1h_{s_is_{i+1};t}
\end{split}
\end{equation}
Now if we remark that $|\int_{t_1}^{t_2}f_{s_it}\der_1g_{s_is_{i+1};t} \, d_th_{s_is_{i+1};t}|\lesssim(s_{i+1}-s_i)^{\rho_1+\beta_1}$, we easily obtain that $\lim_{|\pi_1|\to0}\lim_{|\pi_2|\to0}\sum_{i,j}(a^1-a^3)_{s_{i}s_{i+1};t_{j}t_{j+1}}=0$. The same relation holds true for $a^{2}-a^{3}$ by symmetry, which ends our proof.

\end{proof}

\section{Pathwise Stratonovich formula in the rough case}
\label{sec:pathwise-strato}
In this section, we consider a path $x\in\cp_{1,1}^{\ga_1,\ga_2}$ with $\ga_1,\ga_2>1/3$ and we wish to establish the change of variables formula for $y=\vp(x)$ (with $\vp\in C_{b}^{8}(\R)$) announced in Theorem~\ref{thm:rough-strato-intro}. In such a general context, the integration theory with respect to $x$ relies on the existence of a rough path $\X$ sitting above $x$. Now recall that the first elements of $\X$ have been introduced at Hypothesis~\ref{hyp:rough-path-intro}. However,  as mentioned in Remark \ref{rmk:complete-rough-path}, the rough path still has to be completed and we proceed to its description here.

\smallskip

Let us first introduce another indexing convention for the elements of the rough path $\X$, similarly to what is done at Section \ref{sec:strato-rough-intro}:

\smallskip

\noindent
\emph{(iii)} (Following (ii) at Section \ref{sec:strato-rough-intro}) We shall see that some overlapping integrals in directions 1 and 2 difficult the regularity analysis of certain increments. This will force us to some splitting operations on iterated integrals, leading to some split elements of the rough path $\X$. We indicate this splitting procedure by a $\otimes$ in indices of $\bx$. An example of this operation is given by the increment $\bx^{11;2\otimes2}\in\cac_{2}(\cac_{2}\otimes\cac_{2})$:
\begin{equation*}
\bx^{11;2\otimes2}_{s_1s_2;t_1t_2t_3t_4} 
= \int_{s_{1}}^{s_{2}} \int_{s_{1}}^{\si_{2}} 
\lp \int_{t_{1}}^{t_{2}} d_{12}x_{\si_{1};\tau_{1}} \rp 
\lp \int_{t_{3}}^{t_{4}} d_{12}x_{\si_{2};\tau_{2}} \rp .
\end{equation*}

\smallskip

\noindent
With this additional notation in mind, the complete description of $\X$ is given below:

\begin{hypothesis}\label{hyp:iter-intg-x}
The function $x$ is such that $\der x\in\cp_{2,2}^{\ga_{1},\ga_{2}}$ with $\ga_{1},\ga_{2}>1/3$, and fulfills Hypothesis~\ref{hyp:rough-path-intro}. In addition, the stack $\X$ of iterated integrals related to $x$ is required to contain the elements in the table below, where we let the reader guess the natural H\"older regularities related to each increment.
\begin{table}[htdp]
\caption{Further elements of $\X$}
\begin{center}
\begin{tabular}{|c|c||c|c|}
\hline
\emph{Increment} & \emph{Interpretation}  & \emph{Increment} & \emph{Interpretation}  \\
\hline
$\bx^{1\hone;2\htwo}$ & $\int_{1}\int_{2} d_{12}x d_{\hone\htwo}x$ &  
$\bx^{1\otimes1;22}$ & $\int_{1}\int_{2}d_{12}x\otimes_{1}\int_{1} d_{12}x$  \\
\hline
$\bx^{11\cdot1;022}$ & $\int_{1}d_{1}x\int_{2} d_{12}x\int_{1}d_{12}x$  & 
$\bx^{11\cdot1;222}$ & $\int_{1}\int_{2}d_{12}x d_{12}x\int_{1}d_{12}x$  \\
\hline
$\bx^{1\hone;0\htwo}$ & $\int_{1} d_{1}x\int_{2} d_{\hone\htwo}x$ & 
$\bx^{11\otimes1;222}$ & $\int_{1}\int_{2}d_{12}x d_{12}x\otimes_{1}\int_{1} d_{12}x$  \\
\hline
$\bx^{11\otimes 1;022}$ & $\int_{1} d_{1}x\int_{2} d_{12}x \otimes_{1} \int_{1} d_{12}x$ & 
$\bx^{11\cdot1;2\cdot22}$ & $\int_{1}\int_{2}d_{12}x \int_{2}d_{12}x\int_{1} d_{12}x$  \\
\hline
\end{tabular}
\end{center}
\label{tab:rough-path-2}
\end{table}%
As in Hypothesis~\ref{hyp:rough-path-intro}, the iterated integrals above are assumed to be limits along approximations of $x$ by smooth functions. Furthermore, the rough path $\X$ should also contain all the elements $\bx$ obtained by symmetrizing the increments of Table \ref{tab:rough-path-2} with respect to $1\leftrightarrow 2$, as well as those for which we change the last indices $1;2$ by $\hone;\htwo$. In total, we have to assume the existence of 26 additional increments.
\end{hypothesis}

\smallskip

With the additional notation introduced above, we are now ready to give the expression of $z^1$ in terms of the rough path $\X$.

\begin{proof}[Proof of Theorem \ref{thm:rough-strato-intro}]
Consider a smooth function $x$ and $y=\vp(x)$. According to our notational conventions of Section \ref{sec:iter-intg-plane}, we wish to express $\int_{1}\int_{2} y^{1} \,d_{12}x$ in terms of elementary increments of $y$, plus some elements of the rough path $\X$. As in the Young case, we thus start from expression \eqref{eq:dcp-delta-z-1}, but now the terms $a^{11;02}$, $a^{01;22}$ and $a^{11;22}$ need a further decomposition. We proceed to their analysis.

\smallskip

\smallskip

\noindent\textit{Step 1: First boundary integrals.} Consider the term $a^{11;02}$ in expression \eqref{eq:dcp-delta-z-1}. Two building blocks of our methodology can be observed on the analysis of this term in an elementary way:

\smallskip

\noindent
\emph{(i)} As noticed in the proof of Proposition \ref{prop:young-2d-y-dx}, our definition $a^{11;02}=\int_{1} d_1 y^1 \int_{2}  d_{12} x$ easily yields the identity 
\begin{equation}\label{eq:delta1-a1102}
\der_1 a^{11;02}= \der_1 y^1 \, \der x.
\end{equation}
Indeed, a possible way to understand this relation is to write:
\begin{equation*}
a^{11;02} = [d_1, \, d_1] \, [\id, \, \dt] (y^1, \, x)
\quad\Longrightarrow\quad
\doo a^{11;02} = [\doo, \, \doo] \, [\id, \, \dt] (y^1, \, x) = \doo y^1 \, \der x,
\end{equation*}
where we have invoked Proposition \ref{prop:dif-intg-1} for the first variable only
\smallskip

\noindent
\emph{(ii)} If we only consider the integral  $\int_{1} d_1 y^1$ within $a^{11;02}$, we can write
\begin{equation}\label{eq:int-d1-y}
\int_{1}d_1 y^1 = \int_{1} y^{1} \, d_1 x = y^{1} \der_1 x + \int_{1} d_{1} y^{2} \, d_1 x,
\end{equation}
where we recall that we have set $y^{1}=f'(x)$ and $y^2=f''(x)$ in Notation \ref{not:smooth-path-y-fx}.

\smallskip

\noindent
Plugging relation \eqref{eq:int-d1-y} into the definition of $a^{11;02}$ we obtain
\begin{equation}\label{eq:def-a-11-02}
a^{11;02} = y^{2} \int_{1} d_{1} x \int_{2} d_{12} x + \int_{1} d_{1} y^{2} \, d_1 x \int_{2} d_{12} x
= y^{2} \, \bx^{11;02} + \rho^{11;02},
\end{equation}
where $\bx^{11;02}$ has been defined at Hypothesis \ref{hyp:rough-path-intro} and where we have just defined the remainder $\rho^{11;02}$ by $\rho^{11;02}=\int_{1} d_{1} y^{2} \, d_1 x \int_{2} d_{12} x$. Now notice that $\rho^{11;02}$ is an increment which should belong to $\cp_{2,2}^{3\ga_1,\ga_2}$ if $x\in\cp_{1,1}^{\ga_1,\ga_2}$, and thus should be expressed in terms of $\laa_1$. In order to check this fact, write
\begin{equation*}
\rho^{11;02} = a^{11;02} - y^{2} \, \bx^{11;02},
\end{equation*}
and apply $\doo$ to both sides of this identity. Namely, the increment $\doo a^{11;02}$ is computed according to relation \eqref{eq:delta1-a1102}, while the term $y^{2} \, \bx^{11;02}$ is handled by means of identity \eqref{eq:difrul-C1-C2}:
\begin{equation*}
\doo\lp  y^{2} \, \bx^{11;02} \rp
= -\doo y^{2} \, \bx^{11;02} + y^{2} \, \doo\bx^{11;02}.
\end{equation*}
In addition, Hypothesis \ref{hyp:rough-path-intro} states that $\X$ is geometric, so that $\der_{1}\bx^{11;02}=\der_{1}x\,\der x$ like in the smooth case. We thus end up with:
\begin{equation*}
\doo\rho^{11;02} 
= \doo y^1 \, \der x + \doo y^{2} \, \bx^{11;02} - y^{2} \, \doo x \, \der x  .
\end{equation*}
Observe now, thanks to a simple Taylor expansion, that $\doo y^1 = y^{2} \doo x + r^{y,1}$, where $r^{y,1}$ is an element of $\cp_{2,1}^{2\ga_1,\ga_2}$. We thus get
\begin{equation*}
\doo\rho^{11;02}
= \lp y^{2} \doo x + r^{y,1} \rp \, \der x + \doo y^{1} \, \bx^{11;02} - y^{1} \, \doo x \, \der x
=r^{y,1} \, \der x + \doo y^{2} \, \bx^{11;02},
\end{equation*}
and the last increment is easily seen to be an element of $\cp_{2,2}^{3\ga_1,\ga_2}$, to which $\laa_1$ can be applied. Hence 
\begin{equation*}
\rho^{11;02} = \laa_1\lp r^{y,1} \, \der x + \doo y^{2} \, \bx^{11;02} \rp,
\end{equation*}
and plugging this identity into \eqref{eq:def-a-11-02}, we get
\begin{equation}\label{eq:exp-a-11-02-Lambda1}
a^{11;02} =  y^{2} \, \bx^{11;02} + \laa_1\lp r^{y,1} \, \der x + \doo y^{2} \, \bx^{11;02} \rp.
\end{equation}
Along the same lines, one can show that
\begin{equation}\label{eq:exp-a-01-22-Lambda2}
a^{01;22} =  y^{2} \, \bx^{01;22} + \laa_2\lp r^{y,2} \, \der x + \dt y^{2} \, \bx^{01;22} \rp.
\end{equation}

\smallskip

\noindent\textit{Step 2: First decomposition of the double integral.}
 Let us handle now the term $a^{11;22}$, starting by the equivalent of ingredient (ii) above. We thus write an expansion for the term $\int_{1}\int_{2}d_{12} y$, which can be handled thanks to relation \eqref{eq:ito-plane-smooth2} as
 \begin{equation*}
\int_{1}\int_{2} d_{12} y^1 = \int_{1}\int_{2} y^{2} \, d_{12} x + \int_{1}\int_{2} y^{3} \, d_{1} x \, d_{2} x,
\end{equation*}
and plugging this relation into the definition of $a^{11;22}$ we obtain
\begin{equation}\label{eq:first-dcp-double-intg}
a^{11;22} = 
\int_{1}\int_{2} y^{2} \, d_{12} x \, d_{12} x + \int_{1}\int_{2} y^{3} \, d_{1} x \, d_{2} x \, d_{12} x
\equiv b^{11;22} + b^{\hone 1;\htwo 2}.
\end{equation}
We proceed now to the analysis of the term $b^{11;22}$. To this aim, go back to relation \eqref{eq:dcp-delta-z-1} in order to write
\begin{eqnarray}\label{eq:dcp-b-11-22}
b^{11;22}&=&
\int_{1}\int_{2} \lp 
y^{2} \, \int_{1}\int_{2}  d_{12} x + \int_{1} d_1 y^{2} \int_{2}  d_{12} x + \int_{2} d_2 y^{2} \int_{1}  d_{12} x + \int_{1}\int_{2} d_{12}y^{2} \, d_{12} x  
\rp \, d_{12} x \notag \\
&=& y^{2} \, \bx^{11;22} + b^{111;022} + b^{011;222} + b^{111;222},
\end{eqnarray}
where
\begin{equation}\label{eq:def-b111022-b011222}
b^{111;022} = \int_{1} d_1 y^{2} \int_{2}  d_{12} x \, d_{12} x, \quad
b^{011;222} = \int_{2} d_2 y^{2} \int_{1}  d_{12} x \, d_{12} x,
\end{equation}
and $b^{111;222} = \int_{1}\int_{2} d_{12}y^{2} \, d_{12} x \, d_{12} x$. We shall treat those 3 terms separately.

\smallskip

\noindent\textit{Step 3: Analysis of the boundary integrals.}
The term $b^{111;022}$ is a third order iterated integral in direction 1. One can thus try to apply ingredient (i) above, namely compute $\doo b^{111;022}$ in order to observe if we get an increment in $\cp_{3,2}^{3\ga_1,*}$.  Unfortunately, this brings some extra complications due to overlapping integrations. This is why we have postponed these computations to Lemma \ref{lem:b-111-022} below, which asserts that $\doo b^{111;022}\in\cp_{3,2}^{3\ga_1,\ga_2}$ and is a continuous function of the couple $(\X,y)$. We have thus obtained that $b^{111;022}=\laa_{1}(\rho^{y^{2},1})$ for a certain $\rho^{y^{2},1}\in\cp_{3,2}^{3\ga_1,\ga_2}$. Exactly along the same lines, we also get that $b^{011;222}=\laa_{2}(\rho^{y^{2},2})$.

\smallskip

Going back to expression \eqref{eq:dcp-b-11-22}, we see that we are now only left with the definition of the triple iterated integral $b^{111;222}$.

\smallskip

\noindent\textit{Step 4: Analysis of the triple iterated integral.}
Recall that $b^{111;222} = \int_{1}\int_{2} d_{12}y^{2} \, d_{12} x \, d_{12} x$. It is thus natural to analyze this increment thanks to an application of the operator $\der$. The reader might verify that, applying successively $\doo$ and $\dt$ plus Proposition \ref{prop:dif-intg-1} in each direction, one gets:
\begin{equation*}
\der b^{111;222} = \der y^{2} \, \bx^{11;22} + b^{1\cdot11;22\cdot2}+b^{11\cdot1;2\cdot22}
+ b^{11\cdot1;22\cdot2},
\end{equation*}
where
\begin{equation}\label{eq:def-b-11.1-2.22}
b^{1\cdot11;22\cdot2} = \int_{1}\int_{2} d_{12}y^{2} \int_{1} d_{12}x \int_{2} d_{12}x, \quad\text{ }\quad
b^{11\cdot1;2\cdot22} = \int_{1}\int_{2} d_{12}y^{2} \int_{2} d_{12}x \int_{1} d_{12}x,
\end{equation}
and $b^{11\cdot1;22\cdot2}=( \int_{1}\int_{2} d_{12}y^{2} d_{12}x ) \, \der x$. 
Now we resort to Lemma \ref{lem:b-11.1-2.22} to assert that $b^{1\cdot11;22\cdot2}$ and $b^{11\cdot1;2\cdot22}$ lye into $\cp_{3,3}^{3\ga_{1},3\ga_{2}}$. We shall see at the end of the proof that  $b^{11\cdot1;22\cdot2}$ can also be included in a $\cp_{3,3}^{3\ga_{1},3\ga_{2}}$ term. 

\smallskip

Going back to relation \eqref{eq:dcp-b-11-22} and summarizing our computations from Steps 3 and 4, we have thus obtained that
\begin{equation}\label{eq:dissec-b-11-22}
b^{11;22} = y^{2} \, \bx^{11;22} + \laa_{1}(\rho^{y^{2},1}) + \laa_{2}(\rho^{y^{2},2}) + b^{111;222},
\end{equation}
where $\rho^{y^{2},1}\in\cp_{3,2}^{3\ga_{1},\ga_{2}}$, $\rho^{y^{2},2}\in\cp_{2,3}^{\ga_{1},3\ga_{2}}$ and where $b^{111;222}$ is dissected up to the term $b^{11\cdot1;22\cdot2}$.

\smallskip

\noindent\textit{Step 5: Analysis of $b^{\hone 1;\htwo 2}$.}
We now go back to the term $b^{\hone 1;\htwo 2}$ introduced at equation \eqref{eq:first-dcp-double-intg}. Its analysis is very similar to the one we developed for $b^{11;22}$, so that we only sketch the main differences.

\smallskip

First of all, in order to get an equivalent to relation \eqref{eq:dcp-delta-z-1}, we resort to the differential element $d_{\hone\htwo}x$. This yields:
\begin{equation}\label{eq:dcp-int-y2-d1x-d2x}
\int_{1}\int_{2} y^{3} \, d_{\hone\htwo}x 
=
y^{2} \, \int_{1}\int_{2}  d_{\hone\htwo}x 
+ \int_{1} d_1 y^{2}  \int_{2}  d_{\hone\htwo}x + \int_{2} d_2 y^{3}  \int_{1}  d_{\hone\htwo}x 
+ \int_{1}\int_{2} d_{12}y^{3} \, d_{\hone\htwo}x,
\end{equation}
and we plug this relation into \eqref{eq:first-dcp-double-intg} in order to get
\begin{equation*}
b^{\hone 1;\htwo 2}
= y^{3} \, \bx^{\hone 1;\htwo 2} + b^{1\hone1;0\htwo2} + b^{0\hone1;2\htwo2} 
+ b^{1\hone1;2\htwo2}
\end{equation*}
where
\begin{equation*}
b^{1\hone1;0\htwo2} = \int_{1} d_1 y^{3}  \int_{2}  d_{\hone\htwo} x \, d_{12} x, \qquad
b^{0\hone1;2\htwo2} = \int_{2} d_2 y^{3}  \int_{1}  d_{\hone\htwo} x \, d_{12} x, 
\end{equation*}
and $b^{1\hone1;2\htwo2} = \int_{1}\int_{2} d_{12}y^{3} \, d_{\hone\htwo} x  \, d_{12} x$. 

\smallskip

The 2 terms $b^{1\hone1;0\htwo2}$ and $b^{0\hone1;2\htwo2}$ can now be respectively handled similarly to $b^{111;022}$ and $b^{011;222}$. We obtain that 
\begin{equation*}
b^{1\hone1;0\htwo2} = \laa_{1}(\rho^{y^{3},1}), 
\quad\text{and}\quad
b^{0\hone1;2\htwo2} = \laa_{2}(\rho^{y^{3},2}).
\end{equation*}
As far as $b^{1\hone1;2\htwo2}$ is concerned, we dissect it through the application of $\der$, which yields
\begin{equation*}
\der b^{1\hone1;2\htwo2} = \der y^{3} \, \bx^{1\hone;2\htwo} 
+ b^{1\cdot\hone1;2\htwo\cdot2}+b^{1\hone\cdot1;2\cdot\htwo2}
+ b^{1\hone\cdot1;2\htwo\cdot2},
\end{equation*}
where $b^{1\cdot\hone1;2\htwo\cdot2}$, $b^{1\hone\cdot1;2\cdot\htwo2}$ and
$b^{1\hone\cdot1;2\htwo\cdot2}$ are defined as in \eqref{eq:def-b-11.1-2.22}. In addition, as in Step 4, it can be proven that $\der y^{3} \, \bx^{1\hone;2\htwo} 
+ b^{1\cdot\hone1;2\htwo\cdot2}+b^{1\hone\cdot1;2\cdot\htwo2}\in\cp_{3,3}^{3\ga_1,3\ga_2}$.
Similarly to \eqref{eq:dissec-b-11-22}, we thus end up with a relation of the form
\begin{equation}\label{eq:dissec-b-hat11-hat22}
b^{\hone1;\htwo2} = y^{2} \, \bx^{1\hone;2\htwo} + \laa_{1}(\rho^{y^{3},1}) + \laa_{2}(\rho^{y^{3},2}) + b^{1\hone1;2\htwo2},
\end{equation}
where $\rho^{y^{2},1}\in\cp_{3,2}^{3\ga_{1},\ga_{2}}$, $\rho^{y^{2},2}\in\cp_{2,3}^{\ga_{1},3\ga_{2}}$ and where $b^{1\hone1;2\htwo2}$ is conveniently decomposed up to the term $b^{1\hone\cdot1;2\htwo\cdot2}$.

\smallskip

\noindent\textit{Step 6: Analysis of $b^{11\cdot1;22\cdot2}$ and $b^{1\hone\cdot1;2\htwo\cdot2}$.}
Write $b^{11\cdot1;22\cdot2}+b^{1\hone\cdot1;2\htwo\cdot2}=\ell \, \der x$, with 
\begin{equation*}
\ell = \int_{1}\int_{2} d_{12}y^{2} \, d_{12}x + \int_{1}\int_{2} d_{12}y^{3} \, d_{\hone\htwo}x.
\end{equation*}
We wish to show that $b^{11\cdot1;22\cdot2}+b^{1\hone\cdot1;2\htwo\cdot2}\in\cp_{3,3}^{3\ga_{1},3\ga_{2}}$, which amounts to check that $\ell$ lies into $\cp_{2,2}^{2\ga_{1},2\ga_{2}}$.
In order to prove this claim, start again from a change of variable formula for $y=\vp(x)$ plus elementary manipulations like those of relation \eqref{eq:dcp-delta-z-1}. This yields
\begin{eqnarray*}
\ell &=& \der y^{1} - \int_{1}\int_{2} y^{2} \, d_{12}x - \int_{1}\int_{2} y^{3} \, d_{\hone\htwo}x \\
&=& \der y^{1} - \lp y^{2} \, \bx^{1;2} + a^{11;02}(y^{2}) + a^{01;22}(y^{2}) 
+ a^{1\hone;0\htwo}(y^{3}) + a^{0\hone;2\htwo}(y^{3}) \rp,
\end{eqnarray*}
where $a^{11;02}(y^{2}),\ldots,a^{0\hone;2\htwo}(y^{3})$ are defined similarly to \eqref{eq:def-a-11-02}, with either $y^{2}$ or $y^{3}$ instead of $y^{1}$. With the same calculations as for \eqref{eq:exp-a-11-02-Lambda1}--\eqref{eq:exp-a-01-22-Lambda2}, we thus get
\begin{eqnarray*}
a^{11;02}(y^{2}) &=& 
y^{2} \, \bx^{11;02} + \laa_1\lp r^{y^{2},1} \, \der x + \doo y^{3} \, \bx^{11;02} \rp \\
a^{01;22}(y^{2}) &=& y^{3} \, \bx^{01;22} + \laa_2\lp r^{y^{2},2} \, \der x + \dt y^{3} \, \bx^{01;22} \rp,
\end{eqnarray*}
and
\begin{eqnarray*}
a^{1\hone;0\htwo}(y^{3}) &=& 
y^{4} \, \bx^{1\hone;0\htwo} + \laa_1\lp r^{y^{3},1} \, \bx^{\hone;\htwo} 
+ \doo y^{4} \, \bx^{1\hone;0\htwo} \rp \\
a^{0\hone;2\htwo}(y^{3}) &=& y^{4} \, \bx^{0\hone;2\htwo} 
+ \laa_2\lp r^{y^{3},2} \, \bx^{\hone;\htwo} + \dt y^{4} \, \bx^{0\hone;2\htwo} \rp.
\end{eqnarray*}
These identities easily yield the following claims:

\smallskip

\noindent
\emph{(a)}
The increment $\ell$ is a continuous function of $\bx^{1;1}, \bx^{11;02}, \bx^{01;22}, \bx^{\hone;\htwo}, \bx^{1\hone;0\htwo}, \bx^{0\hone;2\htwo}$ and $y^1$.

\smallskip

\noindent
\emph{(b)}
The fact that $\ell\in\cp_{2,2}^{2\ga_{1},2\ga_{2}}$ is proven thanks to long and tedious Taylor type expansions. We refer to the stability result \cite[Theorem 7.13]{Gp} for further details. 

\smallskip

\noindent
We now plug this information into \eqref{eq:dissec-b-11-22} and \eqref{eq:dissec-b-hat11-hat22} in order to write
\begin{equation*}
b^{11;22} + b^{\hone1;\htwo2}
= 
y^{2} \, \bx^{11;22} + y^{2} \, \bx^{1\hone;2\htwo}
+ \laa_{1}(\rho^{y^{2},1}+\rho^{y^{3},1}) + \laa_{2}(\rho^{y^{2},2}+\rho^{y^{3},2}) 
+\laa(\ell\,\der x).
\end{equation*}
Finally, we propagate this identity back into \eqref{eq:first-dcp-double-intg}, \eqref{eq:exp-a-01-22-Lambda2}, \eqref{eq:exp-a-11-02-Lambda1} and eventually \eqref{eq:dcp-delta-z-1}, which yields our identity \eqref{eq:strato-der-z1-rough}. We leave the details of these last computations to the patient reader.

\smallskip

\noindent\textit{Step 7: Conclusion.}
We have just obtained expression \eqref{eq:strato-der-z1-rough} for $z^{1}$. 
As far as $z^{2}$ is concerned, we start by replacing relation \eqref{eq:dcp-delta-z-1} by:
\begin{equation*}
\int_{1}\int_{2} y^{2} \, d_{\hone\htwo}x 
=
y^{2} \, \int_{1}\int_{2}  d_{\hone\htwo}x 
+ \int_{1} d_1 y^{2}  \int_{2}  d_{\hone\htwo}x + \int_{2} d_2 y^{2}  \int_{1}  d_{\hone\htwo}x 
+ \int_{1}\int_{2} d_{12}y^{2} \, d_{\hone\htwo}x.
\end{equation*}
The expansion then proceeds exactly like in the expansion of $z^{1}$, and details are omitted for sake of conciseness. The claims (ii) and (iii)  of Theorem \ref{thm:rough-strato-intro} follow by a limiting procedure.

\end{proof}

\begin{lemma}\label{lem:b-111-022}
Let $b^{111;022}$ and $b^{011;222}$  be the terms defined by \eqref{eq:def-b111022-b011222}. Then: 

\smallskip

\noindent
\emph{(i)} The increment $\doo b^{111;022}$ is contained in  $\cp_{3,2}^{3\ga_{1},\ga_{2}}$ and is a continuous function of the following elements of $\X$: $\bx^{11;22}, \bx^{11\cdot1;022}, \bx^{1\otimes1;22}, \bx^{11\otimes1;022}$ and $y$.

\smallskip

\noindent
\emph{(ii)} The increment $\dt b^{011;222}$ is contained in$\cp_{2,3}^{\ga_{1},3\ga_{2}}$ and is a continuous function of the following elements of $\X$: $\bx^{11;22}, \bx^{011;22\cdot2}, \bx^{11;2\otimes2}, \bx^{011;02\otimes2}$ and $y$.
\end{lemma}

\begin{proof}
Let us prove the result for $\doo b^{111;022}$, since $\dt b^{011;222}$ is analyzed exactly in the same way. 

\smallskip

Towards this aim, invoking Proposition \ref{prop:dif-intg-1} in direction 1 only we get
\begin{equation}\label{eq:delta1-b-111-022}
\doo b^{111;022} = \doo y^{2} \, \bx^{11;22} +  h^{11\cdot1;022},
\quad\text{with}\quad
h^{11\cdot1;022} = \int_{1} d_1 y^{2} \int_{2}  d_{12} x  \int_{1} d_{12} x.
\end{equation}
Now, under Hypothesis \ref{hyp:iter-intg-x} the term $\doo y^{2} \, \bx^{11;22}$ above is easily seen to lye into $\cp_{3,2}^{3\ga_1,2\ga_2}$. However, the term $h$ is more troublesome. In order to see this additional problem more clearly, let us specify $h$ a little: for  $s_1,s_2,s_3\in\cs_3$ and $t_1,t_2\in\cs_2$ it can be written as
\begin{eqnarray}\label{eq:increment-h-111-022}
h_{s_1 s_2 s_3;t_1 t_2}^{11\cdot1;022}&=&
\int_{t_1}^{t_2} \lp 
\int_{t_1}^{\tau_2}\int_{s_1}^{s_2} \lp \int_{s_1}^{\si_2} d_{1} y^{2}_{\si_1;t_1} \rp 
d_{12} x_{\si_2;\tau_1} \int_{s_2}^{s_3} d_{12} x_{\si_3;\tau_2}
\rp  \notag\\
&=& 
\int_{t_1}^{t_2} \lp \int_{s_1}^{s_2} \doo y^{2}_{s_1\si_2;t_1}  d_{1} \dt x_{\si_2;t_1\tau_2}  \rp 
d_{2} \doo x_{s_2s_3;\tau_2},
\end{eqnarray}
and the regularity of this last increment is far from being obvious to determinate. Furthermore, a closer look at $h$ leads to the following conclusion: the problem in the regularity analysis stems from the overlap in integrations for directions 1 and 2, apparent in the definition of $h$. In order to fix this new problem (which can obviously only occur when integrals of order 3 or more are showing up), we introduce a third ingredient of our methodology, based on the splitting operation of Definition \ref{def:splitting-plane}:

\smallskip

\noindent
\emph{(iii)} The splitting operation can be applied to overlapping terms like $h^{11\cdot1;022}$, in order to separate the term $\int_{1} d_1 y^{2} \int_{2}  d_{12} x$ from the term  $\int_{1} d_{12} x$. This simple trick allows to solve the intricate structure of integration in both directions 1 and 2. Unfortunately, the remaining terms have to be defined separately now, and are not of order 3 anymore. This means that another step of expansion might be necessary for their proper definition. In our running example, this additional step concerns the term $\int_{1} d_1 y^{2} \int_{2}  d_{12} x$ in \eqref{eq:delta1-b-111-022}.

\smallskip

Ingredient (iii) applied to our increment $h^{11\cdot1;022}$ yields the following developments: we have
\begin{equation}\label{eq:split-h}
S_{1}(h^{11\cdot1;022}) = 
\int_{1} d_1 y^{2} \int_{2}  d_{12} x \otimes_{1}  \int_{1} d_{12} x, 
\end{equation}
and let us add the following comments: 

\smallskip

\noindent
(a) In direction 2, the quantity $S_{1}(h^{11\cdot1;022})$ can be simply expressed in terms of iterated integrals of the driving process $x$ and elementary increments of $y^1$.

\smallskip

\noindent
(b) The exact expression of $S_{1}(h^{11\cdot1;022})$ can be obtained from \eqref{eq:increment-h-111-022} by splitting the variables $s_2,s_3$ into 3 variables $s_2,s_3,s_4$:
\begin{equation*}
\lc S_{1}(h^{11\cdot1;022}) \rc_{s_1 s_2 s_3 s_4;t_1 t_2}
=\int_{t_1}^{t_2} \lp \int_{s_1}^{s_2} \doo y^{2}_{s_1\si_2;t_1}  d_{1} \dt x_{\si_2;t_1\tau_2}  \rp 
d_{2} \doo x_{s_3s_4;\tau_2}.
\end{equation*}

\smallskip

\noindent
We now set $h^{11;02}=\int_{1} d_1 y^{2} \int_{2}  d_{12} x$, which is the increment we shall further expand in~\eqref{eq:split-h}. This can be performed exactly as for $a^{11;02}$ in Step 2, just replacing $y$ by $y^{2}$, and enables the following expression which should be compared to equality \eqref{eq:exp-a-11-02-Lambda1}:
\begin{equation*}
h^{11;02} =  y^{3} \, \bx^{11;02} + \laa_1\lp r^{y^{2},1} \, \der x + \doo y^{3} \, \bx^{11;02} \rp,
\end{equation*}
where we recall that $\bx^{11;02}=\int_{1} d_{1}x\int_{2}d_{12}x$.
We plug this relation into \eqref{eq:split-h} and we obtain
\begin{equation}\label{eq:def-S1-h-11.1-022}
S_{1}(h^{11\cdot1;022})=
y^{3} \, \int_{1} d_{1}x\int_{2}d_{12}x \otimes_{1} \int_{1}d_{12}x + h^{11\cdot1;022;\sharp}.
\end{equation}
In the last expression, $h^{11\cdot1;022;\sharp}$ is defined as follows: we have
\begin{eqnarray}\label{eq:exp-eta-11.1-022}
\eta^{11\cdot1;022;\sharp} &\equiv&
\lc  \der_1\otimes_{1}\id\rc h^{11\cdot1;022;\sharp} \notag\\
&=& r^{y^{2},1} \, \int_{1}\int_{2} d_{12}x \otimes_{1} \int_{1} d_{12}x 
+ \doo y^{3} \, \int_{1} d_{1}x\int_{2}d_{12}x \otimes_{1} \int_{1}d_{12}x\notag\\
&=& r^{y^{2},1} \, \bx^{1\otimes1;22} 
+ \doo y^{3} \, \bx^{11\otimes1;022} ,
\end{eqnarray}
and notice that $\eta^{11\cdot1;022;\sharp}$ is only expressed in terms of elementary increments based on $y$ and iterated integrals of $x$. Furthermore, Hypothesis \ref{hyp:iter-intg-x} guarantees that $\eta^{11\cdot1;022;\sharp}$ lies into $[\cac_{2}^{3\ga_1}\otimes\cac_{2}^{\ga_1}](\cac_{2}^{\ga_2})$, and thus $h^{11\cdot1;022;\sharp}=[\laa_{1} \otimes_{1} \id](\eta^{11\cdot1;022;\sharp})$. Going back to \eqref{eq:def-S1-h-11.1-022}, we end up with
\begin{equation*}
S_{1}(h^{11\cdot1;022})= y^{3} \, \int_{1} d_{1}x\int_{2}d_{12}x \otimes_{1} \int_{1}d_{12}x 
+ \lc\laa_{1} \otimes_{1} \id\rc (\eta^{11\cdot1;022;\sharp}).
\end{equation*}
We can now go back to an expression for $h^{11\cdot1;022}$ itself by applying the inverse map $G_{1}$ of $S_{1}$. This gives
\begin{eqnarray*}
h^{11\cdot1;022} &=& y^{3} \, \int_{1} d_{1}x\int_{2}d_{12}x \int_{1}d_{12}x 
+ G_{1}\lp \lc\laa_{1} \otimes_{1} \id\rc (\eta^{11\cdot1;022;\sharp}) \rp  \\
&=& y^{3} \, \bx^{11\cdot1;022} + G_{1}\lp \lc\laa_{1} \otimes_{1} \id\rc (\eta^{11\cdot1;022;\sharp}) \rp.
\end{eqnarray*}
Plugging this relation back in equation \eqref{eq:delta1-b-111-022}, we end up with
\begin{equation}\label{eq:b-111-022-order-3}
\doo b^{111;022} = \doo y^{2} \, \bx^{11;22} + y^{2} \, \bx^{11\cdot 1;022} + G_{1}\lp \lc\laa_{1} \otimes_{1} \id\rc (\eta^{11\cdot1;022;\sharp}) \rp,
\end{equation}
which is now clearly an element of $\cp_{3,2}^{3\ga_1,\ga_2}$ and a continuous function of the increments $\bx^{11;22}, \bx^{11\cdot1;022}, \bx^{1\otimes1;22}, \bx^{11\otimes1;022}$ and $y$ as claimed in our proposition. 

In the same spirit, we also have 
\begin{equation*}
b^{011;222} = \laa_{2}\lc \dt y^{2} \, \bx^{11;22} + y^{3} \, \bx^{011;22\cdot2} 
+ G_{2}\lp \lc\laa_{2} \otimes_{2} \id\rc (\eta^{011;22\cdot2;\sharp}) \rp \rc,
\end{equation*}
where $\eta^{011;22\cdot2;\sharp}$ is defined similarly to $\eta^{11\cdot1;022;\sharp}$ in \eqref{eq:exp-eta-11.1-022}.

\end{proof}

\begin{lemma}\label{lem:b-11.1-2.22}
Let $b^{1\cdot11;22\cdot2}$ and $b^{11\cdot1;2\cdot22}$  be the terms defined by \eqref{eq:def-b-11.1-2.22}. Then: 

\smallskip

\noindent
\emph{(i)} The increment $\doo b^{1\cdot11;22\cdot2}$ is an element of $\cp_{3,3}^{3\ga_{1},3\ga_{2}}$ and is a continuous function of $\bx^{11\cdot1;022}$, $\bx^{11\cdot1;2\cdot22}$, $\bx^{1\otimes1;22}$, $\bx^{11\otimes1;022}$, $\bx^{11\otimes1;222}$ and $y$.

\smallskip

\noindent
\emph{(ii)} The increment $\dt b^{11\cdot1;2\cdot22}$ is an element of $\cp_{3,3}^{3\ga_{1},3\ga_{2}}$ and is a continuous function of $\bx^{011;22\cdot2}$, $\bx^{1\cdot11;22\cdot2}$, $\bx^{11;2\otimes2}$, $\bx^{011;22\otimes2}$, $\bx^{111;22\otimes2}$ and $y$.
\end{lemma}

\begin{proof}

Those two terms contain overlapping integrations in direction 1 and 2 again, and we thus proceed to their dissection through ingredient~(iii). In the case of $b^{11\cdot1;2\cdot22}$ we will use an additional expansion in direction 1 only, and thus write
\begin{equation}\label{eq:splitting-b-111-222}
S_1\lp b^{11\cdot1;2\cdot22} \rp = \int_{1}\int_{2} d_{12}y^{2} \int_{2} d_{12}x \otimes_{1} \int_{1} d_{12}x,
\quad\text{and set}\quad
c^{11;2\cdot2}= \int_{1}\int_{2} d_{12}y^{2} \int_{2} d_{12}x.
\end{equation}
In order to expand $c^{11;2\cdot2}$ one step further in direction 1, observe that
\begin{equation*}
\int_{1}\int_{2} d_{12}y^{2} = \int_{1} \dt d_{1}y^{2} = \int_{1} \dt \lp y^{3} d_{1}x \rp
= \int_{1} \dt  y^{3} \, d_{1}x + \int_{1}   y^{3} \, \dt d_{1}x.
\end{equation*}
Plugging this identity into the definition of $c^{11;2\cdot2}$ we get a decomposition as
\begin{equation*}
c^{11;2\cdot2}= k^{1} + k^{2},
\quad\text{where}\quad
k^{1}= \int_{1} \dt  y^{3} \, d_{1}x \int_{2} d_{12}x, \quad
k^{2}= \int_{1} y^{3} \, \dt d_{1}x \int_{2} d_{12}x .
\end{equation*}
We now apply ingredients (i) and (ii) to analyze those terms. First, one can compute $\doo c^{11;2\cdot2}$ as
\begin{equation}\label{eq:delta-1-c-11-22}
\doo c^{11;2\cdot2} = \int_{2} \doo d_{2}y^{2} \int_{2} \doo d_{2}x
= \der y^{2} \, \der x.
\end{equation}
Next expand $y^{3}$ in direction 1 in both $k^{1}$ and $k^{2}$, which yields
\begin{equation*}
k^{1}= \dt y^{3} \, \bx^{11;02} + \rho^{1},
\quad\text{and}\quad
k^{2}= y^{3} \, \bx^{11;2\cdot2} + \rho^{2},
\end{equation*}
where $\rho^{1}$ and $\rho^{2}$ are two remainder terms morally lying in $\cp_{2,2}^{3\ga_1,2\ga_2}$. Setting $\rho^{c}=\rho^{1}+\rho^{2}$, we have thus obtained the following relation:
\begin{equation*}
c^{11;2\cdot2} = \dt y^{3} \, \bx^{11;02} + y^{3} \, \bx^{11;2\cdot2} + \rho^{c}
\quad\Longrightarrow\quad
\rho^{c} = c^{11;2\cdot2} - \dt y^{3} \, \bx^{11;02} - y^{3} \, \bx^{11;2\cdot2} 
\end{equation*}
Applying $\doo$ to both sides of the last identity plus relation \eqref{eq:delta-1-c-11-22} and Hypothesis \ref{hyp:iter-intg-x}, we end up with:
\begin{equation*}
\doo \rho^{c} =
\lp  \der y^{2} - \dt y^{3} \, \doo x - y^{3} \der x \rp \der x
+\der y^{3} \, \bx^{11;02} + \doo y^{3} \, \bx^{11;2\cdot2}.
\end{equation*}
Furthermore, an elementary Taylor type computation shows that $(\der y^{1} - \dt y^{3} \, \doo x - y^{3} \der x)\in \cp_{2,2}^{2\ga_1,*}$ and thus it is readily checked that $\doo \rho^{c}\in\cp_{3,2}^{3\ga_1,\ga_2}$.
We now plug the decomposition of  $\doo \rho^{c}$ back into the relation \eqref{eq:splitting-b-111-222} defining  $S_1(b^{11\cdot1;2\cdot22})$ in order to obtain
\begin{multline*}
S_1\lp b^{11\cdot1;2\cdot22} \rp 
= \dt y^{3} \, \int_{1} d_{1} x \int_{2} d_{12}x  \otimes_{1} \int_{1} d_{12}x
+ y^{3} \, \int_{1} \int_{2} d_{12} x \int_{2} d_{12}x  \otimes_{1} \int_{1} d_{12}x  \\
+ h^{11\cdot1;2\cdot22;\sharp},
\end{multline*}
where $h^{11\cdot1;2\cdot22;\sharp}$ is defined in the following way: setting $\eta^{11\cdot1;2\cdot22;\sharp}\equiv[\doo\otimes_{1}\id]h^{11\cdot1;2\cdot22;\sharp}$, we have 
\begin{equation*}
\eta^{11\cdot1;2\cdot22;\sharp} = 
\lp  \der y^{3} - \dt y^{3} \, \doo x - y^{3} \der x \rp \, \bx^{1\otimes1;22}  
+\der y^{3} \, \bx^{11\otimes1;022}  
+\doo y^{3} \, \bx^{11\otimes1;222}.
\end{equation*}
As in the proof of Lemma \ref{lem:b-111-022}, notice that $\eta^{11\cdot1;2\cdot22;\sharp}$ is only expressed in terms of elementary increments based on $y$ and $\X$, and belongs to the domain of  $\laa_1\otimes_{1}\id$. Hence, applying the gluing operator $G_1$ we end up with the relation
\begin{equation*}
b^{11\cdot1;2\cdot22} = 
\dt y^{3} \, \bx^{11\cdot1;022} + y^{3} \, \bx^{11\cdot1;2\cdot22} 
+ G_{1}\lp \lc  \laa_1\otimes_{1}\id\rc \eta^{11\cdot1;2\cdot22;\sharp}\rp.
\end{equation*}
In the same manner, we obtain
\begin{equation*}
b^{1\cdot11;22\cdot2} = 
\doo y^{3} \, \bx^{011;22\cdot2} + y^{3} \, \bx^{1\cdot11;22\cdot2} 
+ G_{2}\lp \lc  \laa_2\otimes_{2}\id\rc \eta^{1\cdot11;22\cdot2;\sharp}\rp,
\end{equation*}
where
\begin{equation*}
\eta^{11\cdot1;2\cdot22;\sharp} = 
\lp  \der y^{2} - \doo y^{3} \, \dt x - y^{3} \der x \rp \, \bx^{11;2\times2}
+\der y^{3} \, \bx^{22\otimes2;011} 
+\dt y^{3} \, \bx^{111;22\otimes2},
\end{equation*}
which ends the proof.

\end{proof}

\section{Skorohod's calculus in the Young case}
\label{sec:young-vs-malliavin}
This section is devoted to relate the Young type integration theory introduced at Section~\ref{sec:planar-young} and the Skorohod integral in the plane handled in~\cite{TVp1}. Specifically, we shall first generalize the Skorohod change of variables formula given in~\cite{TVp1} for a fractional Brownian sheet with Hurst parameter greater than $1/2$ to a fairly general Gaussian process. We shall then compare this formula with Theorem \ref{thm:young-strato-intro} item (v).

\smallskip

\smallskip

Before going on with our computations, let us label some notations for further use:

\begin{notation}\label{not:lesssim-partitions}
We write $X\lesssim_{a,b,...}Y$ if there exist a constant $c$ depending on $a,b,...$ such that the quantities $X,Y$ satisfy
$X\leq cY$. For a partition $\{(s_i,t_j)_{i,j}\}$ of a rectangle  $\Delta=[s_1,s_2]\times[t_1,t_2]$, $\Delta_{ij}$ denotes the rectangle $[s_i,s_{i+1}]\times[t_j,t_{j+1}]$.
\end{notation}

\subsection{Malliavin calculus framework}
\label{sec:malliavin-framework}
We consider in this section a centered Gaussian process  $\{x_{s;t} ; \, (s,t)\in[0,1]^2\}$  defined on a complete probability space $(\Omega,\cf,\mathbb{P})$, with covariance function  $\E[x_{s_1;t_1}x_{s_2;t_2}]=R_{s_{1}s_{2};t_{1}t_{2}}$. We now briefly define the basic elements of Malliavin calculus with respect to $x$ and then specify a little the setting under which we shall work.

\subsubsection{Malliavin calculus with respect to $x$}
\label{sec:mall-calc-x}
We first relate a Hilbert space $\mathcal H$ to our process $x$, defined as the closure of the linear space generated by the functions $\{\1_{[0,s]\times[0,t]} ,(s,t)\in[0,1]^2\}$ with respect to the semi define positive form $\langle \1_{[0,s_{1}]\times[0,t_{1}]},\1_{[0,s_{2}]\times[0,t_{2}]}\rangle=R_{s_{1}s_{2};t_{1}t_{2}}$. Then the map $I_1:\1_{[0,s]\times[0,t]}\to x_{s;t}$ can be extended to an isometry between $\mathcal H$ and the first chaos  generated by $\{x_{s;t} ;(s,t)\in[0,1]^2\}$. 

\smallskip

Starting from the space $\ch$, a  Malliavin calculus with respect to $x$ can now be developped in the usual way (see \cite{HJT,Nu06} for further details). Namely, we first define a set of smooth functionals of $x$ by 
\begin{equation*}
\mathcal S:=\left\{f(I_1(\psi_1),\ldots,I_1(\psi_n)) ; 
\, n\in\mathbb N,f\in C^{\infty}_b(\mathbb R^n),\psi_1,\ldots,\psi_n\in\mathcal{H} \right\}
\end{equation*}
and for $F=f(I_1(\psi_1),\ldots,I_1(\psi_n))\in\cs$ we define 
$$
DF=\sum_{i=1}^{n}\partial_if(I_1(\psi_1),\ldots,I_1(\psi_n)) \, \psi_i.
$$
Then $D$ is a closable operator from $L^p(\Omega)$ into $L^p(\Omega,\mathcal H)$. Therefore we can extend $D$ to the closure of smooth functionals under the norm 
$$
\|F\|_{1,p}=(\E[|F|^p]+\E[\|DF\|^p_{\mathcal H}])^{\frac{1}{p}}
$$
The iteration of the operator $D$ is defined in such a way that for a smooth random
variable $F\in\mathcal S$ the iterated derivative $D^{k}F$ is a random variable with values in $\mathcal H^{\otimes k}$. The domain $\mathbb D^{k,p}$ of $D^{k}$ is the completion of the family of smooth random variables $F\in\mathcal S$ with respect to the semi-norm :
$$
\|F\|_{k,p}=\lp \E[|F|^p]+\sum_{j=1}^{k}\E[\|D^{j}F\|^p_{\mathcal H^{\otimes j}}]\rp^{\frac{1}{p}}.
$$
Similarly, for a given Hilbert space $V$ we can define the space $\mathbb D^{k,p}(V)$ of $V$-valued random variables, and $\D^{\infty}(V)=\cap_{k,p\ge 1}\D^{k,p}$. 

\smallskip

Consider now the adjoint $\der^{\diamond}$ of $D$. The domain of this operator is defined as the set of $u\in L^2(\Omega,\mathcal H)$ such that 
$\E[|\langle DF,u\rangle_{\mathcal H}|]\lesssim \|F\|_{1,2}$,
 and for this kind of process $\der^{\diamond}(u)$ (called Skorohod integral of $u$) is the unique element of $L^2(\Omega)$ such that
 $$
\E[\der^{\diamond}(u)F]=\E[\langle DF,u\rangle_{\mathcal H}],
\quad\text{for}\quad
F\in\D^{1,2}.
$$
Note that $\E[\der^{\diamond}(u)]=0$ and  
$$
\E[|\der^{\diamond}(u)|^2]\leq\E[\|u\|^2_{\mathcal H}]+\E[\|Du\|^2_{\mathcal H\otimes \mathcal H}].
$$
The following divergence type property of $\der^{\diamond}$ will be useful in the sequel:
\begin{equation}\label{ipp div}
\der^{\diamond}(Fu)=F\der^{\diamond}(u)-\langle Du,F\rangle_{\mathcal H},
\end{equation}
and we also recall the following compatibility of $\der^{\diamond}$ with limiting procedures:
\begin{lemma}
\label{clos}
let $u_{n}$ be a sequence of elements in $Dom(\der^{\diamond})$, which converges to $u$ in $L^{2}(\Omega,\mathcal H)$. We further assume that $\der^{\diamond}(u_{n})$ converges in $L^{2}(\Omega)$ to some random variable $F\in L^{2}(\Omega)$. Then $u\in Dom(\der^{\diamond})$ and $\der^{\diamond}(u)=F$.
\end{lemma}

\subsubsection{Wick products}
Some of our results below will be expressed in terms of Rieman-Wick sums. We give a brief account on these objects, mainly borrowed from \cite{HJT,HY}.

\smallskip

Among functionals $F$ of $x$ such that $F\in\D^{\infty}$, the set of multiple integrals plays a special role. In order to introduce it in the context of a general process $x$ indexed by the plane, consider  an orthonormal basis $\{ e_n; \, n\ge 1 \}$   of
$\mathcal H$ and let  $\hatotimes$ denote  the symmetric tensor
product.  Then
\begin{equation}\label{e.2.2}
f_n=\sum_{\rm { finite}} f_{i_1, \cdots, i_n}
e_{i_1}\hatotimes\cdots \hatotimes e_{i_n}, \quad  f_{i_1, \cdots,
i_n}\in \R
\end{equation}
is an element of $\mathcal H^{\hatotimes n}$ satisfying the relation:
\begin{equation}\label{eq:def-norm-Hn}
\|f_n\|_{\mathcal H^{\hatotimes n}}^2 =\sum_{\rm { finite}}
|f_{i_1, \cdots, i_n}|^2\, .
\end{equation}
Moreover, $\mathcal H^{\hatotimes n}$  is the completion  of  the set of elements like (\ref{e.2.2}) with respect to the norm~(\ref{eq:def-norm-Hn}).

\smallskip

For an element $f_n\in \mathcal H^{\hatotimes n}$, the multiple
It\^o integral of order $n$ is well-defined. First, any  element of
the form given by (\ref{e.2.2}) can be rewritten as
\begin{equation}\label{htensor}
f_n=\sum_{\rm { finite}} f_{j_1 \cdots j_m} e_{j_1}^{\hatotimes
k_1}\hatotimes\cdots \hatotimes e_{j_m}^{\hatotimes
k_m},\end{equation} where the $j_1,\ldots, j_m$ are different and
$k_1+\cdots+k_m=n$. Then, if $f_n\in \ch^{\hatotimes n}$  is given under the form (\ref{htensor}), define its multiple integral as:
\begin{equation}\label{eq:def-mult-intg}
I_n(f_n)=\sum_{\rm { finite}} f_{j_1, \cdots, j_m} H_{k_1}( I_1(e_{j_1}))\cdots
H_{k_m}(I_1( e_{j_m})),
\end{equation}
where $H_k$ denotes the $k$-th normalized
Hermite polynomial given by
$$H_k(x)=(-1)^k
e^{\frac{x^2}2}\frac{d^k}{dx^k}e^{-\frac{x^2}2}=\sum_{j\le k/2}
\frac{(-1)^j k!}{2^j\,j!\,(k-2j)!} x^{k-2j}.$$ It holds that the
multiple integrals of different order are
orthogonal and that
$$\E\lc |I_n(f_n)|^2\rc=n!\,\|f_n\|_{\ch^{\hatotimes n}}^2.$$
This last isometric property allows to extend the multiple integral
for a  general
$f_n\in\ch^{\hatotimes n}$
by $L^2(\Omega)$ convergence. Finally, one can define the
integral of $f_n\in\ch^{\otimes n}$ by putting $I_n(f_n):=I_n(\tilde
f_n),$ where $\tilde f_n\in\ch^{\hatotimes n}$ denotes the
symmetrized version of $f_n$. Moreover, the chaos expansion theorem
states that any square integrable random variable $F\in
L^2(\Omega,\mathcal G,\PP)$, where $\mathcal G$ is the $\sigma$- field
generated by $x$,
 can be written as
\begin{equation}\label{chaos}
F=\sum_{n=0}^\infty I_n(f_n)
\quad\mbox{with}\quad
\E[F^2]=\sum_{n=0}^\infty  n! \|f_n||_{\mathcal H^{\hatotimes
n}}^2\,.
\end{equation}

\smallskip

With these notations in mind, one way to introduce Wick products on a Wiener space is to impose the relation
\begin{equation}\label{eq:wick-product-mult-intg}
I_n(f_n)\di I_m(g_m)=I_{n+m}(f_n\hatotimes g_m)
\end{equation}
for any $f_n\in\ch^{\hatotimes n}$ and $g_m\in\ch^{\hatotimes m}$, where the multiple integrals $I_n(f_n)$ and $I_m(g_m)$ are defined by (\ref{eq:def-mult-intg}).
If $F=\sum_{n=1}^{N_1}I_n(f_n)$ and
$G=\sum_{m=1}^{N_2}I_m(g_m)$, we define $F\di G$ by
$$F\di G=\sum_{n=1}^{N_1}\sum_{m=1}^{N_2}I_{n+m}(f_n\hatotimes g_m).$$
By a limit argument, we can then extend the Wick product to more general
random variables (see \cite{HY} for further details). In this paper, we will take the limits in the
$L^2(\Omega)$ topology.

\smallskip

Some corrections between ordinary and Wick products will be computed below. A simple example occurs for products of $f(x)$ by a Gaussian increment. Indeed, for a smooth function $f$ and $g_1,g_2\in\ch$, it is shown in \cite{HY} that
\begin{equation}\label{eq:correc-products-1}
f(I_{1}(g_{1})) \diamond I_{1}(g_{2})
= f(I_{1}(g_{1})) \, I_{1}(g_{2}) - f'(I_{1}(g_{1})) \lla g_{1}, g_{2}\rra_{\ch}.
\end{equation}

\smallskip

We now state a result which is proven in \cite[Proposition 4.1]{HJT}.
\begin{proposition}\label{prop:wick-mult-intg}
Let $F\in\D^{k,2}$ and $g\in\mathcal H^{\otimes k}$. Then
\begin{enumerate}
\item $F\diamond I_k(g)$ is well defined in $L^2(\Omega)$.
\item $Fg\in\dom\,\delta^{\di k}.$
\item $F\diamond I_k(g)=\delta^{\di k}(Fg)$.
\end{enumerate}
\end{proposition}

\subsubsection{Further assumptions and preliminary results}
In order to simplify our computations, let us introduce some additional assumptions on the covariance $R$:

\begin{hypothesis}\label{hyp:Young}
The covariance $R$ of our centered Gaussian process $x$ belongs to the space $\cac^{1\text{-var}}([0,1]^4)$, and satisfies a factorization property  of the  form
$$
\E[x_{s_{1};t_{1}}x_{s_{2};t_{2}}]=R_{s_{1}s_{2};t_{1}t_{2}}=R^1_{{s_{1}s_{2}}}R^2_{t_{1}t_{2}},
$$
for two covariance functions $R^{1},R^{2}$ on $[0,1]$. In addition, setting $R^i_a=R^{i}_{aa}$ for $a\in[0,1]$ and $i=1,2$,  we assume that $a\mapsto R^{i}_{a}$ is differentiable and we suppose that
\begin{equation}\label{eq:regular-R}
|2R^{i}_{ab}-R^{i}_{aa}-R^{i}_{bb}|\lesssim |a-b|^{\gamma_i}
\end{equation}
for all $a,b\in[0,1]$, with $\gamma_i>1$. Finally we suppose that $(R^{i})'_{a}=\partial_{a}R^{i}_{aa}\in L^{\infty}([0,1])$.
\end{hypothesis}

The first consequence of our Hypothesis \ref{hyp:Young} is that the regularity of $x$ corresponds to the Young type regularity of Section \ref{sec:planar-young}. Indeed, it is readily checked that relation \eqref{eq:regular-R} yields
$$
\E\left[(\der x_{s_{1}s_{2};t_{1}t_{2}})^2\right]\lesssim|s-s'|^{\gamma_1}|t-t'|^{\gamma_2}.
$$
Since $x$ is Gaussian, an easy application of Kolmogorov's criterion ensures that 
\begin{equation}\label{eq:x-Gauss-Holder}
x\in\cp_{1,1}^{\al_{1},\al_{2}},
\quad\text{with}\quad
\al_{1}=\frac{\ga_{1}}{2}-\epsilon_1>\frac12, \quad \al_{2}=\frac{\ga_{2}}{2}-\epsilon_2 >\frac12,
\end{equation}
for arbitrarily small $\epsilon_1,\epsilon_2>0$. This enables us to appeal to Young's integration theory in order to define integrals of the form $\int_{1}\int_{2}\vp(x)\, d_{12}x$.

\smallskip

Let us quote two lemmas concerning H\"older norms in the plane which will feature in our comparison between Stratonovich and Skorohod integrations. The first one deals with the composition of a H\"older process with a nonlinearity $f$:

\begin{lemma}\label{f-control}
Let $f\in C^{2}(\mathbb R)$, $\theta_1,\theta_2>0$ and a rectangle $\Delta\subset[0,1]^{2}$. Then on $\Delta$ we have that 
$$
\|\der_1y\|_{\theta_1,0}\leq
\|y^1\|_{0;\Delta} \, \|\der_1x\|_{\theta_1,0}
$$
and
$$
\|\der y\|_{\theta_1,\theta_2}\lesssim
\lp \|y^1\|_{0;\Delta} + \|y^2\|_{0;\Delta}\rp 
\|\der x\|_{\theta_1,\theta_2} (1+\|\der x\|_{\theta_1,\theta_2}),
$$
where $\|\cdot\|_{0;\Delta}$ stands for the supremum norm on $\Delta$ and $y^{j}$ still denotes $\vp^{(j)}(x)$.
\end{lemma}

Next we also need an integral semi-norm dominating H\"older's norms in the plane. This is given by the following Garsia type result:
\begin{lemma}\label{lemma:G-R}
Let $p>1$, $\theta_1,\theta_2>0$ and $y\in\cp_{1,1}$. The following relation holds true: 
\begin{equation}\label{eq:GRR-ineq}
\|\der y\|^p_{\theta_1,\theta_2}\lesssim_{\theta_1,\theta_2,p}
\int_{[0,1]^4}\frac{|\der y_{u_1u_2;v_1v_2}|^p}{|u_2-u_1|^{\theta_1p+2}|v_2-v_1|^{\theta_2p+2}}du_1du_2dv_1dv_2.
\end{equation}
\end{lemma}

\smallskip

We now turn to a consequence of our additional Hypothesis \ref{hyp:Young} on  embedding properties of the Hilbert space $\ch$ defined above:

\begin{lemma}
Under Hypothesis \ref{hyp:Young}, we have $\|f\|_{\ch}\le \|f\|_{\infty} \|R\|_{1\text{-var};[0,1]^{4}}$.
\end{lemma}

\begin{proof}
Consider a step function $f=\sum_{ij}a_{ij}\1_{\Delta_{ij}}$ related to a partition $(\Delta_{ij})_{ij}$ of $[0,1]^{2}$. We have
\begin{multline}\label{emb}
\|f\|^{2}_{\mathcal H}=\sum_{i,j,l,k}a_{ij}a_{lk}R^{1}_{s_is_l}R^{2}_{t_jt_k}
=\sum_{i,j,k,l}a_{ij}a_{kl}\int_0^{s_i}\int_{0}^{s_l}\int_0^{t_j}\int_{0}^{t_k}d_{12}R^1_{{s_{1}s_{2}}}d_{12}R^2_{t_{1}t_{2}}
\\
=\int_{[0,1]^4}f_{s_{1};t_{1}} \, f_{s_{2};t_{2}}d_{12}R^1_{{s_{1}s_{2}}}d_{12}R^2_{t_{1}t_{2}}
\leq\|f\|_{\infty}^2 \, \|R\|_{1\text{-var};[0,1]^{4}}.
\end{multline}
The general case now easily follows by density of the step functions in $\ch$.

\end{proof}

\smallskip

Let us now recall that we work under the usual assumptions for Skorohod  type change of variables formulae given at Definition \ref{growuthcond} and referred to as (GC) condition in the sequel. Notice that $\max_{s,t\in[0,1]}( R_{s}^{1}R_{t}^{2})=\max_{s,t\in[0,1]}\E[|x_{s;t}|^2]$. Thus condition (GC) implies that
\begin{equation}\label{cotamoments}
\E \Big[\sup_{s,t\in[0,1]}|f(x_{s;t})|^r\Big]<\infty,
\quad\text{for all}\quad
r\ge 1.
\end{equation}
 
We now state an approximation result  in $\ch$ which proves to be useful in order to get our It\^o type formula.

\begin{proposition}\label{prop:cvgce-r-sums-in-H}
Let $x$ be a centered Gaussian process on $[0,1]$ satisfying Hypothesis~\ref{hyp:Young} and $f\in C^1(\mathbb R)$ such that the growth condition (GC) is fulfilled for $\vp$ and $\vp^{(1)}$. Consider a rectangle $\Delta=[s_1,s_2]\times[t_1,t_2]$ and $\pi_1=(s_i)_i$, $\pi_2=(t_j)_j$ two respective dissections of the intervals $[s_1,s_2]$ and $[t_1,t_2]$. Then 
$$
\lim_{|\pi_1|,|\pi_2|\to0}
\E\left[\Big\|y_{\cdot}\1_{\Delta}-\sum_{i,j}y_{s_i;t_j}\1_{\Delta_{i,j}}\Big\|^2_{\mathcal H}\right]=0,
$$
where we have used Notation \ref{not:lesssim-partitions} for the rectangles $\Delta_{i,j}$.
\end{proposition}

\begin{proof}
Observe first that
$$
y_{s;t}\1_{\Delta}(s,t)-\sum_{i,j}y_{s_i;t_j}\1_{\Delta_{ij}}(s,t)=\sum_{i,j}(y_{s;t}-y_{s_i;t_j})\1_{\Delta_{i,j}}(s,t)
$$
from which the following estimation is easily obtained:
$$
|(y_{s;t}-y_{s_i;t_j})\1_{\Delta_{i,j}}(s,t)|\leq
\left(\sup_{(s,t)\in\Delta}|y^1_{s;t}|\max_{|s_{1}-s_{2}|\leq|\pi_1|,|t_{1}-t_{2}|\leq|\pi_2|}|x_{s_{1};t_{1}}-x_{s_{2};t_{2}}|
\right)
\1_{\Delta_{i,j}}.
$$
Hence if we take expectations in this last estimation and resort to H\"older's inequality, we obtain that 
\begin{multline*}
\E\left[\Big\|y_{\cdot}\1_{\Delta}-\sum_{i,j}y_{s_i;t_j}\1_{\Delta_{i,j}}\Big\|^2_{\infty}\right] \\
\leq\E^{1/2}\left[\sup_{(s,t)\in\Delta}|y^1_{s;t}|^{4}\right] \,
\E^{1/2}\left[\max_{|s-s'|\leq|\pi_1|,|t-t'|\leq|\pi_2|}|x_{s;t}-x_{s_{2};t_{2}}|^{4}\right] \,
\Big\|\sum_{i,j}\1_{\Delta_{ij}}\Big\|_{\infty}.
\end{multline*}
Now the r.h.s of this inequality goes to zero when the mesh of the partitions $\pi_1,\pi_2$ goes to zero by continuity properties of $x$ (see \eqref{eq:x-Gauss-Holder}). Our claim thus easily stems from the embedding~\eqref{emb}.

\end{proof}

\subsection{It\^o-Skorohod type formula}

We now turn to one of the main aim of this article, namely the proof of a Skorohod type change of variable formula for a general Gaussian process $x$ defined on $[0,1]^{2}$, under our assumptions \ref{hyp:Young}. Our starting point is the relation between $z^{1}$ and its Skorohod equivalent.

\begin{proposition}\label{f-conv}
Assume $x$ is a centered Gaussian process on $[0,1]^{2}$ with a covariance function satisfying \eqref{eq:factorize-R-intro}. Consider a function $\vp\in C^{2}(\mathbb R)$ satisfying condition (GC) and a rectangle $\Delta=[s_{1},s_{2}]\times[t_{1},t_{2}]$. Then we have that 
$y^{1}_{\cdot}\1_{\Delta}\in Dom(\der^{\diamond})$, and if we define the increment 
$z^{1,\di}\equiv\der^{\diamond}(y^{1}_{\cdot}\1_{\Delta})$ the following relation holds true: 
\begin{equation}\label{eq:rel-sko-strato-d12x-young}
z^{1,\di}_{s_{1}s_{2};t_{1}t_{2}}
=z^{1}_{s_{1}s_{2};t_{1}t_{2}}
-\frac14\int_{s_{1}}^{s_{2}}\int_{t_{1}}^{t_{2}}y^2_{s;t} \, d_{1}R^{1}_{s} d_{2}R^{2}_{t},
\end{equation}
where $z^{1}$ is given by Proposition~\ref{prop:young-2d-y-dx} and the second integral in the right hand side of  \eqref{eq:rel-sko-strato-d12x-young} is of Riemann-Stietjes type. Moreover, relation \eqref{eq:riemann-wick-z1} holds true in the $L^{2}(\Omega)$ and almost sure sense.
\end{proposition}

\begin{proof}
Consider a sequence of partitions $\pi_{n}=(\pi_{n}^{1},\pi_{n}^{2})$ whose mesh go to 0 as $n\to\infty$. The generic elements of $\pi_{n}$ will be denoted by $(s_{i},t_{j})$. Owing to formula~\eqref{ipp div}, we have that 
\begin{align*}
&\sum_{\pi_{n}}\der^{\diamond}(y^{1}_{s_i;t_j}\1_{\Delta_{ij}})=\sum_{\pi_{n}}y^{1}_{s_i;t_j}\der x_{s_{i}s_{i+1};t_{j}t_{j+1}}-\sum_{\pi_{n}}y^2_{s_i;t_j}\E[x_{s_i;t_j}\der x_{s_{i}s_{i+1};t_{j}t_{j+1}}]
\\
&=\sum_{\pi_{n}}y^{1}_{s_i;t_j}\der x_{s_{i}s_{i+1};t_{j}t_{j+1}}-\sum_{\pi_{n}}y^2_{s_i;t_j}(R^{1}_{s_is_{i+1}}-R^{1}_{s_is_i})(R^{2}_{t_jt_{j+1}}-R^{2}_{t_jt_j})
\equiv A_1^{n}-A_2^{n}.
\end{align*}
We now treat those two terms separately.

\smallskip

\noindent
\textit{Step 1: Estimation of $A_1^{n}$.}
The term $A_1^{n}=\sum_{\pi_{n}}y^{1}_{s_i;t_j}\der x_{s_{i}s_{i+1};t_{j}t_{j+1}}$ is a Riemann type sum. Since $x\in\cp_{1,1}^{\al_1,\al_2}$ with $\al_1,\al_2>1/2$, it converges a.s to $\iint y^{1} \, d_{12}x$ as $n$ goes to $\infty$, according to \eqref{eq:cvgce-riem-sums-z1-z2}. 

\smallskip

The $L^{2}(\Omega)$ convergence of the Riemann sums defining $A_1^{n}$ is more cumbersome, and we have to go back to the definition of the Young integral given by Theorem~\ref{prop:young-2d-y-dx}. Indeed one can write $\int_{s_1}^{s_2}\int_{t_1}^{t_2} y^{1}_{s;t}d_{12}x_{s;t}=\sum_{\pi_{n}}\int_{\Delta_{ij}}y^{1}_{s;t}d_{12}x_{s;t}$, and thanks to \eqref{eq:def-z1-z2-(Id-Lambda)} we get that 
\begin{multline*}
\int_{s_1}^{s_2}\int_{t_1}^{t_2} y^{1}_{s;t}d_{12}x_{s;t}
=A_{1}^{n}+\sum_{\pi_{n}}(\id-\Lambda_2\der_2)(\Lambda_1\der_1)(y^{1}\der x)_{s_{i}s_{i+1};t_{j}t_{j+1}}
\\
+\sum_{\pi_{n}}(\id-\Lambda_1\der_1)(\Lambda_2\der_2)(y^{1}\der x)_{s_{i}s_{i+1};t_{j}t_{j+1}}-\sum_{\pi_{n}}\Lambda\der(y^{1}\der x)_{s_{i}s_{i+1};t_{j}t_{j+1}}.
\end{multline*}
Furthermore, some partial summations can be performed on the terms $\id-\laa_{i}\der_{i}$ for $i=1,2$ and setting $D_{1}^{n}\equiv \int_{s_1}^{s_2}\int_{t_1}^{t_2} y^{1}_{s;t}d_{12}x_{s;t}-A_{1}^{n}$ we deduce that
\begin{multline*}
D_{1}^{n}=\sum_{\pi_{n}^{2}}\lc (\id-\Lambda_1\der_1)(\Lambda_2\der_2)\rc(y^{1}\der x)_{s_1s_2;t_jt_{j+1}}
\\
+\sum_{\pi_{n}^{1}}\lc(\id-\Lambda_2\der_2)(\Lambda_1\der_1)\rc(y^{1}\der x)_{s_is_{i+1};t_1t_2}
-\sum_{\pi_{n}}\Lambda\der(y^{1}\der x)_{s_{i}s_{i+1};t_{j}t_{j+1}}.
\end{multline*}
Recalling the H\"older regularity \eqref{eq:x-Gauss-Holder} of our process $x$, we thus obtain
\begin{equation*}
\begin{split}
\left|D_{1}^{n}\right|
&\lesssim 
|s_2-s_1|^{\al_{1}}\|\der_2y^{1}\|_{0,\al_{1}}\|\der x\|_{\al_{1},\al_{2}}\sum_j|t_{j+1}-t_j|^{2\al_{2}}
\\&+
|t_2-t_1|^{\al_{2}}\|\der_1y^{1}\|_{\al_{1},0}\|\der x\|_{\al_{1},\al_{2}}\sum_{\pi_{n}^{1}}|s_{i+1}-s_{i}|^{2\al_{1}}
\\&+\|\der y^{1}\|_{\al_{1},\al_{2}}\|\der x\|_{\al_{1},\al_{2}}
\sum_{\pi_{n}}|s_{i+1}-s_i|^{2\al_{1}}|t_{j+1}-t_j|^{2\al_{2}}.
\end{split}
\end{equation*}
Now taking expectations in the last relation and using H\"older's inequality we end up with 
\begin{multline}\label{A_1 estim}
\E[|D_{1}^{n}|^2]\lesssim_{s_1,s_2,t_1,t_2,\al_{1},\al_{2}}
\left(|\pi_1|^{4\al_{1}-2}+|\pi_2|^{4\al_{2}-2}\right) \,
\E^{1/2}[\|\der x\|^4_{\al_{1},\al_{2}}] \\
\times\left(\E^{1/2}[\|\der_1y^{1}\|^4_{0,\al_{1}}]
+\E^{1/2}[\|\der_2y^{1}\|^{4}_{\al_{2}}]+\E^{1/2}[\|\der y^{1}\|^{4}_{\al_{1},\al_{2}}]\right),
\end{multline}
for $n$ large enough. 

\smallskip

We are now going to prove that one can recast \eqref{A_1 estim} into
\begin{equation}\label{D1-estim-2}
\E[|D_{1}^{n}|^2]\lesssim_{s_1,s_2,t_1,t_2,\al_{1},\al_{2}}
\left(|\pi_1|^{4\al_{1}-2}+|\pi_2|^{4\al_{2}-2}\right).
\end{equation}
Indeed, applying Lemma \ref{lemma:G-R} to $y=x$, $\theta_{1}=\al_{1}$, $\theta_{2}=\al_{2}$ and $p$ large enough, one can easily check that $\E[\|\der x\|_{\al_{1},\al_{2}}^4]<+\infty$. In addition, combining the fact that  $\E[\|\der x\|_{\al_{1},\al_{2}}^4]<+\infty$ with Lemma~\ref{f-control}  plus condition (GC) on the function $f$, we obtain that 
$$
E[\|\der x\|^4_{\al_{1},\al_2}]^{1/2}
\left(\E[\|\der_1y^{1}\|^4_{0,\al_{1}}]^{1/2}
+\E[\|\der_2y^{1}\|^{4}_{\al_{2}}]^{1/2}
+\E[\|\der y^{1}\|^{4}_{\al_{1},\al_2}]^{1/2}\right)<+\infty. 
$$
Hence inequality \eqref{D1-estim-2} is easily deduced from \eqref{A_1 estim}, and this proves that  $A_{1}^{n}$ converges in $L^2(\Omega)$ to $\int_{s_1}^{s_2}\int_{t_1}^{t_2}y^{1}_{s;t}d_{12}x_{s;t}$ as $n\to\infty$.

\smallskip

\noindent
\textit{Step 2: Estimation of $A_2^{n}$.}
Recall that $A_2^{n}$ is defined by 
\begin{equation*}
A_{2}^{n}=\sum_{\pi_{n}}y^{2}_{s_i;t_j}(R^{1}_{s_{i+1}s_i}-R^1_{s_is_i})(R^{2}_{t_{j+1}t_j}-R^2_{t_jt_j}).
\end{equation*}
In order to treat this term, first remark that  for $k=1,2$ we have
\begin{equation*}
R_{s_is_{i+1}}^{k}-R^{k}_{s_is_i}=\frac12(R^{k}_{s_{i+1}s_{i+1}}-R^{k}_{s_is_i})
+\rho^{k}_{s_is_{i+1}},
\end{equation*}
where $\rho^{k}_{s_is_{i+1}}=\frac12(2R^{k}_{s_is_{i+1}}-R^{k}_{s_is_i}-R^{k}_{s_{i+1}s_{i+1}})$.
Injecting this relation in the definition of the term $A_{2}^{n}$ and recalling that we have set $R_{a}^{k}\equiv R_{aa}^{k}$, we obtain
\begin{eqnarray*}
A_{2}^{n}&=&1/4\sum_{\pi_{n}}y^{2}_{s_i;t_j}(R^{1}_{s_{i+1}}-R^{1}_{s_i})(R^{2}_{t_{j+1}}-R^{2}_{t_j})
\\
&&+1/2\sum_{\pi_{n}}y^{2}_{s_i;t_j}\lc 
(R^{2}_{t_{j+1}}-R^{2}_{t_j})\rho^1_{s_is_{i+1}}
+(R^{1}_{s_{i+1}}-R^{1}_{s_i})\rho^2_{t_jt_{j+1}}
+\rho^1_{s_is_{i+1}}\rho^2_{t_jt_{j+1}}
\rc
\\
&\equiv& A_{21}^{n} + A_{22}^{n} + A_{23}^{n} + A_{24}^{n}.
\end{eqnarray*}
We will now show that
\begin{equation}\label{eq:lim-A21-A24}
\lim_{n\to\infty} A_{21}^{n} = \frac14\int_{s_1}^{s_2}\int_{t_1}^{t_2}y^{2}_{s;t}
\, d_{1}R^1_sd_{2}R^2_t,
\quad\text{and}\quad
\lim_{n\to\infty} \sum_{j=2}^{4} A_{2j}^{n} = 0,
\end{equation}
where the limits are understood in the almost sure and $L^{2}$ sense.

\smallskip

Indeed, it is easily understood that the terms $A_{22}^{n},A_{23}^{n}, A_{24}^{n}$ are remainder terms: according to Hypothesis \ref{hyp:Young} we have that $|\rho^{i}_{ab}|\lesssim|a-b|^{\gamma_i}$, and we get the following inequality for $A_{22}^{n}$:
\begin{eqnarray*}
A_{22}^{n}&\lesssim&
|\pi_1|^{\gamma_1-1}\sup_{(s,t)\in\Delta}|y^{2}_{s;t}| \,
\sum_{\pi_{n}} (s_{i+1}-s_i) \, |R^{2}_{t_{j+1}}-R^2_{t_j}|
\\
&\leq&
|\pi_1|^{\gamma_1-1}\sup_{(s,t)\in\Delta}|y^{2}_{s;t}| \, (s_2-s_1)\int_{t_1}^{t_2}|d_{2}R^{2}_t|.
\end{eqnarray*}
This relation, plus the condition (GC) on $f$, obviously entails that $\lim_{n\to\infty}A_{22}^{n}=0$ in the almost sure and $L^{2}(\Omega)$ sense. The case of $A_{23}^{n}, A_{24}^{n}$ follow exactly along the same lines.

\smallskip

We now focus on the term $A_{21}^{n}$: observe that 
\begin{multline*}
\left|a_1-1/4\int_{s_1}^{s_2}\int_{t_1}^{t_2}y^{2}_{s;t}d_{1}R^1_sd_{2}R^2_t\right|  \\
\lesssim
\sup_{(s,t)\in\Delta}|y^2_{s;t}|
\max_{|s-s'|  \leq|\pi_1|,|t-t'|\leq|\pi_2|}|x_{s;t}-x_{s_{2};t_{2}}|
\int_{s_1}^{s_2}\int_{t_1}^{t_2}|d_{1}R^1_s\|d_{2}R^2_t|.
\end{multline*}
Invoking the same estimates as before for the H\"older norm of $x$ and condition (GC) on $f$, the proof of our assertion \eqref{eq:lim-A21-A24} is now completed. 

\smallskip

\noindent
\textit{Step 3: Conclusion.}
Let us summarize the results obtained in the last two steps: plugging  relation~\eqref{eq:lim-A21-A24} into the definition of $A_{2}^{n}$ and recalling the limiting behavior of  $A_{1}^{n}$ established at Step 1, we have obtained that
\begin{equation*}
\lim_{n\to\infty} \sum_{\pi_{n}}\der^{\diamond}(y^{1}_{s_i;t_j}\1_{\Delta_{ij}})
=
\int_{s_1}^{s_2}\int_{t_1}^{t_2} y^{1}_{s;t}d_{12}x_{s;t} 
-\frac14\int_{s_1}^{s_2}\int_{t_1}^{t_2}y^{2}_{s;t}d_{1}R^1_sd_{2}R^2_t.
%=\lim_{|\pi_1|,|\pi_1|\to0}\sum_{\pi_{n}}f(x_{s_i;t_j})\diamond\der x_{s_{i}s_{i+1};t_{j}t_{j+1}}
\end{equation*}
where the convergence is understood in both a.s and $L^2(\Omega)$ sense. Furthermore, Proposition~\ref{prop:cvgce-r-sums-in-H} asserts that $\sum_{\pi_{n}}y^{1}_{s_i;t_j}\1_{\Delta_{ij}}$ converges in $L^{2}(\Omega,\mathcal H)$  to $y^{1}_{\cdot}\1_{\Delta}$. This finishes our proof  of relation~\eqref{eq:rel-sko-strato-d12x-young} thanks to a direct application of Lemma~\ref{clos}.

\smallskip

As far as expression \eqref{eq:riemann-wick-z1} with Wick-Riemann sums is concerned, recall that we have proved that 
\begin{equation*}
\der^{\diamond}(y^{1}_{\cdot}\1_{\Delta})
=\lim_{|\pi_1|,|\pi_2|\to0}\sum_{\pi_{n}}\der^{\diamond}(y^{1}_{s_i;t_j}\1_{\Delta_{ij}}).
\end{equation*}
Now invoke Proposition \ref{prop:wick-mult-intg} for $k=1$ in order to state that
\begin{equation*}
\der^{\diamond}(y^{1}_{s_i;t_j}\1_{\Delta_{ij}})
=y^{1}_{s_i;t_j} \diamond \der^{\diamond}(\1_{\Delta_{ij}})
=y^{1}_{s_i;t_j}\diamond\der x_{s_{i}s_{i+1};t_{j}t_{j+1}},
\end{equation*}
which ends the proof.

\end{proof}

\smallskip

Proposition \ref{f-conv} gives a meaning to the increment $z^{1,\di}$  and compares them to the corresponding Stratonovich increment $z^{1}$. In order to compare change of variables formulae, we still have to define Skorohod integrals of the form $z^{2,\di}$, which is what we proceed to do now.

\smallskip

To this aim, let us start by some formal considerations: it is easily conceived that 
\begin{equation}\label{eq:mix-term}
\int_{0}^{s}\int_{0}^{t}y^{2}_{u;v}d^{\diamond}_1x_{u;v}d^{\diamond}_2x_{u;v}
=\int_{0}^{s}\int_{0}^{t}\int_{0}^{u}\int_{0}^{v}y^{2}_{u;v}d^{\diamond}_{12}x_{u'v}d^{\diamond}_{12}x_{uv'}
=\der^{\diamond,2}(N(y))
\end{equation}
where, similarly to \cite{TVp1}, we set
$$
N(y)_{u'u;vv'}:=y \, \1_{[0,s]\times[0,v]}(u,v') \, \1_{[0,u]\times[0,t]}(u',v),
$$
where we integrate firstly  in $(u',v)$ and then in $(u,v')$, and where the notation $\der^{\diamond,2}$ specifies that we perform double integrals in the Skorohod sense. Our objective in what follows is to give a rigorous meaning to equation~\eqref{eq:mix-term}. 

\begin{lemma}\label{conver-mixed}
Take up the notation of Proposition \ref{f-conv}, and consider $f\in C^{3}(\mathbb R)$ satisfying condition (GC). For a sequence of partitions $(\pi_{n})_{n\ge 1}$ whose mesh goes to 0  define
\begin{equation}\label{eq:def-a-pi1-pi2-mixed}
a^{\pi_{n}}_{u'u;vv'}=\sum_{i,j}y^{2}_{s_i;t_j} \,
\1_{[0,s_i]\times[t_j,t_{j+1}]}(u',v) \, \1_{[s_i,s_{i+1}]\times[0,t_j]}(u,v').
\end{equation}
Then $a^{\pi_{n}}$ converges to $N(y)$ in $L^2(\Omega,\mathcal \ch^{\otimes2})$ as $n$ goes to infinity.
\end{lemma}

\begin{proof}

First notice that the tensor norm of an element $K\in\mathcal \ch^{\otimes2}$ can be bounded as:
\begin{eqnarray}\label{embb-2}
\|K\|_{\mathcal \ch^{\otimes2}}
&=&\int_{[0,1]^8}K_{a_{1}a'_{1};b_{1}b'_{1}}K_{a_{2}a'_{2};b_{2}b'_{2}}d_{12}R^{1}_{a_1a'_1}d_{12}R^{1}_{a_2a'_2}d_{12}R^{2}_{b_1b'_1}d_{12}R^{1}_{b_2b'_2}
\notag\\
&\leq&\int_{[0,1]^8}|K_{a_{1}a'_{1};b_{1}b'_{1}}K_{a_{2}a'_{2};b_{2}b'_{2}}|
\, |d_{12}R^{1}_{a_1a'_1}| |d_{12}R^{1}_{a_2a'_2}| |d_{12}R^{2}_{b_1b'_1}| 
|d_{12}R^{1}_{b_2b'_2}|.
\end{eqnarray}
Furthermore, a  simple computation shows that
\begin{multline*}
a^{\pi_{n}}_{u'u;vv'}-N(y^{2})_{u'u;vv'}
=\sum_{\pi_{n}}
\lc y^{2}_{s_i;t_j}-y^{2}_{u;v}\rc
\lc\1_{[0,u]\times[t_j,t_{j+1}]}(u',v)\1_{[s_i,s_{i+1}]\times[0,v]}(u,v')\rc
\\
+\sum_{\pi_{n}}y^{2}_{s_i;t_j}
\lc\left(\1_{[0,s_i]\times[0,t_j]}(u',v')-\1_{[0,u]\times[0,v]}(u',v')\right)\1_{[s_i,s_{i+1}]\times[t_j,t_{j+1}]}(u,v)\rc,
\end{multline*}
and thus, 
\begin{multline}\label{eq:bnd-a-p1-pi2}
\lln a^{\pi_{n}}_{u'u;vv'}-N(y^{2})_{u'u;vv'} \rrn
\leq
\Big(\sup_{(a,b)\in[0,s]\times[0,t]}|y^3_{a;b}|\sup_{|a_2-a_1|\leq|\pi_1|,|b_2-b_1|\leq|\pi_2|}|x_{a_2;b_2}-x_{a_1;b_1}|\\
+\max_{i,j}(\1_{[s_i,s_{i+1}]}(u')+\1_{[t_j,t_{j+1}]}(v'))\sup_{(a,b)\in[0,s]\times[0,t]}|y^{2}_{a;b}|
\Big).
\end{multline}
Our claims are now easily derived: on the one hand the right hand side of \eqref{eq:bnd-a-p1-pi2} converges to zero when $n\to\infty$
if $u'\ne s_i$ and $v'\ne t_j$ for all $i,j$. Then using  inequality~\eqref{embb-2} and dominated convergence we obtain that  $a^{\pi_{n}}$ converges a.s to $N(y)$ in $\mathcal \ch^{\otimes2}$. On the other hand, in order to obtain the convergence in $L^2(\Omega,\mathcal \ch^{\otimes2})$ it suffices to use the fact that $f$ satisfies condition (GC) and apply once again dominated convergence.

\end{proof}

Now we are able to define our mixed integral in the Skorohod sense and connect it to the equivalent integral in the Young theory:
\begin{proposition}\label{proposition:contrac 2}
Assume $x$ is a centered Gaussian process on $[0,1]^{2}$ with a covariance function satisfying \eqref{eq:factorize-R-intro}. Consider a function $\vp\in C^{4}(\mathbb R)$ satisfying condition (GC) and a rectangle $\Delta=[s_{1},s_{2}]\times[t_{1},t_{2}]$. Then we have that $N(y)\in Dom(\der^{\diamond,2})$, and if we define $z^{2,\di}=\der^{\diamond,2}(N(y))$  the following relation holds: 
\begin{multline}\label{eq:strato-skoro-mixed-intg}
z^{2,\di}_{s_{1}s_{2};t_{1}t_{2}}
=z^{2}_{s_{1}s_{2};t_{1}t_{2}}
-\frac14\int_{s_{1}}^{s_{2}}\int_{t_{1}}^{t_{2}}y^{2}_{u;v}d_{1}R^1_u \, d_{2}R^2_v 
-\frac12\int_{s_{1}}^{s_{2}}\int_{t_{1}}^{t_{2}}y^{3}_{u;v}R^{1}_u \, d_{2}R^{2}_v \, d_{1}x_{u;v} \\
-\frac12\int_{s_{1}}^{s_{2}}\int_{t_{1}}^{t_{2}}y^{3}_{u;v}R^2_v \, d_{1}R^1_u \, d_{2}x_{u;v} 
+\frac14\int_{s_{1}}^{s_{2}}\int_{t_{1}}^{t_{2}}y^{4}_{u;v}R^1_uR^2_v \, d_{1}R^1_u \, d_{2}R^2_v,
\end{multline}
where $\int_{s_{1}}^{s_{2}}\int_{t_{1}}^{t_{2}}y^{3}_{u;v}R^1_u \, d_{2}R^{2}_v \, d_{1}x_{u;v}$ and $\int_{s_{1}}^{s_{2}}\int_{t_{1}}^{t_{2}}y^{3}_{u;v} R^2_v \, d_{1}R^1_u \,d_{2}x_{u;v}$ are defined according to Proposition \ref{mixed-Young-2D}. Moreover, relation \eqref{eq:riemann-wick-z2} holds true in the $L^{2}(\Omega)$ and almost sure sense.

\end{proposition}

\begin{proof}
Like for Proposition \ref{f-conv}, our strategy is as follows: consider a sequence $\pi_{n}=(\pi_{n}^{1},\pi_{n}^{2})$ whose mesh go to 0 as $n\to\infty$ and set $a^{n}\equiv a^{\pi_{n}}$ defined by \eqref{eq:def-a-pi1-pi2-mixed}. We have seen at~Lemma \ref{conver-mixed} that $\lim_{n\to\infty}a^{n}=N(y)$ in $L^{2}(\Omega,\mathcal \ch^{\otimes2})$. We shall now study the convergence of $\delta^{\diamond}(a^{n})$ by means of Wick-Stratonovich corrections. Then we will conclude by invoking Proposition~\ref{clos}.

\smallskip

\noindent
\textit{Step 1: Wick-Stratonovich corrections.}
According to relation \eqref{eq:wick-product-mult-intg} and Proposition \ref{prop:wick-mult-intg} for $k=2$ we obtain
\begin{eqnarray}\label{eq:sko-a-n-wick-sum}
\der^{\diamond,2}(a^{n})
&=&\sum_{\pi_{n}}\der^{\diamond,2}(y^{2}_{s_i;t_j}
\1_{[0,s_i]\times[t_j,t_{j+1}]}\otimes\1_{[s_i,s_{i+1}]\times[0,t_j]}) \notag\\
&=&\sum_{\pi_{n}}y^{2}_{s_i;t_j}\diamond\der_2x_{s_i;t_jt_{j+1}}\diamond\der_1x_{s_is_{i+1};t_j}.
\end{eqnarray}
We now use Theorem 4.10 in \cite{HY} in order to get that $\der^{\diamond,2}(a^{n})$ can be decomposed as:
\begin{align}\label{eq:dcp-delta-a-n}
&
\sum_{\pi_{n}}y^{2}_{s_i;t_j}(\der_2x_{s_i;t_jt_{j+1}}\diamond\der_1x_{s_is_{i+1};t_j})
-\sum_{\pi_{n}}y^{3}_{s_i;t_j}R^1_{s_i}(R^2_{t_jt_{j+1}}-R^2_{t_jt_j})\der_1x_{s_is_{i+1};t_j}
\notag\\
&-\sum_{\pi_{n}}y^{3}_{s_i;t_j}R^2_{t_j}(R^1_{s_is_{i+1}}-R^1_{s_i,s_i})\der_2x_{s_i;t_jt_{j+1}} \\
&+\sum_{\pi_{n}}y^{4}_{s_i;t_j}R^1_{s_i}R^2_{t_j}(R^1_{s_is_{i+1}}-R^1_{s_is_i})(R^2_{t_jt_{j+1}}-R^2_{t_jt_j}) \equiv B_{1}^{n}-B_{2}^{n}-B_{3}^{n}+B_{4}^{n}. \notag
\end{align}
Like in the proof of Proposition \ref{f-conv}, we treat those 4 terms separately. 

\smallskip

\noindent
\textit{Step 2: Estimation of $B_{1}^{n},\ldots,B_{4}^{n}$.}
The term $B_{1}^{n}$ can be decomposed as
$$
B_{1}^{n}=\sum_{\pi_{n}}y^{2}_{s_i;t_j}\der_1x_{s_is_{i+1};t_j}\der_2x_{s_i;t_jt_{j+1}}-\sum_{\pi_{n}}y^{2}_{s_i;t_j}(R^1_{s_is_{i+1}}-R^1_{s_is_i})(R^2_{t_jt_{j+1}}-R^2_{t_jt_j}).
$$
Moreover, the second term in the r.h.s is the same as $A_{2}^{n}$ in the proof of Proposition \ref{f-conv}, while the convergence for $\sum_{\pi_{n}}y^{2}_{s_i;t_j}\der_1x_{s_is_{i+1};t_j}\der_2x_{s_i;t_jt_{j+1}}$ follows exactly along the same lines as $A_{1}^{n}$ in the same proof. We thus leave to the patient reader the task of showing that
\begin{equation}\label{eq:lim-B-1}
\lim_{n\to\infty} B_{1}^{n} = 
\int_{s_1}^{s_2}\int_{t_1}^{t_2}y^{2}_{s;t}
\, d_{1}x_{s;t} d_{2}x_{s;t}
-\frac14\int_{s_1}^{s_2}\int_{t_1}^{t_2}y^{2}_{s;t}
\, d_{1}R^1_sd_{2}R^2_t,
\end{equation}
and we concentrate now on the other terms in \eqref{eq:dcp-delta-a-n}.

\smallskip

The term $B_{2}^{n}=\sum_{\pi_{n}}y^{3}_{s_i;t_j}R^1_{s_i}(R^2_{t_jt_{j+1}}-R^2_{t_jt_j})\der_1x_{s_is_{i+1};t_j}$ can be decomposed as $B_{2}^{n}=B_{21}^{n}+B_{22}^{n}$, with
\begin{equation*}
B_{21}^{n}=
\frac12\sum_{\pi_{n}}y^{3}_{s_i;t_j}R^1_{s_i}(R^2_{t_{j+1}}-R^2_{t_j})\der_1x_{s_is_{i+1};t_j},
\quad
B_{22}^{n}=
\sum_{\pi_{n}}y^{3}_{s_i;t_j}R^1_{s_i}\rho^2_{t_jt_{j+1}}\der_1x_{s_is_{i+1};t_j},
\end{equation*}
where we recall that we have set $\rho^{k}_{t_jt_{j+1}}=\frac12(2R^{k}_{t_jt_{j+1}}-R^{k}_{t_jt_{j}}-R^{k}_{t_{j+1}t_{j+1}})$ for $k=1,2$.

\smallskip

It is now easily seen that the almost sure and $L^{2}$ convergence of $B_{2}^{n}$ are obtained with the same kind of considerations as for $A_{2}^{n}$ in the proof of Proposition \ref{f-conv}. We get that
\begin{equation*}
\lim_{n\to\infty} B_{2}^{n}
=\frac12 \int_{0}^{s}\int_{0}^{t}y^{3}_{u;v}R^{1}_u \, d_{2}R^{2}_v \, d_{1}x_{u;v},
\end{equation*}
and $B_{3}^{n}, B_{4}^{n}$ are also handled in the same way.

\smallskip

\noindent
\textit{Step 3: Conclusion.}
Thanks to Step 1 and Step 2, we have obtained that $\delta^{\diamond}(a^{n})$ converges to the right hand side of relation \eqref{eq:strato-skoro-mixed-intg} as $n\to\infty$, in both almost sure and $L^{2}$ senses. As mentioned before, this limiting behavior plus the convergence of $a^{n}$ to $N(y)$ established at Lemma \ref{conver-mixed} yield relation \eqref{eq:strato-skoro-mixed-intg} by a direct application of Proposition \ref{clos}. Furthermore, relation \eqref{eq:riemann-wick-z2} is also a direct consequence of relation \eqref{eq:sko-a-n-wick-sum}.

\end{proof}

\smallskip

Notice that our formula \eqref{eq:strato-skoro-mixed-intg} involves some mixed integrals of the form $\int_{0}^{s}\int_{0}^{t}y^{3}_{u;v}R^2_v d_{1}R^1_u$ $d_{2}x_{u;v}$, which are defined as Young type integrals. The following proposition, whose proof is similar to Propositions \ref{f-conv} and \ref{proposition:contrac 2} and is left to the reader for sake of conciseness, gives a meaning to the analogue integrals in the Skorohod setting.

\begin{proposition}\label{proposition:contrac 3}
Let $f\in C^4(\mathbb R)$ be a function satisfying  condition (GC). Then for every fixed $u\in[0,s]$ we have that 
$v\mapsto y^{3}_{u;v}R^2_{v}\in Dom(\der^{\diamond,u})$ where $\der^{\diamond,u}$ is the divergence operator associated to the process $(x_{u;v})_{v\in[0,t]}$. We can thus define $ \int_{0}^{s}\int_{0}^{t}y^{3}_{u;v}R^{2}_{v}d_{1}R^1_u d_{2}^{\di}x_{u;v}$ by :
$$
\int_{0}^{s}\int_{0}^{t}y^{3}_{u;v}R^{2}_{v} \, d_{1}R^1_u \, d^\diamond_{2}x_{u;v}:=\lim_{|\pi|\to0}\sum_{\pi_{n}}\der^{\diamond,s_i}(y^{3}_{s_i;t_j}\1_{[t_j,t_{j+1}]})R^2_{t_j}(R^1_{s_{i+1}}-R^1_{s_i})
$$
where the convergence holds in both $L^2(\Omega)$ and almost sure senses. In addition, we have the following identity:
\begin{equation*}
\int_{0}^{s}\int_{0}^{t}y^{3}_{u;v}R^{2}_{v}\, d_{1}R^1_u\, d^\diamond_{2}x_{u;v} 
=\int_{0}^{s}\int_{0}^{t}y^{3}_{u;v}R^{2}_{v}\, d_{1}R^1_u\, d_{2}x_{u;v}
-\frac12 \int_{0}^{s}\int_{0}^{t}y^{4}_{u;v}R^1_uR^2_v\, d_{1}R^1_u\, d_{2}R^2_v.
\end{equation*}
Finally, the integral $\int_{0}^{s}\int_{0}^{s}y^{3}_{u;v}R^1_{u}d_{2}R^2_vd^\diamond_{1}x_{u;v}$  is defined similarly.
\end{proposition}

\begin{remark}
We have defined all the integrals we needed in order to prove our Skorohod change of variable formula \eqref{eq:skor-formula-smooth}. Indeed, the proof of formula \eqref{eq:skor-formula-smooth} is now easily deduced by injecting the identities of Propositions~\ref{f-conv},~\ref{proposition:contrac 2} and ~\ref{proposition:contrac 3} in the Stratonovich type formula~\eqref{eq:strato-intro-notation-d12}.
\end{remark}

\section{Skorohod's calculus in the rough case}
\label{sec:skorohod-rough}
Our goal in this section is to extend the formulae given in Propositions~\ref{f-conv} and~\ref{proposition:contrac 2} to rougher situations, namely for Gaussian processes in the plane with H\"older regularities smaller than $1/2$. This is however a harder task than in the Young case, and this is why we introduce 2 simplifications in our considerations:

\smallskip

\noindent
(1) Instead of dealing with a general centered Gaussian process whose covariance admits the factorization property of Hypothesis \ref{hyp:Young}, we handle here the case of a fractional Brownian sheet $(x_{s;t})_{(s,t)\in[0,1]^{2}}$ with Hurst parameters $\gamma_1,\gamma_2\in(1/3,1/2]$.  

\smallskip

\noindent
(2) The definition of our Skorohod integrals with respect to $x$ is obtained in the following way:  we first regularize $x$ as a smooth process $x^{n}$. For this process we can still use the formulae of Propositions~\ref{f-conv} and~\ref{proposition:contrac 2} like in the Young case. We shall then perform a limiting procedure on these formulae (this is where the specification of a concrete approximation is important), which will give our Stratonovich-Skorohod corrections. Notice however that the interpretation in terms of Riemann-Wick sums will be lost with this strategy. 

\smallskip

As in the previous section, we start our considerations by specifying the Malliavin framework in which we are working.

\subsection{Further Malliavin calculus tools}
Recall that the covariance function of our fractional Brownian sheet $x$ is given by \eqref{eq:def-cov-fBs}. We can thus consider a Hilbert space  $\mathcal H^{x}$ related to $x$ exactly as in Section \ref{sec:malliavin-framework}, where we now stress the dependence in $x$ of $\mathcal H^{x}$ in order to differentiate it from the Hilbert space related to white noise. In particular we denote by $I_1^{x}$ the isometry between $\mathcal H^{x}$ and the first chaos generated by $x$.

\smallskip

However, the Malliavin structure related to the harmonizable representation of $x$ will also play a prominent role in the sequel. Namely, it is well known (see e.g. \cite{ST94}) that for $s,t\in[0,1]$, $x$ can be represented as  
\begin{equation}\label{eq:harmonizable-fBs}
x_{s;t}=c_{\gamma_1,\gamma_2} \, \hat W(Q_{s;t})
=c_{\gamma_1,\gamma_2} \int_{\mathbb R^2}Q_{s;t}(\xi,\eta) \, \hat W(d\xi,d\eta),
\end{equation}
where $c_{\gamma_1,\gamma_2}$ is a normalization constant whose exact value is  irrelevant for our computations, $W$ is the Fourier transform of the white noise on $\R^{2}$,  and $Q_{s;t}$ is a kernel defined by
\begin{equation}\label{eq:def-Q}
Q_{s;t}(\xi,\eta)=\frac{e^{\imath s\xi}-1}{|\xi|^{\gamma_1+\frac{1}{2}}} \,
\frac{e^{\imath t\eta}-1}{|\eta|^{\gamma_2+\frac{1}{2}}}.
\end{equation}
This induces us to consider the canonical Hilbert space related to $\hat{W}$, that is $\mathcal H^{\hat W}=L^{2}(\mathbb R^{2})$. The relations between Malliavin calculus with respect to $\hat{W}$ and $x$ are then summarized in the next lemma:

\begin{lemma}\label{lem:derivatives-Hx-HW}
Denote by $\D^{x,k,p}$ (resp. $\D^{\hat{W},k,p}$) the Sobolev spaces related to $x$ (resp. $\hat{W}$), and recall the notation $\mathbb L^{1,2}=\mathbb D^{\hat{W},1,2}(L^2(\mathbb R))$ borrowed from \cite{Nu06}. For $\phi:[0,1]^{2}\to\R$, set 
\begin{equation}\label{eq:def-K}
K\phi(\xi,\eta)=\int_{[0,1]^{2}}\phi_{s;t} \, \partial_{s}\partial_{t}Q_{s;t}(\xi,\eta) \, dsdt,
\end{equation}
where we recall that $Q$ is defined by \eqref{eq:def-Q}. Then the following holds true:

\smallskip

\noindent
\emph{(i)} We can represent the space $\mathcal H^{x}$ as the closure of the set of step functions under the norm $\|\phi\|_{\mathcal H^x}=\|K\phi\|_{L^2(\mathbb R^{2})}$. 

\smallskip

\noindent
\emph{(ii)} We have $\mathbb D^{x,1,2}(\mathcal H^{x})=K^{-1}(\mathbb L^{1,2})$. In addition, for any smooth function $F$ and any $\mathcal H^{x}$-valued square integrable random variable $u$ the following identity holds:
$$
\langle u,D^{x}F\rangle_{\mathcal H^{x}}=\langle Ku,D^{\hat W}F\rangle_{L^2(\mathbb R^{2})}.
$$

\smallskip

\noindent
\emph{(iii)} As far as divergence operators are concerned, the relation is 
\begin{equation*}
\dom (\der^{x,\diamond})=K^{-1} \dom(\der^{\hat W,\diamond}),
\quad\text{and}\quad 
\der^{x,\diamond}(u)=\der^{\hat W,\diamond}(Ku).
\end{equation*}
\end{lemma}

\begin{proof}
Let $\phi=\sum_{i,j}\phi_{i,j}\1_{[s_i,s_{i+1}]\times[t_j,t_{j+1}]}$ be a step function. We have that:
\begin{equation}\label{eq:comp1}
\begin{split}
I_1^{x}(\phi)&=\sum_{i,j}\phi_{i,j}\der x_{s_{i}s_{i+1};t_{j}t_{j+1}}
=\sum_{i,j}\phi_{i,j}\hat W(\der Q_{s_is_{i+1}t_{j}t_{j+1}})
\\&=\hat W\left(\sum_{i,j}\phi_{i,j}\der Q_{s_{i}s_{i+1};t_{j}t_{j+1}}\right)
=\hat W(K\phi),
\end{split}
\end{equation}
which easily yields our first claim (i). 

\smallskip

Let now $F$ be a smooth functional of $x$ of the form $F=f(x_{s_1;t_1},\ldots,x_{s_n;t_n})$. Then 
\begin{equation}\label{eq:comp2}
\begin{split}
\mathbb E\left[\langle u,D^{x}F\rangle_{\mathcal H^{x}}\right]
&=\mathbb E\Big[\sum_{l\in\{1,\ldots,n\}}\partial_l f(x_{s_1t_1},\ldots,x_{s_nt_n})\langle u,\1_{[0,s_l]\times[0,t_l]}\rangle_{\mathcal H^{x}} \Big]
\\&=\mathbb E\Big[
\sum_{l\in\{1,\ldots,n\}}\partial_l f(x_{s_1t_1},\ldots,x_{s_nt_n})\langle Ku,K\1_{[0,s_l]\times[0,t_l]}\rangle_{L^{2}(\mathbb R)}\Big],
\end{split}
\end{equation}
and since $K\1_{[0,s_l]\times[0,t_l]}=Q_{s_l,t_l}$ we end up with
\begin{eqnarray*}
\mathbb E\left[\langle u,D^{x}F\rangle_{\mathcal H^{x}}\right]&=&
\mathbb E\Big[\langle Ku,\sum_{l\in\{1,\ldots,n\}}\partial_l f(\hat W(Q_{s_1t_1}),\ldots,\hat W(Q_{s_nt_n}))Q_{s_l,t_l}\rangle_{L^{2}
(\mathbb R)}\Big]
\\&=&\mathbb E\left[\langle Ku,D^{\hat W}F\rangle_{L^2(\mathbb R)}\right],
\end{eqnarray*}
which gives our assertion (ii) by density of smooth functionals. Relation (iii) is easily derived from (ii) by duality.

\end{proof}

\smallskip

Notice that the preceding result can be extended to second order derivatives thanks to a simple tensorization trick. We label here the result for further use:
\begin{lemma}\label{lem:second-derivatives-Hx-HW}
Under the conditions of Lemma \ref{lem:derivatives-Hx-HW}, set 
\begin{equation}\label{eq:def-K2}
\lc K^{\otimes2}\phi\rc \!(\xi_{1}\xi_{2};\eta_{1}\eta_{2})=
\int_{[0,1]^{4}} \phi_{s_{1}s_{2};t_{1}t_{2}} \,
\partial_{st}Q_{s_{1};t_{1}}(\xi_{1},\eta_{1}) \partial_{st}Q_{s_{2};t_{2}}(\xi_{2},\eta_{2}) \,
ds_{1}ds_{2}dt_{1}dt_{2}.
\end{equation}
Then for any smooth functional $F$ and any $(\mathcal H^{x})^{\otimes 2}$-valued square integrable random variable $u$ we have:
$$
\langle u,D^{2,x}F\rangle_{\mathcal H^{x}}=\langle K^{\otimes 2}u,D^{2,\hat W}F\rangle_{L^2(\mathbb R^{4})} .
$$
\end{lemma}

\subsection{Embedding results}
Similarly to \cite{BJ}, we now give an embedding result for the space $\mathcal H^x$ which proves to be useful for further computations.

\begin{lemma}
\label{lemma: Sobolev}
Let $\gamma_1,\gamma_2\in(0,\frac{1}{2}]$. Then the following inequality is satisfied:
\begin{align}\label{eq:sob embb}
&\|u\|^2_{\mathcal H}\lesssim
\int_{\mathbb R^2}\left(\int_{\mathbb R^2}\frac{|\der \tilde u_{s_1s_2; t_1t_2}|^2}{|s_2-s_1|^{2-2\gamma_1}|t_2-t_1|^{2-2\gamma_2}}ds_1dt_1\right)ds_2dt_2 \notag
\\&\hspace{0.5cm}+\int_{\mathbb R^2}\frac{1}{|s_2-s_1|^{2-2\gamma_1}}\left(\int_{0}^{1}\left|\tilde{u}_{s_2;t}- \tilde{u}_{s_1;t}\right|^2dt\right)ds_1ds_2
\\&\hspace{0.5cm}+\int_{\mathbb R^2}\frac{1}{|t_2-t_1|^{2-2\gamma_2}}\left(\int_{0}^{1}\left|\tilde{u}_{s;t_2}-\tilde{u}_{s;t_1}\right|^2ds\right)dt_1dt_2+\int_{[0,1]^{2}}|\tilde{u}_{s;t}|^2dsdt, \notag
\end{align}
where we have set $\tilde u_{s;t}=u_{s;t}\1_{[0,1]^{2}}(s,t)$.
\end{lemma}

\begin{proof}
In this proof we only consider the case $\gamma_1,\gamma_2<1/2$. Indeed, if $\gamma_1=1/2$ or $\gamma_{2}=1/2$ then our process $x$ is simply a Brownian motion in the first or in the second direction, and this situation is handled by $L^{2}$ norms.

\smallskip

For $\gamma_1,\gamma_2<1/2$, definitions \eqref{eq:def-Q} and \eqref{eq:def-K} entail:
\begin{equation*}
\|u\|^2_{\mathcal  H}=\int_{\mathbb R^2}|\xi|^{1-2\gamma_1}|\eta|^{1-2\gamma_2}\left|\int_{[0,1]^{2}}u_{s;t}e^{\imath s\xi+\imath t\eta}dsdt\right|^2d\xi d\eta,
\end{equation*}
from which one deduces that $\ch$ is isometric to $H^{1/2-\gamma_1}\otimes H^{1/2-\gamma_1}$, where $H^{1/2-\gamma_1}$ stands for the Sobolev space $W^{1/2-\gamma_1,2}$.
Now we use the fact that $1/2-\gamma_1\in(0,1/2)$, and recall that the norm defined by 
$$
\mathcal N_{1/2-\gamma_1}^2(\phi)=\int_{\mathbb R^2}\frac{|\phi_{s_1}-\phi_{s_2}|^2}{|s_2-s_1|^{2-2\gamma_1}}ds_1ds_2+\int_{\mathbb R}|\phi_s|^2ds
$$
is equivalent to the usual norm in $H^{1/2-\gamma_1}$. This yields \eqref{eq:sob embb} by tensorization.

\end{proof}

Let us now introduce two more semi-norms:
\begin{equation*}
\|f\|_{\alpha,1}=\sup_{(s_1,s_2,t)\in[0,1]^2\times[0,1]}\frac{\left|\der_1f_{s_1s_2;t}\right|}{|s_2-s_1|^{\alpha}},
\quad\text{and}\quad
\|f\|_{\beta,2}=\sup_{(s,t_1,t_2)\in[0,1]^{2}}\frac{\left|\der_2f_{s;t_1t_2}\right|}{|t_2-t_1|^{\beta}}.
\end{equation*}
With these notations in hand, the following embedding result is easily deduced from Lem\-ma~\ref{lemma: Sobolev}.

\begin{corollary}\label{prop:embbeding-cam}
Let $\gamma_1,\gamma_2\in(0,1/2)$ and $u\in\cp_{1,1}^{\alpha_1,\alpha_2}$ such that $0<\frac{1}{2}-\alpha_i<\gamma_i$. Then we have the following embedding: 
\begin{equation}\label{eq:embed-H-holder}
\|u\|_{\mathcal H}\lesssim \cn_{\alpha_1,\alpha_2}(u),
\quad\text{where}\quad
\cn_{\alpha,\beta}(f)=\|f\|_{\alpha,\beta}+\|f\|_{\alpha,1}+\|f\|_{\beta,2}+\|f\|_{\infty}.
\end{equation}
\end{corollary}

\subsection{Strategy and preliminary results}
The strategy we shall develop in order to extend Proposition~\ref{f-conv} (and also Proposition~\ref{proposition:contrac 2}) to the rough case is based on a regularization of $x$. Specifically, for a strictly positive integer $n$, set 
\begin{equation}\label{eq:reg-eq1}
x^{n}_{s;t}
= c_{\gamma_1,\gamma_2}
\int_{|\xi|,|\eta|\leq n} Q_{s;t}(\xi,\eta) \, \hat W(d\xi,d\eta),
\end{equation}
where we recall that $x$ and $Q$ are respectively defined by \eqref{eq:harmonizable-fBs} and \eqref{eq:def-Q}. For fixed $n$, it is readily checked that $x^n$ is a regular Gaussian process. Its covariance function is given by $R^{n}_{s_1s_2;t_1t_2}=R^{1,n}_{s_1s_2}R^{2,n}_{t_1t_2}$, where 
\begin{equation}\label{eq:reg-cov1}
R^{i,n}_{ab}=c_{\gamma_i}\int_{|\xi|\leq n}\frac{(e^{\imath a\xi}-1)(e^{-ib\xi}-1)}{|\xi|^{2\gamma_i+1}}d\xi,
\quad\text{for}\quad
i=1,2,
\end{equation}
and hence $R^n$ is a regular function which satisfies Hypothesis ~\ref{hyp:Young}. One can thus apply Proposition~\ref{f-conv} and obtains the following Skorohod-Stratonovich comparison:
\begin{equation}\label{eq:correc-sko-strat-regularized}
z^{1,n,\di}_{s_{1}s_{2};t_{1}t_{2}}
\equiv \der^{x^{n},\diamond}(y_{\cdot}^{n,1}\1_{\Delta})
=\int_{\Delta}y^{n,1}_{s;t}d_{12}x^{n}_{s;t}
-\frac14\int_{\Delta}y^{n,2}_{s;t}d_{1}R^{1,n}_{s}d_{2}R^{2,n}_{t},
\end{equation}
for $\vp\in C^6(\mathbb R)$ satisfying condition (GC), and where we set again $\Delta=[s_{1},s_{2}]\times[t_{1}t_{2}]$. Our goal is now to take limits in equation~ \eqref{eq:correc-sko-strat-regularized}. 

\smallskip

A first observation in this direction is that equation \eqref{eq:correc-sko-strat-regularized} involves Skorohod integrals with respect to $x^{n}$. The fact that a different integral has to be defined for each $n$ is somehow clumsy, and this is why we have decided to express all integrals with respect to $\hat{W}$ in the remainder of our computations. Namely, the same computations as for equations ~\eqref{eq:comp1} and~\eqref{eq:comp2} entail that $\der^{x^n,\diamond}(y^n)=\der^{\hat W,\diamond}(K^ny^n)$, where $K^n$ is the operator defined by  
\begin{equation}\label{eq:def-Kn}
K^n\phi(\xi,\eta)
= \1_{(|\xi|,|\eta|\leq n)} \int_{[0,1]^{2}}\phi_{s;t} \, \partial_{s}\partial_{t}Q_{s;t}(\xi,\eta)dsdt.
\end{equation}
With this representation in hand, our limiting procedure can be decomposed as follows:

\smallskip

\noindent
$\bullet$ Take $L^{2}$ limits in the right hand side of equation \eqref{eq:correc-sko-strat-regularized} by means of rough paths techniques.

\smallskip

\noindent
$\bullet$ Show that $K^{n}y^n$ converges in $L^2(\Omega,L^2(\mathbb R))$ to $Ky$.

\smallskip

\noindent
Thanks to the closability of $\der^{\hat W,\diamond}$, this will show the convergence of $\der^{x^{n},\diamond}(y^n\1_{[0,1]^{2}})$ to $\der^{x,\diamond}(y\1_{[0,1]^{2}})$ and our Skorohod-Stratonovich correction formula will be obtained in this way.

\smallskip

We now state and prove 3 useful lemmas for our future computations. The first one deals with convergence of covariance functions:
\begin{lemma}\label{lemma:con-cov}
For $i=1,2$, set $R^{i}_u=u^{2\gamma_i}$. Then for all $\ep>0$ we have
$$
\lim_{n\to+\infty}\|R^{i,n}-R^{i}\|_{2\gamma_i-\epsilon}=0,
$$
where $R^{i,n}$ is defined by \eqref{eq:reg-cov1}.
\end{lemma}
\begin{proof}
We recall that $c_{\gamma_i}\int_\mathbb R\frac{|e^{\imath a\xi}-1|^2}{|\xi|^{2\gamma_i+1}}d\xi=a^{2\gamma_i}$. Then an elementary computation shows that
$$
|\der_i(R^{i,n}-R^{i})_{ab}|=\left|c_{\gamma_i}\int_{|\xi|\geq n}\frac{cos(a\xi)-cos(b\xi)}{|\xi|^{2\gamma_i+1}}d\xi\right|\lesssim_{\gamma_i,\epsilon}|a-b|^{2\gamma_i-\epsilon}\int_{|x|\geq n}|\xi|^{-1-\epsilon}d\xi,
$$
which gives
$
\|R^{n,i}-R^{u}\|_{2\gamma_i-\epsilon}\lesssim_{\gamma_i,\epsilon}\int_{|\xi|\geq n}|\xi|^{-1-\epsilon}d\xi,
$ 
and this finishes the proof.

\end{proof}

Our second preliminary result ensures that $x^{n}$ is an accurate approximation of $x$:

\begin{proposition}\label{proposition:conve-reg}
Let $p>1$ and $0<\ep<\min(\ga_1,\ga_2)$. Then we have the following convergence:
$$
\lim_{n\to\infty}\E\left[\|x^n-x\|^p_{\gamma_1-\epsilon,\gamma_2-\epsilon}\right] = 0.
$$
In addition, there exists  $\lambda>0$  such that 
\begin{equation}\label{eq:exp-moments-xn}
\sup_{n\in\mathbb N}\E\left[e^{\lambda\sup_{(s,t)\in[0,1]^{2}}|x^n_{s;t}|^2}\right]<+\infty.
\end{equation}
\end{proposition}

\begin{proof}
The definitions \eqref{eq:harmonizable-fBs} of $x$, \eqref{eq:reg-eq1} of $x^{n}$ plus formula \eqref{eq:def-Q} for $Q$ allow to write, for all $n\ge 1$:
\begin{equation}
\begin{split}
\E[|\der(x^{n}-x)_{s_1s_2; t_1t_2}|^2]&
=\int_{|x|,|y|\geq n}\frac{|e^{\imath s_2x}-e^{\imath s_1x}|^2|e^{\imath t_2y}-e^{\imath t_1y}|^2}{|x|^{2\gamma_1+1}|y|^{\gamma_2+1}}dxdy
\\&\lesssim_{\gamma_1,\gamma_2}|s_1-s_1|^{2(\gamma_1-\epsilon/2)}|t_2-t_1|^{2(\gamma_2-\epsilon/2)}
I_{n}
\end{split}
\end{equation}
where we have set $I_{n}=\int_{|x|,|y|\geq n} \frac{dx\,dy}{|x|^{1+\epsilon}|y|^{1+\epsilon}}$ (this quantity is obviously finite). Hence by Gaussian hypercontractivity and Lemma~\ref{lemma:G-R} we obtain 
\begin{equation*}
\begin{split}
\E[[\|x^n-x\|^p_{\gamma_1-\epsilon,\gamma_2-\epsilon}]
&\lesssim_{p,\gamma_1,\gamma_2,\epsilon}
\int_{[0,1]^{4}}\frac{\E[|\der x_{s_1s_2; t_1t_2}|^2]^{p/2}}{|s_2-s_1|^{(\gamma_1-\epsilon)p+2}|t_2-t_1|^{(\gamma_2-\epsilon)p+2}}dt_2ds_2dt_1ds_1
\\&\lesssim (I_n)^{p/2}\int_{[0,1]^{4}} |s_2-s_1|^{p\epsilon/2-2}|t_2-t_1|^{p\epsilon/2-1}
dt_2ds_2dt_1ds_1\lesssim (I_n)^{p/2}.
\end{split}
\end{equation*}
Since $\lim_{n\to\infty}I_{n}=0$, we thus get $\lim_{n\to\infty}\E[[\|x^n-x\|^p_{\gamma_1-\epsilon,\gamma_2-\epsilon}]=0$, which is our first claim. 

\smallskip

We now focus on the exponential integrability of $\sup x^{n}$. Notice that for a fixed $n$ one can easily get those exponential estimates thanks to Fernique's lemma. However, we claim some uniformity in $n$ here, and we thus come back to uniform estimates of moments in order to prove \eqref{eq:exp-moments-xn}. Let then $r=\max(\lfloor\frac{1}{\gamma_1-\epsilon}\rfloor,\lfloor\frac{1}{\gamma_2-\epsilon}\rfloor)+1$ and remark that $\|x^n\|_{\infty}\leq \|x^n\|_{\epsilon,\epsilon}$. We thus use a decomposition of the form
$\E[e^{\lambda\sup_{(s,t)\in[0,1]^{2}}|x^n_{s;t}|^2}]=I^{1,n}(\lambda)+I^{2,n}(\lambda)$, where
\begin{equation*}
I^{1,n}(\lambda)=\sum_{l=0}^{r-1}\frac{\lambda^l}{l!}\E[\|x^n\|_{\infty}^{2l}], 
\quad\text{and}\quad
I^{2,n}(\lambda)=\sum_{l=r}^{+\infty}\frac{\lambda^l}{l!}\E[\|x^n\|^{2l}_{\infty}].
\end{equation*}
We now bound those 2 terms separately: one the one hand, it is readily checked that
$$
I^{1,n}(\lambda)
\leq \max_{i=0,\ldots,r}
\sup_{n\in\mathbb N}\E[\|x^n\|_{\epsilon,\epsilon}^{2i}]<+\infty,
$$
for $\epsilon<\min(\gamma_1,\gamma_2)$. On the other hand,  the bound on $I^{2,n}(\lambda)$ is obtained invoking Lemma~\ref{lemma:G-R} again. Indeed, starting from  expression \eqref{eq:GRR-ineq} and introducing a standard Gaussian random variable $\cn$, it is easily seen that
$$
\E[\|x^n\|_{\epsilon,\epsilon}^{2l}]\leq C_{2l,\epsilon,\epsilon}\E[\mathcal N^{2l}]\int_{[0,1]^2\times[0,1]^2}\frac{\E[|\der x^n_{s_1s_2t_1t_1}|^2]^l}{|s_2-s_1|^{2l\epsilon+2}|t_2-t_1|^{2l\epsilon+2}}ds_1ds_2dt_1dt_2
$$ 
with $\mathcal N$ is a Gaussian random variable $\mathcal N(0,1)$. Now we have  
\begin{multline*}
\sup_{n\in\mathbb N}\E\Big[|\der x^{n}_{s_1s_2; t_1t_2}|^2\Big]
\leq\int_{\mathbb R^2}\frac{|e^{\imath s_2x}-e^{\imath s_1x}|^2|e^{\imath t_2y}-e^{\imath t_1y}|^2}{|x|^{2\gamma_1+1}|y|^{\gamma_2+1}}dxdy
\\
\leq|s_2-s_1|^{2\gamma_1}|t_2-t_1|^{2\gamma_2}\int_{\mathbb R^2}\frac{|e^{\imath x}-1|^2|e^{\imath y}-1|^2}{|x|^{2\gamma_1+1}|y|^{\gamma_2+1}}dxdy
\lesssim |s_2-s_1|^{2\gamma_1}|t_2-t_1|^{2\gamma_2}.
\end{multline*}
Furthermore, it can be shown that $C_{2l,\epsilon,\epsilon}$ is of the form $M^{l}$ for a given $M>1$. Thus
$$
\sup_{n\in\mathbb N}\E[\|x^n\|_{\epsilon,\epsilon}^{2l}]\lesssim
 M^{l} \, \E\left[\mathcal N^{2l}\right],
$$
from which the relation $I^{2,n}(\lambda)<\infty$ is easily obtained. This finishes the proof of \eqref{eq:exp-moments-xn}.

\end{proof}

With classical considerations concerning compositions of H\"older functions with non linearities, we finally get the following result which is labelled for further use. Its proof is omitted for sake of conciseness.
\begin{lemma}\label{lemma:h-p-r}
Let $\rho_{1},\rho_{2}\in(0,1)$,  $x^1,x^2$ two increments lying in $\cp_{1,1}^{\rho_1,\rho_2}$, and $f\in C^{3}(\mathbb R)$ satisfying  condition (GC).  Then we have: 
\begin{equation}\label{eq:bnd-f(x1)-f(x2)}
\cn_{\rho_1,\rho_2}(f(x^1)-f(x^2))\lesssim 
c_{x^{1},x^{2}} \, \cn_{\rho_1,\rho_2}(x^1-x^2) \, 
\lc 1+\cn_{\rho_1,\rho_2}(x^1)+\cn_{\rho_1,\rho_2}(x^2) \rc^2
\end{equation}
where we recall that $\cn_{\al_{1},\al_{2}}(f)$ has been defined at equation \eqref{eq:embed-H-holder}. In the relation above we have also set
 $c_{x^{1},x^{2}}=\exp({\theta(\sup_{(s,t)\in[0,1]^{2}}|x^1_{s;t}|^2+\sup_{(s,t)\in[0,1]^{2}}|x^2_{s;t}|^2)})$, where $\theta$ is the constant featuring in condition (GC).
\end{lemma} 

\subsection{It\^o-Skorohod type formula}
We now turn to the limiting procedure in equation~\eqref{eq:correc-sko-strat-regularized}, beginning with the term involving covariances only:

\begin{proposition}\label{con-trace-term1}
Let $f\in C^{6}(\mathbb R)$ be a function satisfying condition (GC) with a small parameter $\lambda>0$, and $x^{n}$ be the regularized version of $x$ defined by \eqref{eq:reg-eq1}. Then the following convergence:
\begin{equation}\label{eq:lim-xn-x-covariance}
\lim_{n\to+\infty}\int_{[0,1]^{2}}y^n_{s;t} \, d_{1}R^{1,n}_s d_{2}R^{2,n}_t=\gamma_1\gamma_2\int_{[0,1]^{2}}y_{s;t}s^{2\gamma_1-1}t^{2\gamma_2-1}dsdt
\end{equation}
holds in $L^2(\Omega)$.
\end{proposition}

\begin{proof}
The integrals involved in \eqref{eq:lim-xn-x-covariance} are all of Young type. Owing to Proposition \ref{mixed-Young-2D}, we thus have:
$$
\int_{[0,1]^{2}}y^n_{s;t}d_{1}R^{1,n}_sd_{2}R^{2,n}_t=
\lc (\id-\Lambda_1\der_1)(\id-\Lambda_2\der_2) \rc (y^n\der_1R^{1,n}\der_2R^{2,n}).
$$
By continuity of the sewing map, the desired convergence will thus stem from the relations $\lim_{n\to 0} A^{1,n}=0$ and $\lim_{n\to 0} A^{2,n}=0$, where for $\epsilon>0$ we set:
\begin{equation*}
A^{1,n}:=\sum_{i=1}^{2} \|\der_{i}R^{i,n}-\der_{i} R^{i}\|_{2\gamma_i-\epsilon},
\quad\text{and}\quad
A^{2,n}:=\cn_{\gamma_1-\epsilon,\gamma_2-\epsilon}(y^n-y).
\end{equation*}

Now the relation $\lim_{n\to 0} A^{1,n}=0$ is obviously a direct consequence of Lemma ~\ref{lemma:con-cov}. As far as $A^{2,n}$ is concerned, we start from relation \eqref{eq:bnd-f(x1)-f(x2)} and apply H\"older's inequality. This yields
\begin{multline*}
\E[(\cn_{\gamma_1-\epsilon,\gamma_2-\epsilon}(y^n-y))^2] \\
\lesssim
\E^{1/4}[(c_{x^{1},x^{2}})^8] \,
\E^{1/2}[\|x^n-x\|^4_{\gamma_1-\epsilon,\gamma_2-\epsilon}] \,
\E^{1/4}[(1+N_{\rho_1,\rho_2}(x^n)+N_{\rho_1,\rho_2}(x))^{16}].
 \end{multline*} 
 Then according to Proposition~\ref{proposition:conve-reg} we see that the r.h.s of this last equation vanishes when $n$ goes to infinity, which proves our claim.
 
\end{proof}

We now compute the correction terms in $z^{1}$, that is the equivalent of Proposition~\ref{f-conv}.
\begin{proposition}\label{Th:contrac-rough1}
Let $x$ be a fBs with Hurst parameters $\ga_{1},\ga_{2}>1/3$.
Consider a function $\vp\in C^{8}(\mathbb R)$ satisfying condition (GC) and a rectangle $\Delta=[s_{1},s_{2}]\times[t_{1},t_{2}]$. Then we have that 
$y^{1}_{\cdot}\1_{\Delta}\in Dom(\der^{\diamond})$, and if we define the increment 
$z^{1,\di}\equiv\der^{\diamond}(y^{1}_{\cdot}\1_{\Delta})$ the following relation holds true: 
\begin{equation}\label{eq:correc-sko-strat-f-d12x}
z^{1,\di}_{s_{1}s_{2};t_{1}t_{2}}
=z^{1}_{s_{1}s_{2};t_{1}t_{2}}
-\gamma_1\gamma_2\int_{\Delta} y^{2}_{s;t}s^{2\gamma_1-1}t^{2\gamma_2-1}dsdt,
\end{equation}
where $z^{1}$ is the rough integral given by Theorem \ref{thm:rough-strato-intro}.
\end{proposition}

\begin{proof}
Let us start from the corrections for the regularized process $x^{n}$, for which we can appeal to Proposition~\ref{f-conv}. We obtain relation \eqref{eq:correc-sko-strat-regularized}, written here again for convenience:
\begin{equation}\label{eq:correc-sko-strat-regularized-2}
z^{1,n,\di}_{s_{1}s_{2};t_{1}t_{2}}
\equiv \der^{x^{n},\diamond}(y_{\cdot}^{n,1}\1_{\Delta})
=\int_{\Delta}y^{n,1}_{s;t}d_{12}x^{n}_{s;t}
-\frac14\int_{\Delta}y^{n,2}_{s;t}d_{1}R^{1,n}_{s}d_{2}R^{2,n}_{t}.
\end{equation}
Now putting together Proposition~\ref{con-trace-term1}  and the continuity of the rough path integral, we get convergence of the r.hs of  \eqref{eq:correc-sko-strat-regularized-2} in $L^2(\Omega)$. Thus one can write, in the a.s and $L^{2}(\Omega)$ sense:
$$
\lim_{n\to+\infty}\der^{x^{n},\diamond}(y^n)
=\int_{\Delta} y^{1}_{s;t}d_{12}x_{s;t}
-\gamma_1\gamma_2\int_{\Delta}y^{2}_{s;t}s^{\gamma_1-1}t^{\gamma_2-1} \, dtds,
$$
where the integral with respect to $x$ is interpreted in the sense of Theorem \ref{thm:rough-strato-intro}.

\smallskip

Let us further analyze the convergence of $\der^{x^{n},\diamond}(y^n)$: recall that this quantity can be written as $\der^{\hat W,\diamond}(K^ny^n)$, where $K^n$ is defined by \eqref{eq:def-Kn} or specifically as
\begin{equation}\label{eq:Kn-for-fBm}
(K^{n}y^n)(\xi,\eta)=
\frac{\imath \xi \, \imath \eta}{|\xi|^{\gamma_1+1/2}|\eta|^{\gamma_2+1/2}}\left(\int_{\Delta}y^n_{uv}e^{\imath \xi u+\imath \eta v}dudv\right)\1_{(|\xi|,|\eta|\leq n)}
\end{equation}
Hence, owing to closability of the operator $\der^{\hat W,\diamond}$, the proof of \eqref{eq:correc-sko-strat-f-d12x} is reduced to show that $K^{n}y^{n,1}$ converges in $L^{2}(\Omega; L^{2}(\Delta))$ to $Ky^{1}$. Now expression \eqref{eq:Kn-for-fBm} easily entails that
\begin{equation*}
\begin{split}
\left|\left|K^{n}y^{n,1}-Ky\right|\right|_{L^2(\mathbb R)}
&\leq \|y^{n,1}-y^{1}\|_{\mathcal H}
\\&+\int_{|\xi|,|\eta|\geq n} |\xi|^{1-2\gamma_1}|\eta|^{1-2\gamma_2}\left|\int_{\Delta}y^{1}_{u;v}e^{\imath \xi u+\imath \eta v}dudv\right|^2d\xi d\eta,
\end{split}
\end{equation*}
and we shall bound the 2 terms on the r.h.s of this inequality.

\smallskip

Indeed, on the one hand we  consider $\gamma_1,\gamma_2>1/4$ and  $\epsilon>0$ small enough. This gives 
\begin{multline*}
\E\left[ \int_{\|(\xi,\eta)\|_\infty\geq n}|\xi|^{1-2\gamma_1}|\eta|^{1-2\gamma_2}\left|\int_{\Delta}y_{u;v}e^{\imath \xi u+\imath \eta v}dudv\right|^2d\xi d\eta\right] \\
\lesssim n^{-\epsilon} \, 
\E\left[(\cn_{\gamma_1-\epsilon,\gamma_2-\epsilon}(y))^2\right]
\stackrel{^{n\to+\infty}}{\longrightarrow}0.
\end{multline*}
On the other hand,  Corollary ~\ref{prop:embbeding-cam}  asserts that
$\|y-y^n\|_{\mathcal H}\lesssim 
\cn_{\gamma_1-\epsilon,\gamma_2-\epsilon}(y^n-y)$, and the r.h.s of this relation vanishes as $n\to\infty$ thanks to Proposition~\ref{proposition:conve-reg}. This concludes our proof.  

 \end{proof}

In order to complete our comparison between It\^o and Stratonovich formulae, we still have to compare the Skorohod type increment $z^{2,\di}$ and the rough integral $z^{2}$. As a previous step, let us give an intermediate result concerning some mixed integrals in $R,x$: 
\begin{proposition}\label{prop:conv-trace-terms-3}
Let $\vp\in C^6(\mathbb R)$ and recall that for the fractional Brownian sheet $x$ we have $R^i_u=u^{2\gamma_i}$ for $i=1,2$. Then the integral
\begin{equation}\label{eq:trace-term1}
\int y^{1}R^1 \, d_{1}x \, d_{2}R^2  \\
=\lc(\id-\Lambda_1\der_1)(\id-\Lambda_2\der_2)\rc\!
\lp y^{1}R^1\der_1x\der_2R^2+1/2y^{2}R^1(\der_1x)^2\der_2R^2 \rp
\end{equation}
is well defined a.s, in the sense of Proposition ~\ref{proposition:integ-2d}. Moreover the following convergence takes place in $L^2(\Omega)$:
\begin{equation}\label{eq:cvgce-regul-mixed-R-x}
\lim_{n\to+\infty}\int_{\Delta}y^{1,n}_{uv}R^{1,n}_ud_{2}x^n_{uv}d_{2}R^{2,n}_v
=\int_{\Delta}y^{1}_{u;v}R^1_ud_{2}x_{u;v}d_{2}R^2_v.
\end{equation}
Finally, the same kind of result is still verified when one interchanges directions 1 and 2 in relation \eqref{eq:trace-term1}.
\end{proposition}

\begin{proof}
Let us first check that the integral in \eqref{eq:trace-term1} is well-defined in the sense of Proposition~\ref{mixed-Young-2D}. To this aim, set $A=y^{1}R^1 \der_1x \,\der_2R^2+1/2y^{2}R^1(\der_1x)^2\der_2R^2$. Then a simple application of Proposition \ref{prop:difrul} yields $\der_1A= z\, \der_2R^2$, with
\begin{equation*}
z
=\left(y^{2}\der_1x-\der_1y^{1}\right) R^1\der_1x-1/2\der_1(y^{2}R^1)(\der_1x)^2-y^{1}\der_1R^1\der_1x.
\end{equation*}
It is thus easily seen that $\der_1A\in \cp_{3,2}^{3\gamma_1-\epsilon,2\gamma_2}$ for an arbitrary small $\epsilon$, thanks to the fact that  $x\in\cp_{1,1}^{\gamma_1-\epsilon,\gamma_2-\epsilon}$ and $\der_1y^{1}-y^{2}\der_1x\in\cp_{2,1}^{2\gamma_1-\epsilon,\gamma_2}$ almost surely. Notice that with the same kind of considerations we also  have that $\der_2A\in\cp_{2,3}^{\gamma_1,3\gamma_2-\epsilon}$.

\smallskip

Let us now compute $\der A$: we have 
\begin{equation*}
\der A=-\der_2z \, \der_2R^2 = -\lp A_1+A_2 \rp \der_2R^2 ,
\end{equation*}
where
\begin{equation*}
A_{1} = g \, R^{1} \der_{1} x,
\quad\text{with}\quad
g= \der y^{1}-\der_2 y^{2}\der_1x-y^{2}\der x,
\end{equation*}
and where setting $(\der_1x\circ_1\der x)_{s_1s_2; t_1t_2}=\der_1x_{s_1s_2t_1}\der x_{s_1s_2; t_1t_2}$ similarly to Definition \ref{def:g-circ-h}, we have
\begin{multline*}
A_{2}=
\der_2y^{1}\der_1R^1\der_1x
-\{\der_1y^{1}-y^{2}\der_1x\}R^1\der x-1/2\der y^{2}(\der_1x)^2
\\
-1/2\der_1y^{2}\{\der x\circ_1\der_1x+\der_1x\circ_1\der x\}.
\end{multline*}
The reader can now easily check that $A_{2}\in\cp_{2,2}^{3\gamma_1-\epsilon,\gamma_2-\epsilon}$. In order to check the regularity of $A_1$, observe that $g$ is of the form $g=\der_{2}h$, with
\begin{eqnarray}\label{eq:repres-h-taylor}
h_{s_1s_2;t}&:=&
(\der_1y^{1}_{s_1s_2;t}-y^{2}_{s_1;t})\der_1x_{s_2s_2;t}  \notag\\
&=&\left(\int_0^1d\theta \, \theta\int_0^1d\theta'y^2(x_{s_1;t}+\theta\theta'\der_1x_{s_1s_2;t})\right)(\der_1x_{s_1s_2;t})^2.
\end{eqnarray}
Computing $\der_2 h$ with formula \eqref{eq:repres-h-taylor}, one obtains that  $A^1=\der_2hR^1\der_1x\in\cp_{2,2}^{3\gamma_1-\epsilon,\gamma_2-\epsilon}$.

\smallskip

Let us summarize our last considerations: we have seen that both $A_{1}$ and $A_{2}$ lye into $\cp_{2,2}^{3\gamma_1-\epsilon,\gamma_2-\epsilon}$, and recalling that  $\der A= -( A_1+A_2 ) \der_2R^2$, we obtain  $\der A\in \cp_{3,3}^{3\gamma_1-\epsilon,3\gamma_2-\epsilon}$. We have also checked that $\der_1A\in \cp_{3,2}^{3\gamma_1-\epsilon,2\gamma_2}$ and $\der_2A\in\cp_{2,3}^{2\gamma_1,3\gamma_2-\epsilon}$. Gathering all this information, we have checked the assumptions of Proposition~\ref{proposition:integ-2d} for the increment $A$, which justifies expression \eqref{eq:trace-term1}.

\smallskip

Now we focus on the convergence formula \eqref{eq:cvgce-regul-mixed-R-x}. We start by observing that for all $n\ge 1$ the following representation holds true: 
$$
\int_{\Delta}y^{n,1}_{uv}R^{1,n}_ud_{2}x^n_{uv}d_{2}R^{2,n}_v=
\lc (\id-\Lambda_1\der_1)(\id-\Lambda_2\der_2)\rc\!A^n
$$
with $A^n=y^{n,1}R^{1,n}\der_2R^{2,n}+1/2y^{n,1}R^{1,n}(\der_1x^n)^2\der_2R^{2,n}$. Hence, owing to the continuity of the planar  sewing-maps $(\Lambda_i)_{i=1,2}$ and $\Lambda$, our claim  \eqref{eq:cvgce-regul-mixed-R-x} is reduced to prove that  the sequences $\|A^n-A\|_{\gamma_1-\epsilon,2\gamma_2-\epsilon},\|\der_1(A^n-A)\|_{3\gamma_1-\epsilon,2\gamma_2-\epsilon},\|\der_2(A^n-A)\|_{\gamma_1-\epsilon,3\gamma_2-\epsilon}$ and $\|\der(A^n-A)\|_{3\gamma_1-\epsilon,3\gamma_2-\epsilon}$ converge in $L^2(\Omega)$ and almost surely to $0$. Furthermore, it is readily checked that those convergences all stem from the relations
\begin{equation}\label{eq:conv-1-Rn-R}
\lim_{n\to 0}
\|R^{i,n}-R"\|_{2\gamma_i-\epsilon}
+\sum_{i=0}^{5}\cn_{\gamma_1-\epsilon,\gamma_2-\epsilon}(y^{n,i}-y^i)
+\|(\der_1x^n)^2-(\der_1x)^2\|_{2\gamma_1-\epsilon,1}
=0
\end{equation}
 and 
\begin{equation}\label{eq:conv-2-taylor-xn-x}
\lim_{n\to 0}
\|\der_2\lp h-h^{n} \rp\|_{2\gamma_1-\epsilon,\gamma_2-\epsilon}=0,
\quad\text{with}\quad
h^{n}=\der_1y^{n,1}-y^{n,2}\der_1x^{n},
\end{equation}
where the limits take place in some $L^{p}(\Omega)$ with a sufficiently large $p$, and where we recall that $h$ is defined by \eqref{eq:repres-h-taylor}.
We now turn to the proof of those two relations.

\smallskip

To begin with, note that the convergence~\eqref{eq:conv-1-Rn-R} easily stems from Lemma~\ref{lemma:con-cov} for the terms $R$, Proposition \ref{proposition:conve-reg} for the terms $x$ and Lemma \ref{lemma:h-p-r} for the terms $y$. In order to prove \eqref{eq:conv-2-taylor-xn-x}, we invoke again the integral representation \eqref{eq:repres-h-taylor} for both $h^{n}$ and $h$. Then some elementary considerations (omitted here for sake of conciseness) allow to reduce the problem to the following relation:
\begin{equation*}
L^p(\Omega)-\lim_{n\to\infty}\lp
\E[\cn_{\gamma_1-\epsilon,\gamma_1-\epsilon}^{p}(x^n-x)]
+
\sum_{i=0}^{3}\E[\cn_{\gamma_1-\epsilon,\gamma_2-\epsilon}^{p}(y^{n,i}-y^i)]
\rp
=0.
\end{equation*}
This last relation is a direct consequence of Proposition \ref{proposition:conve-reg} and composition with non linearities, whenever $f$ satisfies the growth condition (GC) with a small parameter $\lambda>0$. The proof is now finished.
  
\end{proof}

We can now state our result concerning the It\^o-Stratonovich correction for the mixed stochastic integral $\int y d_{1}xd_{2}x$:
\begin{theorem}\label{theorem:conv-1}
Let $x$ be a fBs with Hurst parameters $\ga_{1},\ga_{2}>1/3$.
Consider a function $\vp\in C^{8}(\mathbb R)$ satisfying condition (GC) and a rectangle $\Delta=[s_{1},s_{2}]\times[t_{1},t_{2}]$. Then we have that $N(y^{2})\in Dom(\der^{\diamond,2})$, and if we define the Skorohod integral $z^{2,\di}$ as $\der^{\diamond,2}(N(y^{2}))$,  the following particular case of relation~\eqref{eq:strato-skoro-mixed-intg} holds: 
\begin{multline}\label{eq:correc-ito-strato-mixed-fBs}
z^{2,\di}_{s_{1}s_{2};t_{1}t_{2}}
=z^{2}_{s_{1}s_{2};t_{1}t_{2}}
-\gamma_1\gamma_2\int_{\Delta}y^{2}_{u;v}u^{2\gamma_1-1}v^{2\gamma_2-1}dudv 
-\gamma_2\int_{\Delta}y^{3}_{u;v}u^{2\gamma_1}v^{2\gamma_2-1}dvd_{1}x_{u;v} \\
-\gamma_1\int_{\Delta}y^{3}_{u;v}u^{2\gamma_1-1}v^{2\gamma_2}dud_{2}x_{u;v} 
+\gamma_1\gamma_2\int_{\Delta}y^{4}_{u;v}u^{4\gamma_1-1}v^{4\gamma_2-1}dudv.
\end{multline}
\end{theorem}

\begin{proof}
We follow the same strategy as for Theorem~\ref{Th:contrac-rough1}: apply first Proposition~\ref{proposition:contrac 2} for the regularized process $x^n$, which yields: 
 \begin{align}\label{eq:reg-contrac-2}
&\int_{\Delta}y^{n,2}_{u;v}d^\diamond_{1}x^n_{uv}d^\diamond_{2}x^n_{uv}
=\int_{\Delta}y^{n,2}_{u;v}d_{1}x^n_{u;v}d_{2}x^n_{uv}
-1/4\int_{\Delta}y^{n,2}_{u;v}d_{1}R^{1,n}_ud_{2}R^{2,n}_v\notag\\
&\hspace{2.5cm}-1/2\int_{\Delta}y^{n,3}_{u;v}R^{1,n}_ud_{2}R^{2,n}_vd_{1}x_{u;v}
-1/2\int_{\Delta}y^{n,3}_{u;v}R^{2,n}_vd_{1}R^{1,n}_ud_{2}x^n_{uv} \notag\\
&\hspace{2.5cm}+
1/4\int_{\Delta}y^{n,4}_{u;v}R^{1,n}_uR^{2,n}_vd_{1}R^{1,n}_ud_{2}R^{n,2}_v.
\end{align} 
Now our preliminary results allow to take limits in relation \eqref{eq:reg-contrac-2}. Indeed, owing to Propositions~\ref{con-trace-term1} and~\ref{prop:conv-trace-terms-3} plus the continuity of the rough increment $z^{2}$ given at Theorem \ref{thm:rough-strato-intro}, we obtain the convergence in $L^2(\Omega)$ for the four first terms in the r.h.s of  equation~\eqref{eq:reg-contrac-2}. Moreover the last term also converges in $L^2(\Omega)$, thanks to the same arguments as in the proof of the Proposition~\ref{con-trace-term1}. We thus get the convergence of the r.h.s of equation~\eqref{eq:reg-contrac-2} to the r.h.s  of equation~\eqref{eq:correc-ito-strato-mixed-fBs}, and also the fact that $z^{2,\di}$ converges in $L^2(\Omega)$. Like in the proof of Theorem \ref{Th:contrac-rough1}, the proof of~\eqref{eq:correc-ito-strato-mixed-fBs} is thus reduced to show the $L^{2}$ convergence of the integrand defining $z^{2,\di}$.

\smallskip

However, mimicking again the proof of Theorem \ref{Th:contrac-rough1}, it is easily seen that
$$
\int_{0}^{1}\int_{0}^{1}y^{n,2}_{u;v}d^\diamond_{1}x^n_{uv}d^\diamond_{2}x^n_{uv}
:=\der^{x^n,\diamond,2}\lp N(y^{n,2}) \rp
=\der^{\hat W,\diamond,2}\lp K^{n,\otimes2}N(y^{n,2}) \rp,
$$  
where we recall that  $K^{\otimes2}$ is defined by \eqref{eq:def-K2} and
$$
\lc K^{n,\otimes2}\phi\rc \!(x_{1}x_{2};y_{1}y_{2})=
\1_{\{|x_{1}|, |x_{2}|, |y_{1}|, |y_{2}|\leq n\}} \, \lc K^{\otimes2}\phi\rc \!(x_{1}x_{2};y_{1}y_{2}).
$$
It thus remains to show that $K^{n,\otimes2}(N(y^{n,2}))$ converges to $K^{\otimes2}(N(y^{2})$ in $L^2(\Omega,L^2(\mathbb R^4))$. Towards this aim, we introduce the further notation $u_{s;t}(\xi,\eta)=y^{2}_{s;t}(e^{\imath s\xi}-1)(e^{\imath t\eta}-1)$,  $u^n_{s;t}(\xi,\eta)=y^{n,2}_{s;t}(e^{\imath s\xi}-1)(e^{\imath t\eta}-1)$ and 
\begin{eqnarray*}
\hat{u}(\xi_1,\xi_2;\eta_1,\eta_2)&=&\int_{\Delta}u_{s;t}(\xi_1,\eta_1)e^{\imath s\xi_2+\imath t\eta_2}dsdt
\\
\hat{u}^{n}(\xi_1,\xi_2;\eta_1,\eta_2)&=&\int_{\Delta}u^{n}_{s;t}(\xi_1,\eta_1)e^{\imath s\xi_2+\imath t\eta_2}dsdt.
\end{eqnarray*}
Then note that 
$$
\|(K^n)^{\otimes2}(N(y^n))-K^{\otimes2}(N(y)\|^2_{L^2(\mathbb R^4)}\leq I_1^n+I_2^n+I_3^n,
$$ 
where 
\begin{eqnarray*}
I_1^n&=&\int_{|\xi_1|,|\eta_1|\geq n}
\left(\int_{\mathbb R^2}|\xi_2|^{1-2\gamma_1}|\eta_2|^{1-2\gamma_2}
\left|\hat{u}(\xi_1,\xi_2;\eta_1,\eta_2)\right|^2d\xi_2d\eta_2\right)
\frac{d\xi_1d\eta_1}{|\xi_1|^{2\gamma_1+1}|\eta_1|^{2\gamma_2+1}} \\
I_2^n&=&\int_{\mathbb R^2}
\left(\int_{|\xi_2|,|\eta_2|\geq n}|\xi_2|^{1-2\gamma_1}|\eta_2|^{1-2\gamma_2}
\left|\hat{u}(\xi_1,\xi_2;\eta_1,\eta_2)\right|^2d\xi_2d\eta_2\right)
\frac{d\xi_1 d\eta_1}{|\xi_1|^{2\gamma_1+1}|\eta_1|^{2\gamma_2+1}}
\end{eqnarray*}
and
\begin{multline*}
I_3^n 
=\int_{\mathbb R^2}
\left(\int_{\mathbb R^2}|\xi_2|^{1-2\gamma_1}|\eta_2|^{1-2\gamma_2}
\left|\hat{u}(\xi_1,\xi_2;\eta_1,\eta_2)-\hat{u}^{n}(\xi_1,\xi_2;\eta_1,\eta_2)\right|^2d\xi_2d\eta_2\right) \\
\times\frac{d\xi_1d\eta_1}{|\xi_1|^{2\gamma_1+1}|\eta_1|^{2\gamma_2+1}}.
\end{multline*}
In order to bound those 3 terms, observe that
$$
\cn_{\gamma_1-\epsilon,\gamma_2-\epsilon}(u(\xi,\eta))\lesssim 
\cn_{\gamma_1-\epsilon,\gamma_2-\epsilon}(y)(1+|\xi|^{\gamma_1-\epsilon}+|\eta|^{\gamma_2-\epsilon}+|\xi|^{\gamma_1-\epsilon}|\eta|^{\gamma_2-\epsilon})
$$ 
and thus Corollary~\ref{prop:embbeding-cam} entails that 
$$
\E[I_1^n]\lesssim \E[\cn^{2}_{\gamma_1-\epsilon,\gamma_2-\epsilon}(y^{2})^2]
\int_{|\xi|,|\eta|\geq n}\frac{1}{|\xi|^{\epsilon+1}|\eta|^{\epsilon+1}}d\xi d\eta
\stackrel{n\to+\infty}{\longrightarrow}0,
$$
and 
$$
\E[I_2^n]\lesssim n^{-\epsilon}\E[\cn^{2}_{\gamma_1-\epsilon,\gamma_2-\epsilon}(y^{2})]
\stackrel{n\to+\infty}{\longrightarrow}0.
$$
As far as $I_3^n$ is concerned, we remark that 
$$
\cn_{\gamma_1-\epsilon,\gamma_2-\epsilon}(u(\xi,\eta)-u^n(\xi,\eta))
\lesssim 
\cn_{\gamma_1-\epsilon,\gamma_2-\epsilon}(y^{2}-y^{n,2})(1+|\xi|^{\gamma_1-\epsilon}+|\eta|^{\gamma_2-\epsilon}+|\xi|^{\gamma_1-\epsilon}|\eta|^{\gamma_2-\epsilon})
$$
and then we can conclude along the same lines as in Theorem~\ref{theorem:conv-1} that $\E[I_3^n]$ vanishes as $n$ goes to infinity. This finishes the proof.

\end{proof}

The last step in order to go from Theorem \ref{theorem:conv-1} to Theorem~\ref{thm:fBs-skoro-intro} is to convert  Stratonovich into Skorohod type integrals in the right hand side of relation~\eqref{eq:correc-ito-strato-mixed-fBs}. To this aim, let us first recall  the following one-parameter result:  
\begin{proposition}
Let $B$ a fractional brownian motion with hurst parameters $1/2\geq\gamma>1/3$ then we have that $u\mapsto\varphi(B_u)u^{2\gamma}\in Dom(\der^{\diamond,B})$ and we have that 
$$
\int_{[0,1]}\varphi(B_{u})u^{2\gamma}d^{\diamond}B_u=\int_{[0,1]}\varphi(B_u)u^{2\gamma}dB_u-\gamma\int_{[0,1]}\varphi'(B_u)u^{4\gamma-1}du
$$
\end{proposition}\label{proposition:ito-strato-1d}

\begin{proof}
Use exactly the same arguments of the Proposition~\eqref{Th:contrac-rough1} for the one parameters setting.
\end{proof}

Now the Corollary below is the key to the conversion of Theorem \ref{theorem:conv-1} into Theorem~\ref{thm:fBs-skoro-intro}:
\begin{corollary}
For $\gamma_i>1/3$ and $\varphi\in C^{6}(\mathbb R)$ then for every $v\in[0,1]$ $u\mapsto y^3_{u;v}u^{2\gamma_1}\in Dom(\der^{\diamond,x_{.;v}})$ and the following formula hold true 
$$
\int_{\Delta}y^3_{u;v}u^{2\gamma_1}v^{2\gamma_2-1}d^{\diamond}_1x_{u;v}dv=\int_{\delta}y^{3}_{u;v}u^{2\gamma_1}v^{2\gamma_2-1}d_1x_{u;v}d_2v-\gamma_1\int_{\Delta}y^4_{u;v}u^{4\gamma_1-1}v^{4\gamma_2-1}dudv
$$
\end{corollary}
\begin{proof}
we recall that $x_{u,v}=^{law}v^{\gamma_2}B_u$ with $B$ is a fBm with hurst parameter $\gamma_1$ and then it suffice to use the proposition \eqref{proposition:ito-strato-1d}
\end{proof}


\begin{thebibliography}{99}

\bibitem{AMN}
E. Al\`{o}s, O. Mazet, D. Nualart:
Stochastic calculus with respect to Gaussian processes.
{\it Ann. Prob.} {\bf 29} (2001), no. 2, 766--801.

\bibitem{BJ}
X. Bardina, M. Jolis:
Multiple fractional integral with Hurst parameter less than $\frac12$.
{\it Stoch. Proc. Appl.} {\bf 116} (2006), 463--479.

\bibitem{CW}
Cairoli, J. Walsh:
Stochastic integrals in the plane. 
\textit{Acta Math.} \textbf{134} (1975), 111--183.

\bibitem{Gp}
K.Chouk, M. Gubinelli: Rough sheets. Work in progress (2013).

\bibitem{DMN}
G. Da Prato, P. Malliavin, Paul, D. Nualart: 
Compact families of Wiener functionals. 
{\it C. R. Acad. Sci. Paris Sér. I Math.} {\bf 315} (1992), no. 12, 1287--1291.

\bibitem{FR}
P. Friz, S. Riedel:
Convergence rates for the full Gaussian rough paths.
\textit{Arxiv} preprint (2011).

\bibitem{FV-bk}
P. Friz, N. Victoir:
\emph{Multidimensional dimensional processes seen as rough paths.}
Cambridge University Press, 2010.

\bibitem{Gu}
M. Gubinelli:
Controlling rough paths.
{\it J. Funct. Anal.} {\bf 216}, 86-140 (2004).


\bibitem{GT} M. Gubinelli, S. Tindel:
Rough evolution equations.
{\it Ann. Probab.}  {\bf 38}  (2010),  no. 1, 1--75.

\bibitem{HJT}
Y. Hu, M. Jolis, S. Tindel:
On Stratonovich and Skorohod stochastic calculus for Gaussian processes.
{\it Ann. Probab.}  {\bf 41}  (2013),  no. 3, 1656--1693.

\bibitem{HY}
Y. Hu, J-A. Yan:
Wick calculus for nonlinear Gaussian functionals.
Acta Math. Appl. Sin. Engl. Ser. 25 (2009), no. 3, 399-414.

\bibitem{LN}
J. Le\'on, D. Nualart:
An extension of the divergence operator for Gaussian processes. 
\textit{Stochastic Process. Appl.} \textbf{115} (2005), no. 3, 481--492.

\bibitem{Nu84}
D. Nualart:
Une formule d'It\^o pour les martingales continues \`a deux indices et quelques applications. \textit{Ann. Inst. H. Poincaré Probab. Statist.} \textbf{20} (1984), no. 3, 251--275.

\bibitem{Nu06}
D. Nualart: \emph{The Malliavin Calculus and Related
Topics.} Probability and its Applications. Springer-Verlag, 2nd
Edition, (2006).

\bibitem{QT} L. Quer-Sardanyons, S. Tindel:
The 1-d stochastic wave equation driven by a fractional Brownian sheet. 
\textit{Stochastic Process. Appl.} \textbf{117} (2007), no. 10, 1448--1472.

\bibitem{ST94}
Samorodnitsky, G., Taqqu, M. S.: \emph{Stable non-Gaussian random
processes.} Chapman and Hall, (1994).

\bibitem{TVp1}
C.A. Tudor, F. Viens:
It\^o formula and local time for the fractional Brownian sheet. 
{\it Electron. J. Probab.} {\bf 8} (2003), no. 14, 31 pp.

\bibitem{TVp2}
C.A. Tudor, F. Viens:
It\^o formula for the two-parameter fractional Brownian motion using the extended divergence operator. 
{\it Stochastics} {\bf 78} (2006), no. 6, 443--462.

\bibitem{WZ}
E. Wong, M. Zakai:
Differentiation formulas for stochastic integrals in the plane. 
\textit{Stochastic Processes Appl.} \textbf{6} (1977/78), no. 3, 339--349.




\end{thebibliography}
\end{document}